\documentclass{amsart}

\usepackage[margin=1in]{geometry}

\usepackage{graphicx}
\usepackage{amsmath,amssymb,amsfonts}
\usepackage{mathrsfs}
\usepackage{amsthm}
\usepackage[colorlinks,allcolors=blue]{hyperref}
\usepackage{enumerate} 
\usepackage{bbm} 
\usepackage{bm} 
\usepackage{tikz} 
\usetikzlibrary{cd} 
\usepackage{makecell} 
\usepackage{floatrow} 

\usepackage{thmtools, thm-restate} 

\newtheorem{theorem}{Theorem}[section]
\newtheorem{lemma}[theorem]{Lemma}
\newtheorem{proposition}[theorem]{Proposition}

\theoremstyle{definition}
\newtheorem{definition}[theorem]{Definition}
\newtheorem{remark}[theorem]{Remark}
\newtheorem{example}[theorem]{Example}
\newtheorem{question}[theorem]{Question}
\newtheorem{problem}[theorem]{Problem}
\newtheorem{conjecture}[theorem]{Conjecture}
\numberwithin{equation}{section}

\newcommand{\N}{\mathbb{N}}
\newcommand{\Z}{\mathbb{Z}}
\newcommand{\Q}{\mathbb{Q}}
\newcommand{\R}{\mathbb{R}}
\newcommand{\C}{\mathbb{C}}
\newcommand{\T}{\mathbb{T}}

\newcommand{\F}{\mathbb{F}}

\newcommand{\p}{\mathfrak{p}}

\renewcommand{\P}{\mathbb{P}}

\newcommand{\bfP}{\mathbf{P}}

\newcommand{\mO}{\mathcal{O}}

\newcommand{\es}{\emptyset}

\newcommand{\eps}{\varepsilon}

\newcommand{\lpf}{\normalfont\text{lpf}}

\renewcommand{\tilde}{\widetilde}
\renewcommand{\hat}{\widehat}

\newcommand{\floor}[1]{\left\lfloor #1 \right\rfloor}

\DeclareMathOperator*{\E}{\text{\Large $\mathbb{E}$}}

\newcommand{\UClim}{\textup{UC-}\lim}

\newcommand{\ind}{\mathbbm{1}}

\newcommand{\ddeg}{\text{d-}\deg}

\newcommand{\innprod}[2]{\left\langle #1, #2 \right\rangle}
\newcommand{\norm}[2]{\left\| #2 \right\|_{#1}}

\newcommand{\Hil}{\mathcal{H}}

\renewcommand{\char}{\textup{char}}

\title[Polynomial actions of rings of integers and quasirandomness of Paley-type graphs]{Polynomial actions of rings of integers of global fields and quasirandomness of Paley-type graphs}

\date{\today}

\author{Ethan Ackelsberg}
\address{\'{E}cole Polytechnique F\'{e}d\'{e}rale de Lausanne, 1015-Lausanne, Switzerland}
\email{ethan.ackelsberg@epfl.ch}

\author{Vitaly Bergelson}
\address{Ohio State University, Columbus, OH 43210 USA}
\email{vitaly@math.ohio-state.edu}



\begin{document}

\maketitle

\begin{abstract}
	The goal of this paper is to undertake an in-depth study of the general phenomenon behind the Furstenberg--S\'{a}rk\"{o}zy theorem, which, in its modern form due to Kamae and Mend\`{e}s-France \cite{km}, states that if $E$ is a subset of the integers with positive density and $P(x) \in \Z[x]$ is an intersective polynomial (i.e., has a root mod $N$ for every $N \in \N$), then there are distinct elements $x, y \in E$ such that $x - y = P(n)$ for some some $n \in \Z$.
    
    In this paper, we identify a natural algebraic framework (rings of integers of global fields) for Furstenberg--S\'{a}rk\"{o}zy-type theorems.
    One of our main results establishes necessary and sufficient conditions for a polynomial to satisfy the Furstenberg--S\'{a}rk\"{o}zy theorem over the ring of integers of a global field, providing an extension of the result of Kamae and Mend\`{e}s-France and including (in an amplified form) results from \cite{br, bl} as rather special cases.

    The Furstenberg--S\'{a}rk\"{o}zy phenomenon goes beyond infinite rings and has interesting additional aspects when considered from an asymptotic point of view for families of finite rings.
    As an example, classical exponential sum estimates can be used to show that large subsets of finite fields contain the asymptotically (as the size of the field grows) ``correct'' number of pairs $(x,y)$ whose difference $y - x$ is a square.
    In \cite{ab_Sarkozy}, the class of polynomials satisfying this strong form of the Furstenberg--S\'{a}rk\"{o}zy theorem over finite fields was fully classified.
    In the present paper, we establish asymptotic results characterizing sequences of finite principal ideal rings that produce ``correct'' statistics in the Furstenberg--S\'{a}k\"{o}zy theorem and show that these families are much more general than the family of finite fields.
    
    As an application of our enhanced forms of the Furstenberg--S\'{a}rk\"{o}zy theorem over finite rings, we produce new families of examples of quasirandom graphs of algebraic origin, which can be viewed as a far-reaching generalization of the classical family of Paley graphs.
    The production of these new examples hinges on a two-way connection between \emph{asymptotic total ergodicity}---the phenomenon responsible for enhanced versions of the Furstenberg--S\'{a}rk\"{o}zy theorem over finite fields and rings---and quasirandomness.
\end{abstract}


\section{Introduction}

The main results obtained in this paper belong to the junction of three areas: number theory in global fields, ergodic theory, and the theory of quasirandomness of graphs.
The motivation for our work comes from the desire to better understand the general phenomenon which is epitomized by the Furstenberg--S\'{a}rk\"{o}zy theorem, stating that every set of positive density in the integers contains a square difference.
(More generally, one can obtain values of any so-called \emph{intersective} polynomial as a difference of elements of an arbtrary set of positive density in $\Z$; see Section \ref{sec:FS and TE} for a precise formulation.)
One of our results extends and unifies variants of the Furstenberg--S\'{a}rk\"{o}zy theorem previously obtained in \cite{bl} and \cite{br} and which deal with large subsets of the polynomial ring over a finite field and of rings of integers of number fields respectively.
The unifying framework is that of rings of integers of global fields, which we discuss below in Section \ref{sec:setting}.

The phenomenon behind the Furstenberg--S\'{a}rk\"{o}zy theorem also has a manifestation in large classes of finite rings, including finite fields as well as modular rings $\Z/N\Z$, Galois rings, pseudo-Galois rings, and chain rings, among others.
Furstenberg--S\'{a}rk\"{o}zy-type results in finite rings have an asymptotic nature and can be separated into two differently-behaving regimes of high characteristic (characteristic growing to infinity in a sequence of rings) and low characteristic (characteristic bounded while the size of the ring grows to infinity).

The high characteristic regime was treated in \cite{bb1,bb2}, where it was demonstrated that Furstenberg--S\'{a}rk\"{o}zy-type theorems in finite rings are tightly connected with the idea of \emph{asymptotic total ergodicity}.
Moreover, asymptotic total ergodicity allows one to circumvent ``local obstructions'' that are unavoidable in the setting of the integers and more general families of infinite rings.
As a consequence, it was shown in \cite{bb1,bb2} that one obtains the asymptotically ``correct statistics'' for polynomial configurations (involving, unlike the case of $\Z$, not necessarily intersective polynomials) in large subsets of finite rings (of high characteristic) when the sequence of rings under consideration is asymptotically totally ergodic.

What follows is a description of the main achievements of this paper.

First, we identify an algebraic framework (rings of integers of global fields) that unifies previous Furstenberg--S\'{a}rk\"{o}zy-type theorems into a single theorem, which has amplified versions of results from \cite{br, bl} as rather special cases and answers a question of L\^{e}, Liu, and Wooley \cite{llw}.

Second, by considering quotients of rings of integers or global fields, we obtain a rich class of finite rings that includes all of the examples handled in \cite{bb1,bb2} as well as several additional interesting sequences of rings in the low characteristic regime.
The algebraic structure of these rings is such that we are able to extend the notion of asymptotic total ergodicity to be applicable also in the low characteristic regime, and this expanded notion of asymptotic total ergodicity appropriately captures the statistical behavior of polynomial configurations in subsets of finite rings in both high and low characteristic.
This can be seen as a completion of the program set out in \cite{bb1}.

Finally, we establish in this paper a two-way connection between asymptotic total ergodicity and quasirandomness of certain graphs associated to polynomials over finite rings.
The original Furstenberg--S\'{a}rk\"{o}zy theorem came from Lov\'{a}sz's conjecture motivated by the properties of the classical Paley graphs (to be defined in Section \ref{sec:finitary applications}), which are well known to be quasirandom.
We will show that the key property of finite fields (over which Paley graphs are defined) responsible for generating quasirandomness is their asymptotic total ergodicity.
By connecting quasirandomness with asymptotic total ergodicity, we identify the algebraic and dynamical underpinnings for the quasirandomness of Paley graphs and are thus able to produce new families of examples of quasirandom graphs associated to polynomials over more general finite rings. \\

The structure of the paper is as follows.
In Section \ref{sec:setting}, we explain the algebraic setting for our work, collecting important facts about rings of integers of global fields and their quotients.
Section \ref{sec:FS and TE} discusses the main ergodic-theoretic content of the paper, including a discussion of the dynamical approach to the classical Furstenberg--S\'{a}rk\"{o}zy theorem and statements of our main results dealing with extensions of the Furstenberg--S\'{a}rk\"{o}zy theorem to the setting developed in Section \ref{sec:setting}.
After illuminating in Section \ref{sec:FS and TE} the central role played by the dynamical notion of \emph{total ergodicity} when dealing with polynomial configurations, we develop a finitary asymptotic version of total ergodicity in Section \ref{sec:ATE}.
The remainder of the main results appear in Section \ref{sec:finitary applications}.
Included therein are new Furstenberg--S\'{a}rk\"{o}zy-type theorem for finite principal ideal rings and a two-way connection between asymptotic total ergodicity and quasirandomness of Paley-type graphs.
The proofs of the main results appear in the remaining sections: ergodic theoretic results are proved in Sections \ref{sec:TE polynomial} and \ref{sec:nec/suff}; the main estimate connecting asymptotic total ergodicity with polynomial averages is proved in \ref{sec:ATE proof}; applications to finitary combinatorial results in the spirit of the Furstenberg--S\'{a}rk\"{o}zy theorem are proved in Section \ref{sec:FS good rings}; and results relating quasirandomness of Paley-type graphs with asymptotic total ergodicity are proved in Section \ref{sec:quasirandomness}.


\section*{Acknowledgments}

The first author acknowledges support from the Swiss National Science Foundation under Grant TMSGI2-211214.


\section{A disclaimer on our use of the term ``characteristic''}

One of the goals of this paper is to produce strong combinatorial results about polynomial patterns in large subsets of rings of integers of global fields and, as a finitary counterpart, in large subsets of finite principal ideal rings.
When working with polynomial expressions, the characteristic of the underlying ring is an important parameter to track.
In particular, whether or not the degree of a polynomial is smaller than the characteristic of the ring plays a crucial role.
In order to sensibly talk about the ``high'' and ``low'' characteristic settings for infinite rings, we use the slightly unusual convention that the characteristic of a ring is always equal to the additive order of the multiplicative identity element $1$, meaning, for example, that the characteristic of the integers is infinite (rather than the more typical convention that $\char(\Z) = 0$).


\section{Rings of integers of global fields and their quotients} \label{sec:setting}

As we have already stated, a natural setup for treating Furstenberg--S\'{a}rk\"{o}zy-type results is the context of rings of integers of global fields.
In this section, we discuss global fields and their properties that will be utilized in our later combinatorial results.

A \emph{global field} is a field of one of two types:

\begin{itemize}
	\item	a finite extension of $\Q$ (\emph{algebraic number field})
	\item	a finite extension of the field of rational functions $\F_q(t)$ over a finite field $\F_q$ (\emph{global function field}).
\end{itemize}

Global fields are important objects in algebraic number theory and are often studied in conjunction with their completions, known as \emph{local fields}.
We do not deal with local fields directly in this paper, as the family of global fields are sufficiently rich on their own for the combinatorial applications we have in mind; see Lemma \ref{lem: local field quotients} below.

Given a global field $K$, the \emph{ring of integers} $\mO_K$ is the integral closure of $\Z$ (in the case that $K$ is an algebraic number field) or $\F_q[t]$ (if $K$ is a global function field) inside of $K$.
(Recall that the \emph{integral closure} of a ring $R$ inside of another ring $S$ is the set of elements $s \in S$ that are obtained as roots of monic polynomials with coefficients coming from $R$.)
Our main object of study is the Furstenberg--S\'{a}rk\"{o}zy phenomenon in rings of integers of global fields and their (finite) quotients.
In the following proposition, we collect several basic properties of rings of integers of global fields.

\begin{proposition} \label{prop: rings of integers}
	Let $K$ be a global field and $\mO_K$ its ring of integers.
	\begin{enumerate}
		\item	$\mO_K$ is a Dedekind domain: every nontrivial ideal $\{0\} \lneq I \lneq \mO_K$ factors (uniquely, up to permutation) as a product of prime ideals;
		\item	Every nonzero ideal $\{0\} \ne I \le \mO_K$ has finite index in $\mO_K$;
		\item	The characteristic of $\mO_K$ is either infinite (if $K$ is an algebraic number field) or a prime number (if $K$ is a global function field).
	\end{enumerate}
\end{proposition}

\begin{proof}
    Property (1) follows from \cite[Theorem 3.29]{milne}.

    Property (2) is classical and can be found, e.g., in \cite[Section I.8]{janusz}.

    (3) If $K$ is an algebraic number field, then $\Z \subseteq \mO_K$, so $\mO_K$ has characteristic infinity.
    Similarly, if $K$ is a global function field, then $\mO_K$ contains $\F_q[t]$ for some finite field $\F_q$, so $\char(\mO_K) = \char(\F_q)$ is prime.
\end{proof}

The Dedekind property of rings of integers of global fields (item (1) in Proposition \ref{prop: rings of integers}) enables us to make the following definition.

\begin{definition} \label{defn: lpf}
    Given a nontrivial ideal $\{0\} \lneq I \lneq \mO_K$ and its factorization $I = \p_1^{e_1} \dots \p_r^{e_r}$ into primes, we define the \emph{least prime factor} of $I$ as $\lpf(I) = \min_{1 \le i \le r} [\mO_K : \p_i]$.
\end{definition}

Equivalently, $\lpf(I)$ is the minimum index of all proper ideals containing $I$.
As the next proposition shows, $\lpf(I)$ can also be computed in terms of ideals of the quotient ring $\mO_K/I$.

\begin{proposition} \label{prop: lpf finite ring ideals}
	Let $K$ be a number field.
	Let $\{0\} \ne I \le \mO_K$ be a nonzero ideal.
	Then
	\begin{align*}
		\lpf(I) & = \min \left\{ |J| : J~\text{is a nonzero ideal of the quotient ring}~\mO_K/I \right\} \\
		 & = \min \left\{ [\mO_K/I : J] : J~\text{is a proper ideal of the quotient ring}~\mO_K/I \right\}.
	\end{align*}
\end{proposition}

\begin{proof}
	The ideals of $\mO_K/I$ are all products of the form $J = \p_1^{s_1} \dots \p_r^{s_r}$ with $0 \le s_i \le e_i$.
	Such an ideal has cardinality
	\begin{equation*}
		|J| = \prod_{i=1}^r [\mO_K : \p_i]^{e_i - s_i}
	\end{equation*}
	and index
	\begin{equation*}
		[\mO_K/I : J] = \prod_{i=1}^r [\mO_K : \p_i]^{s_i}
	\end{equation*}
    Both quantities are clearly minimized by $\min_{1\le i \le r} [\mO_K:\p_i] = \lpf(I)$.
\end{proof}

Based on Proposition \ref{prop: lpf finite ring ideals}, we define $\lpf$ for a finite commutative ring by
\begin{equation*}
	\lpf(R) = \min \left\{  [R : J] : J~\text{is a proper ideal of}~R \right\}.
\end{equation*}

\begin{example} \label{eg: F_4 realization}
	The field with four elements, $\F_4$, can be obtained as a quotient of a ring of integers of a global field in many different ways.
	The important feature of Proposition \ref{prop: lpf finite ring ideals} is that no matter how $\F_4$ is generated as $\mO_K/I$, the value of $\lpf(I)$ with always be the same ($\lpf(\F_4) = 4$) and is computable inside the quotient ring $\F_4$ itself.
	We check this for a few explicit constructions of $\F_4$.
	
	\begin{itemize}
		\item	The most straightforward way of obtaining $\F_4$ is $\F_4 = \F_4[t]/\langle t \rangle$.
			Here, the ideal $I = \langle t \rangle$ is a prime ideal with index $[\F_4[t] : I] = 4$.
		\item	We can also obtain $\F_4$ as the quotient ring $\F_2[t]/\langle t^2 + t + 1 \rangle$.
			Once again, $I = \langle t^2 + t + 1 \rangle$ is a prime ideal of index 4 in $\F_2[t]$.
		\item	Finally, we can obtain $\F_4$ from an algebraic number field as follows.
			Let $\mO_K = \Z[i]$, and let $I = \langle 1 + i \rangle$.
			Then $\mO_K/I$ is isomorphic to $\F_4$, and $I$ is a prime ideal of index 4.
	\end{itemize}
\end{example}

The finite commutative rings that we will study in this paper are the quotient rings $\mO_K/I$, where $\mO_K$ is the ring of integers of a global field $K$ and $I$ is a nontrivial ideal of $\mO_K$.
It turns out that this class of rings has a simple characterization:

\begin{proposition} \label{prop: good iff principal}
    Let $R$ be a finite commutative ring.
    The following are equivalent:
    \begin{enumerate}[(i)]
        \item there exists a global field $K$ and a nontrivial ideal $I \le \mO_K$ such that $R \cong \mO_K/I$;
        \item $R$ is a finite principal ideal ring.
    \end{enumerate}
\end{proposition}

The proof of Proposition \ref{prop: good iff principal} is made easier with the help of the following lemma, showing that finite quotients of rings of integers of local fields\footnote{As stated above, local fields are completions of global fields. Every local field is (isomorphic to) one of the following:
\begin{itemize}
    \item the real numbers $\R$,
    \item the complex numbers $\C$,
    \item a finite extension of the $p$-adic numbers $\Q_p$ for a prime number $p$, or
    \item the field of formal Laurent series $\F_q((t^{-1}))$ over a finite field $\F_q$.
\end{itemize}} can always be obtained as quotients of rings of integers of global fields.

\begin{lemma} \label{lem: local field quotients}
    Let $L$ be a non-Archimedean local field\footnote{Local fields are equipped with absolute value functions satisfying $|xy| = |x||y|$ and the triangle inequality $|x+y| \le |x|+|y|$.
    Such absolute value functions satisy exactly one of the following properties:
    \begin{itemize}
        \item Archimedean property: for every $x \ne 0$, there exists $n \in \N$ such that $|nx| > 1$
        \item ultrametric triangle inequality: $|x+y| \le \max\{|x|,|y|\}$.
    \end{itemize}
    If the absolute value satisfies the Archimedean property, the local field is said to be Archimedean.
    If not (in which case it satisfies the ultrametric triangle inequality), then the local field is said to be non-Archimedean.
    Out of the list of local fields given in the previous footnote, only $\R$ and $\C$ are Arhcimedean and the rest are non-Archimedean.} with ring of integers $\mO_L$, and let $J$ an ideal of $\mO_L$ of finite index.
    Then there exists a global field $K$ and an ideal $I \le \mO_K$ such that $\mO_L/J \cong \mO_K/I$.
\end{lemma}

\begin{proof}
    Suppose $L$ is the completion of a global field $K$.
    Then we also have that $J$ is the completion of an ideal $I \le K$, and $\mO_L/J \cong K/I$; see \cite[Lemma 7.25]{milne}.
\end{proof}

We can now prove Proposition \ref{prop: good iff principal}.

\begin{proof}[Proof of Proposition \ref{prop: good iff principal}]
    (i)$\implies$(ii).
    We learned of this argument from a comment by Pete L. Clark on \emph{MathOverflow} \cite{clark_MO}.
    Suppose $R \cong \mO_K/I$ for some global field $K$ and ideal $I \le \mO_K$.
    The ring $\mO_K$ is a Dedekind domain by Proposition \ref{prop: rings of integers}, so its nontrivial quotients are principal Artinian rings by \cite[Theorem 20.11]{clark_book}.
    Hence, $R \cong \mO_K/I$ is a principal ideal ring. \\

    (ii)$\implies$(i).
    Suppose $R$ is a finite principal ideal ring.
    We may factor (uniquely, up to permutation) $R = R_1 \times \ldots \times R_k$ as a product of local rings by \cite[Theorem 3.1.4]{biniflamini}.
    Since each of the local factors is a quotient of $R$, the local factors are themselves principal ideal rings.
    For each $i$, we can obtain $R_i$ as a quotient of the ring of integers of a $p$-adic field for suitable $p$ (depending on $i$) and then use Lemma \ref{lem: local field quotients}, a form of weak approximation, and the Chinese remainder theorem to realize $R$ as a quotient ring of the form $\mO_K/I$ for some global field $K$ and ideal $I \le \mO_K$.
    We refer the reader to a \emph{MathOverflow} comment of Keith Conrad \cite{conrad_MO} for details.
\end{proof}

All of the explicit families of examples of finite commutative rings appearing in \cite{bb2} are principal ideal rings and can thus be realized as quotients of rings of integers of global fields:

\begin{itemize}
    \item Finite fields.
        A finite field $\F$ can be obtained, for example, as the quotient $\F[t]/\langle t \rangle$.
    \item   Modular rings.
        The modular ring $\Z/N\Z$ is manifestly a quotient of $\Z$.
    \item   Galois rings.
        The Galois ring $GR(p^n,r)$ is usually constructed as a quotient ring $GR(p^n,r) = \Z[x]/\langle p^n, f(x) \rangle$, where $f$ is a monic polynomial of degree $r$ that is irreducible mod $p$.
        However, one can instead construct $GR(p^n,r)$ as a quotient $\Z[\xi]/\langle p^n \rangle$ where $\xi$ is a root of the polynomial $f$.
    \item Pseudo-Galois rings.
        Pseudo-Galois rings are a generalization of Galois rings and are of the form $PGR(N,r) = \Z[x]/\langle N, f(x) \rangle$ for some $N$ and a degree $r$ polynomial $f$ that is irreducible mod $N$.
        Again, these can be obtained as quotients $PGR(N,r) = \Z[\xi]/\langle N \rangle$.
    \item Chain rings.
        A ring $R$ is a chain ring if all of its ideals form a chain under inclusion.
        All finite chain rings are obtained as quotients $GR(p^n,r)[x]/\langle p^{n-1}x^t, g(x)\rangle$, where $g$ is an Eisenstein polynomial of degree $k$.
        The parameters $p, n, r, t, k$ are called the \emph{invariants} of the chain ring.
        Chain rings are also examples of quotient rings obtained from global fields, as can be seen by first taking the quotient $GR(p^n,r)[x]/\langle g(x) \rangle$, which is isomorphic to $\Z[\xi,\zeta]/\langle p^n \rangle$, where $\zeta$ is a root of $g$ and $GR(p^n,r) \cong \Z[\xi]/\langle p^n \rangle$.
\end{itemize}

However, not every finite commutative ring is principal.
For example, the ideal $\langle x,y \rangle$ in $\F_p[x,y]/\langle x,y\rangle^2$ is nonprincipal; see \cite{conrad_MO, clark_MO}.


\section{Totally ergodic actions of rings of integers of global fields and polynomial ergodic averages} \label{sec:FS and TE}

We begin this section by reviewing Furstenberg's dynamical approach to the Furstenberg--S\'{a}rk\"{o}zy theorem over $\Z$.
This will enable our discussion in the rest of the section, which deals with extensions of the Furstenberg--S\'{a}rk\"{o}zy theorem to rings of integers of global fields.

The first step is to translate a combinatorial statement about subsets of $\Z$ into a dynamical statement about actions of $\Z$ on probability spaces via measure-preserving transformations.
This translation is provided by the Furstenberg correspondence principle (Theorem \ref{thm: correspondence} below).
Recall that an \emph{(invertible) measure-preserving system} is a triple $(X, \mu, T)$, where $(X, \mu)$ is a standard probability space (for which we omit explicit reference to the $\sigma$-algebra on $X$) and $T : X \to X$ is measurable and invertible with measurable inverse such that the measure $\mu$ is preserved by $T$, i.e. for any measurable subset $A \subseteq X$, one has $\mu(T^{-1}A) = \mu(A)$.
The system $(X, \mu, T)$ is \emph{ergodic} if every $T$-invariant measurable subset $A \subseteq X$ has measure 0 or 1.

\begin{theorem} \label{thm: correspondence}
    Let $E \subseteq \Z$ be a set with positive upper Banach density
    \begin{equation*}
        d^*(E) = \limsup_{N - M\to \infty} \frac{\left| E \cap \{M+1, M+2, \dots, N\}\right|}{N-M} > 0.
    \end{equation*}
    Then there exists an ergodic invertible measure-preserving system $(X, \mu, T)$ and a measurable set $A \subseteq X$ with $\mu(A) = d^*(E)$ such that
    \begin{equation*}
        d^* \left( \bigcap_{j=1}^k (E - n_j) \right) \ge \mu \left( \bigcap_{j=1}^k T^{-n_j}A \right)
    \end{equation*}
    for every $k \in \N$ and $n_1, \dots, n_k \in \Z$.
\end{theorem}

\begin{remark}
    The formulation of the Furstenberg correspondence principle given in Theorem \ref{thm: correspondence} (with the ability to assume that the measure-preserving system $(X, \mu, T)$ is ergodic) first appeared in \cite[Proposition 3.1]{bhk}.
    The proof given in \cite{bhk} is based on an earlier argument of Furstenberg \cite[Lemma 3.7]{furstenberg_book} with only small modifications to guarantee ergodicity of the resulting system.
    For additional discussion on the Furstenberg correspondence principle and its different variants, see \cite[p. 7]{bf}.
\end{remark}

Now we wish to prove a statement about recurrence in an ergodic measure-preserving system $(X, \mu, T)$.
We will make use of the following notions from ergodic theory.

\begin{itemize}
    \item The transformation $T$ induces a unitary operator (the \emph{Koopman operator}) on $L^2(\mu)$ by $U_Tf = f \circ T$.
    \item We say that a number $\lambda \in \C$ is an \emph{eigenvalue} of the system $(X, \mu, T)$ if it is an eigenvalue of the Koopman operator $U_T$.
    That is, there exists a nonzero $f \in L^2(\mu)$ such that $U_Tf = \lambda f$.
    (We work here with $L^2(\mu)$ as a Hilbert space over the complex numbers.)
    All eigenvalues are of modulus 1, since $U_T$ is unitary.
    The number $1$ is always an eigenvalue (consider the constant function $f = 1$), and for ergodic systems, all of the eigenspaces are one-dimensional.
    If an eigenvalue $\lambda$ is a root of unity, we call $\lambda$ a \emph{rational eigenvalue}.
    \item The system $(X, \mu, T)$ is \emph{periodic} if there exists $k \in \N$ such that $T^kx = x$ for almost every $x \in X$.
    \item A system $(Y, \nu, S)$ is a \emph{factor} of $(X, \mu, T)$ if there exist full measure subsets $X_0 \subseteq X$ and $Y_0 \subseteq Y$ and a measurable surjection $\pi : X_0 \to Y_0$ such that $\pi \circ T = S \circ \pi$ and $\pi_*\mu = \nu$.
\end{itemize}

Consider now an ergodic system $(X, \mu, T)$ and a measurable set $A \subseteq X$ with $\mu(A) > 0$.
The function $\varphi(n) = \mu(A \cap T^{-n}A)$ is positive definite so can be represented as the Fourier transform of a positive measure $\nu$ on $\T$ by Herglotz's theorem; see, e.g., \cite[Chapter 4]{nadkarni}.
That is,
\begin{equation} \label{eq:spectral form of correlation sequence}
    \mu(A \cap T^{-n}A) = \int_0^1 e^{-2\pi inx}~d\nu(x)
\end{equation}
for every $n \in \Z$.
For the purposes of studying polynomial configurations in sets of positive density in $\Z$, we want to analyze the behavior of $\mu(A \cap T^{-P(n)}A)$ for a polynomial $P(x) \in \Z[x]$.
Due to the identity \eqref{eq:spectral form of correlation sequence}, the sequences $e^{2\pi i P(n)x}$ for $x \in [0,1)$ play a pivotal role in this analysis.
There is a simple dichotomy in the behavior of polynomial phase functions: if $x$ is irrational and $P$ is nonconstant, then
\begin{equation} \label{eq:uniform Weyl}
    \lim_{N-M \to \infty} \frac{1}{N-M} \sum_{n=M+1}^N e^{2\pi i P(n)x} = 0
\end{equation}
by Weyl's equidistribution theorem\footnote{Weyl proved his eponymous equidistribution theorem in a weaker form than we have stated here.
Namely, he showed
\begin{equation*}
    \lim_{N \to \infty} \frac{1}{N} \sum_{n=1}^N e^{2\pi i P(n)x} = 0,
\end{equation*}
corresponding to \emph{uniform distribution} of the sequence $P(n)x$ mod 1; see \cite[Theorem 9]{weyl}.
(The limit \eqref{eq:uniform Weyl} corresponds to \emph{well-distribution} of the sequence $P(n)x$ mod 1.)
However, the note after Theorem 9 in Weyl's paper is sufficient to deduce well-distribution.
The first source with an explicit statement of the well-distribution of polynomial sequences mod 1 appears to be the work of Lawton \cite{lawton}.}; if $x$ is rational or $P$ is constant, then $e^{2\pi i P(n)x}$ is a periodic sequence.
To distinguish these behaviors, we may decompose the measure $\nu$ as $\nu = \nu_{rat} + \nu_{te}$, where $\nu_{rat}(B) = \nu(B \cap \Q)$ and $\nu_{te}(B) = \nu(B \setminus \Q)$ for Borel subsets $B \subseteq [0,1)$.
The measure $\nu_{rat}$ is discrete, and each of its atoms corresponds to a rational eigenvalue of $T$.
(Namely, if $q$ is an atom of $\nu_{rat}$, then $e^{2\pi i q}$ is an eigenvalue of $T$.)
As an immediate consequence, if $T$ has no nontrivial rational eigenvalues and $P$ is nonconstant,
\begin{equation*}
    \lim_{N-M \to \infty} \frac{1}{N-M} \sum_{n=M+1}^N \mu(A \cap T^{-P(n)}A) = \nu(\{0\}) = \lim_{N-M \to \infty} \frac{1}{N-M} \sum_{n=M+1}^N \mu(A \cap T^{-n}A),
\end{equation*}
and the final expression on the right hand side is equal to $\mu(A)^2$ by the mean ergodic theorem.
In particular, there exist (many values of) $n \in \Z$ such that $\mu(A \cap T^{-P(n)}A) > 0$.
The presence or absence of rational eigenvalues is characterized by the dynamical property of \emph{total ergodicity}, which is usually defined using condition (i) in the following theorem.

\begin{theorem} \label{thm: TE equivalences Z}
    Let $(X, \mu, T)$ be an ergodic system.
    The following are equivalent:
    \begin{enumerate}[(i)]
        \item for any $k \in \N$, the system $(X, \mu, T^k)$ is ergodic;
        \item $(X, \mu, T)$ has no nontrivial periodic factors;
        \item ($X, \mu, T)$ has no nonzero rational eigenvalues;
        \item for any nonconstant polynomial $P(x) \in \Z[x]$ and any $f \in L^2(\mu)$,
            \begin{equation*}
                \lim_{N-M \to \infty} \frac{1}{N-M} \sum_{n=M+1}^N T^{P(n)}f = \int_X f~d\mu.
            \end{equation*}
    \end{enumerate}
\end{theorem}

Now, in the case that $T$ is not totally ergodic, the measure $\nu_{rat}$ takes the form $\nu_{rat} = \mu(A)^2 \delta_0 + \sum_j c_j \delta_{q_j}$ for some finite or countably infinite sequence of rational points $q_j \in (0,1)$, where the coefficients $c_j > 0$ satisfy $\sum_j c_j = \nu_{rat}([0,1)) - \mu(A)^2 \le \nu([0,1)) - \mu(A)^2 = \mu(A) - \mu(A)^2$.
Let $\eps > 0$ and take $J$ sufficiently large so that $\sum_{j>J}c_j < \eps$.
Then, up to an $\eps$ approximation, we have
\begin{equation*}
    \mu(A \cap T^{-P(n)}A) \approx \mu(A)^2 + \sum_{j=1}^J c_j e^{2\pi i P(n)q_j} + \int_0^1 e^{2\pi i P(n)x}~d\nu_{te}(x).
\end{equation*}
In light of Weyl's equidistribution theorem, averaging over $n \in \Z$ eliminates the last term, so to prove positivity of the expression $\mu(A \cap T^{-P(n)}A)$, we want to establish positivity of the term
\begin{equation*}
    \sum_{j=1}^J c_j e^{2\pi i P(n)q_j}
\end{equation*}
for a large set of $n \in \Z$.
This is achievable under the following assumption about the polynomial $P$:

\begin{definition} \label{defn: intersective Z}
    A polynomial $P(x) \in \Z[x]$ is called \emph{intersective} if $P$ has a root mod $m$ for every $m \in \N$.
\end{definition}

Supposing that $P$ is intersective and taking $m$ to be a common denominator for all of the numbers $q_j$, $j=1, \dots, J$, we may find $r \in \N$ such that $P(mn+r)q_j \in \Z$ for every $n \in \Z$ and $j \in \{1, \dots, J\}$.
Then
\begin{equation*}
    \lim_{N-M \to \infty} \frac{1}{N-M} \sum_{n=M+1}^N \mu(A \cap T^{-P(mn+r)}A) \ge \mu(A)^2 + \sum_{j=1}^J c_j - \eps \ge \mu(A)^2 - \eps.
\end{equation*}
The argument sketched above leads to the following ``necessary and sufficient'' form of the Furstenberg--S\'{a}rk\"{o}zy theorem:

\begin{theorem} \label{thm: necessary and sufficient FS}
    Let $P(x) \in \Z[x]$.
    The following are equivalent:
    \begin{enumerate}[(i)]
        \item $P$ is intersective;
        \item for any measure-preserving system $\left( X, \mu, T \right)$ and any $A \subseteq X$ with $\mu(A) > 0$, there exists $n \in \Z \setminus \{0\}$ such that
        \begin{equation*}
            \mu \left( A \cap T^{-P(n)}A \right) > 0.
        \end{equation*}
        \item for any measure-preserving system $\left( X, \mu, T \right)$, any $A \subseteq X$ with $\mu(A) > 0$, and any $\eps > 0$, the set
        \begin{equation*}
            \left\{ n \in \Z : \mu \left( A \cap T^{-P(n)}A \right) > \mu(A)^2 - \eps \right\}
        \end{equation*}
        is syndetic.
        \item for any set $E \subseteq \Z$ with $d^*(E) > 0$, there exist $x, y \in E$ and $n \in \Z \setminus \{0\}$ such that $x - y = P(n)$.
    \end{enumerate}
\end{theorem}

\begin{remark} \label{rem: necessary and sufficient FS}
    The observation that intersectivity plays a crucial role in the recurrence properties of polynomial actions is due to Kamae and Mend\`{e}s-France \cite{km}, who proved that the set of values of an intersective polynomial is a \emph{van der Corput set}.
    We will not need the notion of van der Corput sets elsewhere in the paper and so do not give its formal definition.
    The important property for our present discussion (which was also observed by Kamae and Mend\`{e}s-France) is that the set of values $\{P(n) : n \in \Z\}$ being a van der Corput set implies that item (iv) holds.

    The equivalence between (ii) and (iv) can be deduced as a straightforard application of the Furstenberg correspondence principle (Theorem \ref{thm: correspondence}), and the upgrade from (ii) to the \emph{a priori} stronger statement in (iii) follows easily from the method of proof outlined above.
\end{remark}

One of the goals of this paper is to establish a version of Theorem \ref{thm: necessary and sufficient FS} in the setting of rings of integers of global fields.
In the infinite characteristic case (rings of integers of number fields), there is a fairly direct extension, which was carried out in \cite{br}.
The finite characteristic setting (rings of integers of global function fields) is more subtle.
Nevertheless, we obtain extensions that are valid for all global fields.
These are necessarily more cumbersome than the results in the special cases of the integers and rings of integers of number fields, so we provide comments throughout on the simplifications that arise from having infinite characteristic.
It is useful to keep in mind the ring $\F_p[t]$ as a motivating example for statements involving the characteristic.

Let $\mO_K$ be the ring of integers of a global field $K$.
By a \emph{measure-preserving $\mO_K$-system}, we will mean a tuple $(X, \mu, (T_n)_{n \in \mO_K})$ such that $(X, \mu)$ is a standard probability space and $(T_n)_{n \in \mO_K}$ is an action of the group $(\mO_K,+)$ by measure-preserving transformations.
The system $(X, \mu, (T_n)_{n \in \mO_K})$ is \emph{ergodic} if every $(T_n)_{n \in \mO_K}$-invariant set has measure zero or one.
We say that a measure-preserving $\mO_K$-system $\left( X, \mu, (T_n)_{n \in \mO_K} \right)$ is \emph{periodic} if there exists $m \in \mO_K \setminus \{0\}$ such that $T_{mn+r} = T_r$ for every $n, r \in \mO_K$.
In this case, we say that $m$ is a \emph{period} of the system.
Note that the collection of all periods (together with 0) forms an ideal of $\mO_K$.
The ring $\mO_K$ is not necessarily a principal ideal domain, so it does not make sense in this context to describe an element $m \in \mO_K \setminus \{0\}$ as \emph{the} period of a system.
Instead, we describe the periodicity of an $\mO_K$-system in terms of its ideal of periods.

In the case $\mO_K = \Z$, every ergodic action on a finite set is periodic, and the same remains true for rings of integers of number fields.
However, when $K$ is a global function field of characteristic $p$, the group $(\mO_K,+)$ may act ergodically on a finite set without being periodic.
Namely, if $H \le (\mO_K,+)$ is a finite index subgroup that is not an ideal, then the natural action of $\mO_K$ on the finite set $\mO_K/H$ is not periodic in the sense defined above.
If $\mO_K = \F_p[t]$, one can take as an example the additive subgroup $H = \{\sum_j a_jt^j \in \F_p[t] : a_1 = 0\}$.
For additional discussion, see Remark \ref{rem: total ergodicity} below.

A nonzero function $f \in L^2(\mu)$ is an \emph{eigenfunction} of the system $(X, \mu, (T_n)_{n \in \mO_K})$ if there is a map $\chi : \mO_K \to S^1$ such that $f \circ T_n = \chi(n)f$ for every $n \in \mO_K$.
Since $(T_n)_{n \in \mO_K}$ is a group action, the map $\chi$ is necessarily a group character, which we call the \emph{eigencharacter} associated to $f$. 
We say that a character $\chi : \mO_K \to S^1$ is \emph{rational} if it is periodic in that sense that there exists $m \in \mO_K \setminus \{0\}$ such that $\chi(mn+r) = \chi(r)$ for every $n, r \in \mO_K$.
If $\chi$ is not rational, we say it is \emph{irrational}.

A first step towards generalizing Theorem \ref{thm: necessary and sufficient FS} to the setting of rings of integers of global fields is to observe that, as described in Theorem \ref{thm: TE equivalences Z} for actions of $\Z$, polynomial ergodic averages (in the low degree/high characteristic regime) are controlled by the rational eigencharacters of the underlying system.
We introduce the following terminology and notation for discussing averages in groups.
A \emph{F{\o}lner sequence} in $(\mO_K,+)$ is a sequence of finite subsets $(\Phi_N)_{N \in \N}$ of $\mO_K$ such that for every $n \in \mO_K$,
\begin{equation*}
    \lim_{N \to \infty} \frac{|\Phi_N \triangle (\Phi_N + n)|}{|\Phi_N|} = 0.
\end{equation*}
If $(x_n)_{n \in \mO_K}$ is a sequence indexed by $\mO_K$ in a topological vector space $V$, we say that $(x_n)_{n \in \mO_K}$ has \emph{uniform Ces\`{a}ro limit} equal to $x \in V$, denoted by $\UClim_{n \in \mO_K} x_n = x$, if for every F{\o}lner sequence $(\Phi_N)_{N \in \N}$ in $\mO_K$,
\begin{equation*}
    \lim_{N \to \infty} \frac{1}{|\Phi_N|} \sum_{n \in \Phi_N} x_n = x.
\end{equation*}

\begin{theorem} \label{thm: TE}
    Let $K$ be a global field with ring of integers $\mO_K$.
	Suppose $(T_n)_{n \in \mO_K}$ is an ergodic measure-preserving action of $(\mO_K, +)$ on a probability space $(X, \mu)$.
	The following are equivalent:
	\begin{enumerate}[(i)]
		\item	for any nonzero ideal $\{0\} \ne I \le \mO_K$, the system $\left( X, \mu, (T_n)_{n \in I} \right)$ is ergodic;
		\item	$\left( X, \mu, (T_n)_{n \in \mO_K} \right)$ has no nontrivial periodic factors;
		\item	$\left( X, \mu, (T_n)_{n \in \mO_K} \right)$ has no nonzero rational eigencharacters;
		\item	for any nonconstant polynomial $P(x) \in \mO_K[x]$ of degree $\deg{P} < \textup{char}(K)$ and any $f \in L^2(\mu)$,
			\begin{equation} \label{eq:polynomial avg proj}
				\UClim_{n \in \mO_K} T_{P(n)} f = \int_X f~d\mu.
			\end{equation}
	\end{enumerate}
\end{theorem}

We remind the reader that we use the convention that $\textup{char}(K) = \infty$ when $K$ is a number field, so the condition $\deg{P} < \textup{char}(K)$ appearing in item (iv) is trivially satisfied for every polynomial $P$ in the number field setting.

We will prove Theorem \ref{thm: TE} in Section \ref{sec:TE polynomial}.

\begin{definition} \label{defn: TE}
    We say that a system $(X, \mu, (T_n)_{n \in \N})$ is \emph{ergodic along ideals} if it satisfies the properties in Theorem \ref{thm: TE}.
\end{definition}

When considering limits of polynomial ergodic averages of the form $\UClim_{n \in \mO_K} T_{P(n)}f$ for polynomials of high degree ($\deg{P} \ge \textup{char}(K)$), the presence of certain irrational eigencharacters can contribute in unexpected ways, and \eqref{eq:polynomial avg proj} may fail even in systems without rational eigencharacters.
This is illustrated by the following example.

\begin{example} \label{eg: weakly rational}
    Let $b = (b_j)_{j=0}^{\infty}$ be the sequence
    \begin{equation*}
        b_j = \begin{cases}
            \floor{j\sqrt{2}} \bmod{3}, & \text{if}~3 \nmid j; \\
            0, & \text{if}~3 \mid j~\text{and}~j \ne 0; \\
            1, & \text{if}~j = 0.
        \end{cases}
    \end{equation*}
    (The exact form of $b$ is not important for the present example.
    The relevant property is that $b$ is an aperiodic sequence that is eventually periodic along the subsequence indexed by multiples of $3$.)
    Consider the character $\chi : \F_3[t] \to S^1$ defined by
    \begin{equation*}
        \chi \left( \sum_{j=0}^k a_j t^j \right) = e \left( \frac{\sum_{j=0}^k a_jb_j}{3} \right),
    \end{equation*}
    where $e(x) = e^{2\pi ix}$.
    The aperiodicity of the sequence $b$ means that $\chi$ is an irrational character, and hence the action of $(\F_3[t],+)$ on the cube roots of unity $C_3$ given by $T_n \zeta = \chi(n)\zeta$ is ergodic along ideals.
    However, for any $n = \sum_{j=0}^k a_jt^j \in \F_3[t]$, we have
    \begin{equation*}
        \chi(n^{6}) = \chi \left( (n^2)^3 \right) = \chi \left( \sum_{i,j=0}^k a_i a_j t^{3(i+j)} \right) = e \left( \frac{\sum_{i,j=0}^k a_i a_j b_{3(i+j)}}{3} \right) = e \left( \frac{a_0^2}{3} \right).
    \end{equation*}
    Thus, rather than obtaining the limiting behavior in \eqref{eq:polynomial avg proj}, we have
    \begin{equation*}
        \UClim_{n \in \F_3[t]} T_{n^6}f(\zeta) = \frac{1}{3} f(\zeta) + \frac{2}{3} f \left( e^{2\pi i/3} \zeta \right)
    \end{equation*}
    for $f : C_3 \to \C$.
\end{example}

In order to obtain a valid result dealing with high degree polynomials, we need to understand and identify the crucial properties underlying situations similar to the preceding example.

\begin{definition} \label{defn: additive}
    A polynomial $\eta(x) \in \mO_K[x]$ is \emph{additive} if $\eta(x+y) = \eta(x) + \eta(y)$ for every $x, y \in \mO_K$.
\end{definition}

In characteristic $p$, additive polynomials are exactly the polynomials of the form $\eta(x) = \sum_{j=0}^N c_j x^{p^j}$ for $N \ge 0$ and $c_j \in \mO_K$.
We say that a character $\chi : \mO_K \to S^1$ is \emph{weakly rational} if there exists an additive polynomial $\eta(x) \in \mO_K[x]$ such that $\chi \circ \eta$ is a nontrivial rational character.
The character in Example \ref{eg: weakly rational} is an example of a weakly rational character, with $n \mapsto \chi(n^3)$ being rational.
A nontrivial character is \emph{strongly irrational} if it is not weakly rational.
We say that a system $(X, \mu, (T_n)_{n \in \mO_K})$ is \emph{strongly ergodic} if it is ergodic and has no nontrivial weakly rational eigencharacters.

\begin{theorem} \label{thm: strongly ergodic}
    Let $K$ be a global field with ring of integers $\mO_K$.
    Let $(X, \mu, (T_n)_{n \in \N})$ be an ergodic $\mO_K$-system.
    The following are equivalent:
    \begin{enumerate}[(i)]
        \item $(X, \mu, (T_n)_{n \in \N})$ is strongly ergodic;
        \item for any polynomial $P(x) \in \mO_K[x]$ with $P(0) = 0$ and any $f \in L^2(\mu)$,
			\begin{equation*}
				\UClim_{n \in \mO_K} T_{P(n)} f = \pi_P(f),
			\end{equation*}
			where $\pi_P$ is the orthogonal projection onto the space $\Hil_P = \{g \in L^2(\mu) : T_{P(n)} g = g~\text{for all}~n \in \mO_K\}$.
    \end{enumerate}
\end{theorem}

Note that condition (ii) in Theorem \ref{thm: strongly ergodic} differs from condition (iv) in Theorem \ref{thm: TE} in that there is no restriction on the degree of the polynomial $P$ but $P$ must have zero constant term.
The cost of upgrading to the high degree regime is the requirement of strong ergodicity and that we only obtain a projection $\pi_P(f)$ rather than the integral $\int_X f~d\mu$.
The phenomenon that a polynomial ergodic average converges to a nontrivial projection resembles the polynomial ergodic theorem for actions of fields due to Larick (see \cite{larick} and \cite[Theorem 3.10]{blm}).

\begin{remark} \label{rem: total ergodicity}
    For $\Z$-systems (and, more generally, for $\mO_K$-systems when $K$ is a number field), the notions of ergodic along ideals and strongly ergodic coincide and are equivalent to the classical notion of total ergodicity.
    From a group-theoretic point of view, the notion of total ergodicity has other generalizations: for instance, one may hypothesize that the action of every finite index subgroup is ergodic (or, equivalently, that there are no nonconstant functions $f \in L^2(\mu)$ with finite orbit under the action of the group).

    All three notions (ergodicity along ideals, strong ergodicity, and ergodicity along finite index subgroups) become distinct when $K$ is a global function field.
    To see this, note that if $K$ is a global function field of characteristic $p$, then every group character $\chi : \mO_K \to S^1$ takes values in the $p$th roots of unity.
    Therefore, $\ker{\chi}$ is a subgroup of index $p$.
    It follows that the presence of \emph{any} nontrivial eigencharacter (rational or not) means that there is a finite index subgroup along which the action is non-ergodic, and the corresponding eigenfunction will have finite orbit (consisting of multiples of the eigenfunction by $p$th roots of unity).
    Hence, ergodicity along every finite index subgroup is equivalent to \emph{weak mixing} (absence of eigencharacters) for an action of $\mO_K$ when $K$ is a global function field.

    The relations between the different properties are summarized by the following diagrams.

    \begin{figure}[H]
        \begin{floatrow}
        \centering
        \ffigbox{\begin{tikzcd}
            \text{NUMBER FIELDS} \\
            \text{weakly mixing} \arrow[d,Rightarrow] \\
            \makecell{\text{ergodic along ideals} \\ \text{strongly ergodic} \\ \text{ergodic along finite index subgroups}} \arrow[d,Rightarrow] \\
            \text{ergodic}
        \end{tikzcd}}{\caption{Hierarchy of ergodicity properties for actions of rings of integers of number fields.}}
        \ffigbox{\begin{tikzcd}
            \text{FUNCTION FIELDS} \\
            \makecell{\text{ergodic along finite index subgroups} \\ \text{weakly mixing}} \arrow[d,Rightarrow] \\
            \text{strongly ergodic} \arrow[d,Rightarrow] \\
            \text{ergodic along ideals} \arrow[d,Rightarrow] \\
            \text{ergodic}
        \end{tikzcd}}{\caption{Hierarchy of ergodicity properties for actions of rings of integers of function fields.}}
        \end{floatrow}
    \end{figure}
\end{remark}

Our next main result is a general criterion for the Furstenberg--S\'{a}rk\"{o}zy theorem over rings of integers of global fields.
In order to state the result, we need a notion of density for subsets of the ring of integers of a global field.
Given a global field $K$ with ring of integers $\mO_K$ and a subset $E \subseteq \mO_K$, the \emph{upper Banach density} of $E$ is given by
\begin{equation*}
    d^*(E) = \sup_{(\Phi_N)} \limsup_{N \to \infty} \frac{|E \cap \Phi_N|}{|\Phi_N|},
\end{equation*}
where the supremum is taken over all F{\o}lner sequences $(\Phi_N)_{N \in \N}$ in $(\mO_K,+)$.

\begin{restatable}{theorem}{NecessarySufficient} \label{thm: global field Sarkozy}
    Let $K$ be a global field with ring of integers $\mO_K$.
    Let $P(x) \in \mO_K[x]$.
    The following are equivalent:
    \begin{enumerate}[(i)]
        \item for any nonzero ideal $\{0\} \ne I \le \mO_K$, there exists $n \in \mO_K$ such that $P(n) \in I$;
        \item for any measure-preserving system $\left( X, \mu, (T_n)_{n \in \mO_K} \right)$ and any $A \subseteq X$ with $\mu(A) > 0$, there exists $n \in \mO_K$, $n \ne 0$, such that
        \begin{equation*}
            \mu \left( A \cap T_{P(n)}^{-1}A \right) > 0;
        \end{equation*}
        \item for any measure-preserving system $\left( X, \mu, (T_n)_{n \in \mO_K} \right)$, any $A \subseteq X$ with $\mu(A) > 0$, and any $\eps > 0$, the set
        \begin{equation*}
            \left\{ n \in \mO_K : \mu \left( A \cap T_{P(n)}^{-1}A \right) > \mu(A)^2 - \eps \right\}
        \end{equation*}
        is syndetic;
        \item for any set $E \subseteq \mO_K$ with $d^*(E) > 0$, there exist distinct $x, y \in E$ such that $x - y = P(n)$ for some $n \in \mO_K$;
        \item for any finite coloring of $\mO_K$, there exist distinct $x, y$ of the same color such that $x - y = P(n)$ for some $n \in \mO_K$.
        \item for any finite coloring of $\mO_K$, there exist distinct $x, y, z$ all of the same color such that $x - y = P(z)$.
    \end{enumerate}
\end{restatable}

\begin{definition} \label{defn: intersective}
    A polynomial $P(x) \in \mO_K[x]$ is called \emph{intersective} if it satisfies the properties in Theorem \ref{thm: global field Sarkozy}.
\end{definition}

\begin{remark}
    Note that Theorem \ref{thm: global field Sarkozy} reduces to Theorem \ref{thm: necessary and sufficient FS} in the case $\mO_K = \Z$.
    The general notion of an intersective polynomial in a ring of integers of a number field was formulated in \cite{br}; see \cite[Theorem 1.17]{br} for a statement relating items (i) and (iii) in the number field context.
    In the global function field setting (when $K$ is a finite extensions of $\F(t)$ with $\F$ a finite field), substantial progress toward Theorem \ref{thm: global field Sarkozy} was made in \cite{bl}.
    However, the definition of ``intersective'' given in \cite{bl} is not sufficiently general to allow for the equivalences expressed in Theorem \ref{thm: global field Sarkozy}.
    In the remark following Theorem 9.5 in \cite{bl}, intersective polynomials over $\F_p[t]$ are defined as polynomials $P(x) \in (\F_p[t])[x]$ such that for any finite index subgroup $\Lambda \le (\F_p[t], +)$, there exists $m \in \F_p[t]$ such that $P(nm) \in \Lambda$ for every $n \in \Lambda$.
	The definition of intersective we have given here deals with a wider class of polynomials and is the correct notion to characterize the Furstenberg--S\'{a}rk\"{o}zy theorem in function fields, as demonstrated by Theorem \ref{thm: global field Sarkozy}.
	An example of an intersective polynomial that does not fit the condition in \cite{bl} is $P(x) = x+1$.
	Taking $\Lambda$ to be the subgroup $\Lambda = t\F_p[t]$ of index $p$, we have $P(nm) \equiv 1 \pmod{\Lambda}$ for every $n \in \Lambda, m \in \F_p[t]$, so the condition from \cite{bl} is not satisfied.
	However, $P(-1) = 0$, so $P$ is intersective (according to Definition \ref{defn: intersective}).

    In the special case $K = \F_q(t)$, Theorem \ref{thm: global field Sarkozy} also answers a question of L\^{e}, Liu, and Wooley (\cite[Question 1]{llw}), which asks if, for the set of values $\{P(n) : n \in \F_q[t]\}$ of a polynomial $P(x) \in (\F_q[t])[x]$, intersecting with every nonzero ideal is equivalent to intersecting with every subgroup of finite index.
    Indeed, $\{P(n) : n \in \F_q[t]\}$ intersecting with every nonzero ideal is the content of item (i), and this implies that $\{P(n) : n \in \F_q[t]\}$ intersects an arbitrary finite index subgroup $H \le \F_q[t]$ by considering the finite coloring of $\F_q[t]$ by cosets of $H$ and applying item (v).
\end{remark}

Another meaningful family of polynomials from the point of view of the limiting behavior of ergodic averages is the class of polynomials that are \emph{good for irrational equidistribution}.

\begin{definition} \label{defn: irrational equidistribution}
    Given a global field $K$ with ring of integers $\mO_K$, we say that $P(x) \in \mO_K[x]$ is \emph{good for irrational equidistribution} if for any irrational character $\chi : (\mO_K,+) \to S^1$, one has
    \begin{equation*}
            \UClim_{n \in \mO_K} \chi(P(n)) = 0.
    \end{equation*}    
\end{definition}

In the number field setting, every nonconstant polynomial is good for irrational equidistribution:

\begin{proposition} \label{prop: number field irrational equidistribution}
    Let $K$ be an algebraic number field with ring of integers $\mO_K$, and let $P(x) \in \mO_K[x]$ be a nonconstant polynomial.
    Then $P$ is good for irrational equidistribution.
\end{proposition}

\begin{proof}
    This essentially follows from Weyl's equidistribution theorem \cite{weyl} after some algebraic processing.

    Let $\chi : \mO_K \to S^1$ be an irrational character.
    Note that as an additive group, $(\mO_K,+)$ is isomorphic to $(\Z^d,+)$ for some $d \in \N$.
    Letting $b_1, \ldots, b_d \in \mO_K$ be an integral basis, we can therefore write
    \begin{equation*}
        \chi \left( \sum_{i=1}^d n_i b_i \right) = e^{2\pi i \sum_{i=1}^d n_i\alpha_i}
    \end{equation*}
    for some $(\alpha_1, \ldots, \alpha _d) \in \T^d$.
    By \cite[Lemma 2.1]{ab-roi}, a character is rational if and only if the corresponding element of the torus has rational coordinates.
    Thus, at least one of the coordinates $\alpha_1, \ldots, \alpha_d$ is irrational.
    Moreover, by \cite[Proposition 2.9]{ab-roi}, the polynomial $P$ satisfies
    \begin{equation*}
        P \left( \sum_{i=1}^d n_i b_i \right) = \sum_{i=1}^d P_i(n_1, \ldots, n_d) b_i,
    \end{equation*}
    where the polynomials $P_1, \ldots, P_d$ are algebraically independent over $\Q$.
    We thus have
    \begin{equation*}
        \chi \left( P \left( \sum_{i=1}^d n_i b_i \right) \right) = e^{2\pi i \sum_{i=1}^d P_i(n_1, \ldots, n_d)\alpha_i},
    \end{equation*}
    and the polynomial $\sum_{i=1}^d P_i(n_1, \ldots, n_d)\alpha_i$ has at least one irrational coefficient other than the constant term, so
    \begin{equation*}
        \UClim_{(n_1, \ldots, n_d) \in \Z^d} \chi \left( P \left( \sum_{i=1}^d n_i b_i \right) \right) = 0
    \end{equation*}
    by Weyl's equidistribution theorem \cite[Satz 20]{weyl} (see also \cite[Remark 3.10]{ab-roi}).
\end{proof}

In the global function field setting, a partial characterization (namely, for polynomials with integer coefficients) of when a polynomial is good for irrational equidistribution is provided by \cite[Theorem 1.15]{ab_Sarkozy}.
We include below some illustrative examples (and nonexamples) of polynomials that are good for irrational equidistribution; see \cite[Theorem 1.15 and Example 1.16]{ab_Sarkozy} for explanations.

\begin{example} \label{eg: irrational equidistribution}
    Let $K$ be a global function field with characteristic $\char(K) = p < \infty$.
    \begin{itemize}
        \item If $P(x) = \sum_{j=1}^N a_j x^{r_j}$ and $r_j$ is coprime to $p$ for every $j$, then $P$ is good for irrational equidistribution.
        \item The additive polynomial $P(x) = x^p$ is not good for irrational equidistribution.
        \item The polynomial $P(x) = x^{p^2} + x^{2p} + x$ is good for irrational equidistribution, while $Q(x) = x^{p^2} + x^{2p}$ is not.
        \item The polynomial $P(x) = x^p + x^2$ is good for irrational equidistribution, but $Q(x) = x^{2p} + x^2$ is not.
        However, $Q(x) + x$ is again good for irrational equidistribution.
    \end{itemize}
\end{example}

We give a characterization of polynomials that are good for irrational equidistribution in terms of the limiting behavior of corresponding polynomial ergodic averages in the following theorem, to be proved in Section \ref{sec:TE polynomial}:

\begin{theorem} \label{thm: TE + ud}
    Let $K$ be a global field with ring of integers $\mO_K$, and let $P(x) \in \mO_K[x]$ be a polynomial.
    The following are equivalent:
    \begin{enumerate}[(i)]
        \item $P$ is good for irrational equidistribution;
        \item for any measure-preserving system $(X, \mu, (T_n)_{n \in \mO_K})$ that is ergodic along ideals and any $f \in L^2(\mu)$,
            \begin{equation*}
                \UClim_{n \in \mO_K} T_{P(n)}f = \int_X f~d\mu.
            \end{equation*}
    \end{enumerate}
\end{theorem}

As we shall see in Theorem \ref{thm: equivalences} below, the improved behavior of ergodic averages along polynomials that are good for irrational equidistribution (converging to the integral rather than a more complicated projection) leads to stronger combinatorial statements than what can be obtained for arbitrary polynomials.


\section{Asymptotic total ergodicity} \label{sec:ATE}

The phenomena identified in the previous section can be detected already in an asymptotic form in the finitary setup of measure-preserving actions by rotations on finitely many points.
As a consequence, in Section \ref{sec:finitary applications}, we obtain new combinatorial results related to polynomial configurations in large subsets of finite rings and to quasirandomness of graphs associated to polynomials over finite rings.

As noted in Theorem \ref{thm: TE equivalences Z}, if a system is totally ergodic, then polynomial ergodic averages in the system converge to the ``correct limits.''
A rotation on finitely many points cannot be totally ergodic, but replacing the condition ``$T^k$ is ergodic for any $k \in \N$'' with the weaker condition ``$T^k$ is ergodic for any $k \le D$'' opens up the possibility of approximating the properties of totally ergodic systems (with the approximation improving as $D \to \infty$).
This is the basis for the notion of \emph{asymptotic total ergodicity} developed in \cite{bb1, bb2}.
In the following theorem, we collect results from \cite{bb1} revealing the main features of asymptotic total ergodicity in the context of modular rings.

\begin{theorem} \label{thm: bb}
    Let $(N_m)_{m \in \N}$ be an increasing sequence in $\N$.
    We denote by $T_m : \Z/N_m\Z \to \Z/N_m\Z$ the transformation $T_m(n) = n + 1$, and we let $\mu_m$ be the normalized counting measure on $\Z/N_m\Z$.
    The following are equivalent:
    \begin{enumerate}[(i)]
        \item the smallest prime factor of $N_m$ tends to $\infty$;
        \item for any $k \in \N$, there exists $M \in \N$ such that if $m \ge M$, then $T_m^k$ is ergodic;
        \item for any $k \in \N$,
            \begin{equation*}
                \lim_{m \to \infty} \sup_{A, B \subseteq \Z/N_m\Z} \left| \frac{1}{N_m} \sum_{n \in \Z/N_m\Z} \mu_m(A \cap T_m^{kn}B) - \mu_m(A) \mu_m(B)\right| = 0;
            \end{equation*}
        \item for any nonconstant polynomial $P(x) \in \Z[x]$,
            \begin{equation*}
                \lim_{m \to \infty} \sup_{A,B \subseteq \Z/N_m\Z} \left| \frac{1}{N_m} \sum_{n \in \Z/N_m\Z} \mu_m(A \cap T_m^{P(n)}B) - \mu_m(A) \mu_m(B)\right| = 0.
            \end{equation*}
    \end{enumerate}
\end{theorem}

\begin{proof}
    This is a reformulation of results appearing in \cite{bb1}.
    We give a brief explanation of how to deduce Theorem \ref{thm: bb} from the main results of \cite{bb1}.
    
    Let (iii') be the statement:
    \begin{enumerate}
        \item[(iii')] for any $k \in \N$,
            \begin{equation*}
                \lim_{m \to \infty} \sup_{A \subseteq \Z/N_m\Z} \left| \frac{1}{N_m} \sum_{n \in \Z/N_m\Z} \mu_m(A \cap T_m^{kn}A) - \mu_m(A)^2\right| = 0.
            \end{equation*}
    \end{enumerate}
    In \cite[Proposition 1.3]{bb1}, it is shown that (i) and (iii') are equivalent.

    The implication (i)$\implies$(iv) follows from \cite[Theorem 2.6]{bb1}, and the implications (iv)$\implies$(iii)$\implies$(iii') are trivial.

    Ergodicity of the transformation $T^k_m$ is equivalent to $k$ being coprime to $N_m$, so (ii) is also equivalent to (i).
\end{proof}

Items (ii) and (iii) in Theorem \ref{thm: bb} are asymptotic formulations of the notion of total ergodicity, while item (iv) connects to the Furstenberg--S\'{a}rk\"{o}zy theorem, since the quantity $\mu_m(A \cap T_m^{P(n)}B)$ is a normalized count of polynomial configurations $(x, x+P(n)) \in A \times B$.
Theorem \ref{thm: bb} thus shows that the relationship between total ergodicity and polynomial ergodic averages captured by Theorem \ref{thm: TE equivalences Z} is preserved in an asymptotic form.
An interesting point to note is that we do not need to make any assumptions about the polynomial $P$ in item (iv).
This is in contrast to the situation in $\Z$, where local obstructions must be overcome and polynomials must be intersective in order to satisfy the Furstenberg--S\'{a}rk\"{o}zy theorem (see Definition \ref{defn: intersective Z} and Theorem \ref{thm: necessary and sufficient FS}).

For $N \in \N$, define $\lpf(N)$ to be the least prime factor of $N$.
Motivated by the equivalences in Theorem \ref{thm: bb}, we say that a sequence of modular rings $(\Z/N_m\Z)_{m \in \N}$ is \emph{asymptotically totally ergodic} if $\lpf(N_m) \to \infty$.
In \cite{bb2}, it was shown that the notion of asymptotic total ergodicity and its connection with polynomial configurations as expressed in Theorem \ref{thm: bb} extends to other sequences of finite rings $(R_n)_{n \in \N}$ when $\lpf$ of the characteristic of $R_n$ tends to infinity.
This leaves unaddressed the phenomenon of asymptotic total ergodicity for sequences of rings where the characteristic is fixed.
In order to capture asymptotic total ergodicity in the low characteristic setting, we must revisit $\lpf$ from a more ring-theoretic point of view.
The dynamical results from the previous section (Theorems \ref{thm: TE} and \ref{thm: strongly ergodic}) provide some insight about the proper generalization of $\lpf$ to a wider class of finite rings.

We now consider the following generalization of rotations on finitely many points.
Let $K$ be a global field with ring of integers $\mO_K$.
Given a nontrivial ideal $\{0\} \ne I \lneq \mO_K$, the group $(\mO_K,+)$ acts on the quotient ring $\mO_K/I$ by rotations $T_n : x \mapsto x+n$ for $n \in \mO_K$.
(We remind the reader that every nonzero ideal in the ring of integers of a global field is of finite index, so the quotient ring $\mO_K/I$ is finite; see item (2) in Proposition \ref{prop: rings of integers}.)
This action is periodic along the ideal $I$, i.e. if $m \in I$, then $T_{mn+r} = T_r$ for every $n, r \in \mO_K$.
As demonstrated by Theorem \ref{thm: TE}, the obstacle to ``correct'' limiting behavior of polynomial ergodic averages is failure of ergodicity along ideals, so the action of $\mO_K$ on $\mO_K/I$ by rotations will not produce ``correct'' limiting values for polynomial averages.
However, in analogy with Theorem \ref{thm: bb}, one may guess that if the ideal $I$ and all of its prime factors have index at least $D$ in $\mO_K$, then polynomial averages approach the ``correct'' limiting value asymptotically as $D \to \infty$.
Indeed, we can prove the following generalization of Theorem \ref{thm: bb}.

\begin{theorem} \label{thm: bb low characteristic}
    Let $K$ be a global field with ring of integers $\mO_K$.
    Let $(I_m)_{m \in \N}$ be a sequence of nontrivial ideals of $\mO_K$.
    For $n \in \mO_K$, we denote by $T_{I_m,n} : \mO_K/I_m \to \mO_K/I_m$ the transformation $T_{I_m,n}(x) = x + n$, and we let $\mu_m$ be the normalized counting measure on $\mO_K/I_m$.
    The following are equivalent:
    \begin{enumerate}[(i)]
        \item the minimal index of a prime factor of $I_m$ tends to $\infty$;
        \item for any $k \in \mO_K \setminus \{0\}$, there exists $M \in \N$ such that if $m \ge M$, then $(T_{I_m,kn})_{n \in \mO_K}$ is ergodic;
        \item for any $k \in \mO_K \setminus \{0\}$,
            \begin{equation*}
                \lim_{m \to \infty} \sup_{A, B \subseteq \mO_K/I_m} \left| \frac{1}{[\mO_K:I_m]} \sum_{n \in \mO_K/I_m} \mu_m(A \cap T_{I_m,kn}B) - \mu_m(A) \mu_m(B)\right| = 0;
            \end{equation*}
        \item for any nonconstant polynomial $P(x) \in \mO_K[x]$ with $\deg{P} < \textup{char}(K)$,
            \begin{equation*}
                \lim_{m \to \infty} \sup_{A, B \subseteq \mO_K/I_m} \left| \frac{1}{[\mO_K:I_m]} \sum_{n \in \mO_K/I_m} \mu_m(A \cap T_{I_m,P(n)}B) - \mu_m(A) \mu_m(B)\right| = 0;
            \end{equation*}
        \item for any nonconstant polynomial $P(x) \in \mO_K[x]$ with $P(0) = 0$,
            \begin{equation*}
                \lim_{m \to \infty} \sup_{A, B \subseteq \mO_K/I_m} \left| \frac{1}{[\mO_K:I_m]} \sum_{n \in \mO_K/I_m} \mu_m(A \cap T_{I_m,P(n)}B) - \frac{1}{|H_{I_m,P}|} \sum_{n \in H_{I_m,P}} \mu_m(A \cap T_{I_m,n}B) \right| = 0,
            \end{equation*}
            where $H_{I_m,P}$ is the subgroup of $(\mO_K/I_m,+)$ generated by $\{P(n) : n \in \mO_K/I_m\}$.
    \end{enumerate}
\end{theorem}

Note that item (ii) in Theorem \ref{thm: bb low characteristic} expresses the property of ergodicity along ideals in an asymptotic form.
The equivalence of item (ii) with item (i) allows for a formulation purely in terms of ring theory without needing to check a dynamical statement.
Using the notation introduced in Section \ref{sec:setting} (see Definition \ref{defn: lpf}), item (i) is the statement $\lpf(I_m) \to \infty$.
Hence, just as in the case of modular rings, the quantity $\lpf$ captures the phenomenon of asymptotic total ergodicity.

Item (iv) generalizes the corresponding item from Theorem \ref{thm: bb}, and we once again encounter an interesting phenomenon that does not occur in the infinite ring $\mO_K$.
Namely, we have a Furstenberg--S\'{a}rk\"{o}zy-type result that makes no assumption about the constant term of $P$ (or, more generally, about intersectivity of $P$) and applies to arbitrary pairs of sets $(A,B)$.

Item (v) in Theorem \ref{thm: bb low characteristic} reveals a new phenomenon in the low characteristic setting.
When $K$ is a field of finite characteristic, there exist polynomials $P(x) \in \mO_K[x]$ with $\deg{P} \ge \textup{char}(K)$ for which the subgroup $H_{I_m,P}$ is a proper subgroup of $\mO_K/I_m$ for sequences of ideals $(I_m)_{m \in \N}$ satisying property (i), so the limiting property expressed in item (iv) may fail when $\deg{P} \ge \textup{char}(K)$.
The fact that, under the assumption of asymptotic total ergodicity expressed by item (i), we can always approximate a polynomial average by a linear average over a subgroup is nevertheless sufficient for establishing rather strong versions of the Furstenberg--S\'{a}rk\"{o}zy theorem in finite rings.
If the constant term is nonzero, say $P(0) = c$, then the average over the subgroup $H_{I_m,P}$ in item (v) must be replaced by an average over the coset $H_{I_m,P} + c$. In order to obtain meaningful combinatorial results, we need this coset to again be a subgroup (i.e., $c \in H_{I_m,P})$.
We can safely ignore the constant term in the case that $H_{I_m,P} = \mO_K/I_m$, and this leads to an interesting and nontrivial class of polynomials in the finite characteristic setting.
As we shall see, if a polynomial $P$ is good for irrational equidistribution (see Definition \ref{defn: irrational equidistribution}), then $H_{I,P} = \mO_K/I$ for all ideals with $\lpf(I)$ sufficiently large,\footnote{It is an interesting question whether the converse is also true. That is, if there exists $D \in \N$ such that $H_{I,P} = \mO_K/I$ for all ideals $I$ with $\lpf(I) \ge D$, must $P$ be good for irrational equidistribution? In Section \ref{sec: open problems}, we discuss some special cases where this is known to be true and indicate some of the difficulties that arise in the general case.} and we can use this property to obtain even stronger combinatorial results for polynomials that are good for irrational equidistribution.
In Section \ref{sec:finitary applications}, we give applications of Theorem \ref{thm: bb low characteristic} to problems in additive combinatorics and quasirandomness of graphs. \\

Our method of proof actually produces quantitative rates of convergence to $0$ in items (iii)--(v) of Theorem \ref{thm: bb low characteristic} in terms of the rate at which $\lpf(I_m)$ diverges to $\infty$.
An important parameter for this quantitative sharpening is the \emph{derivational degree} of a polynomial.
Let $P(x) \in \mO_K[x]$.
For $n \in \mO_K$, we define $\partial_n P$ to be the polynomial $\partial_nP(x) = P(x+n) - P(x)$ and write $\partial^k_{n_1, \dots, n_k}$ for the iterated differencing operation $\partial_{n_1} \circ \dots \circ \partial_{n_k}$.
The \emph{derivational degree} of $P(x) \in \mO_K[x]$ is the minimal $k \in \N$ such that $\partial^{k+1}_{n_1, \dots, n_{k+1}}P(x) = 0$ for every $n_1, \dots, n_{k+1}, x \in \mO_K$.
In the case $\textup{char}(K) = \infty$, the derivational degree of polynomials agrees with the usual degree.
When $\textup{char}(K) < \infty$ and $\deg{P} \ge \textup{char}(K)$, the derivational degree is always bounded above by $\deg{P}$ but can be much smaller.
For example, every additive polynomial $P(x) = \sum_{j=0}^N c_j x^{p^j}$ has derivational degree equal to 1.
Also, if $K$ has characteristic $p < \infty$, then the derivational degree of the monomial $x^d \in \mO_K[x]$ is equal to the sum of the digits in the base $p$ expansion of $d$ (see \cite[1.7]{blm}).

The next theorem provides a quantitative amplification of item (v) in Theorem \ref{thm: bb low characteristic}.
We express the result in terms of $L^2$ norms, as this form of the theorem is more readily proved using tools from Fourier analysis.
For a function $f : S \to \C$ defined on a finite set $S$, we define the $L^2$ norm by
\begin{equation*}
    \norm{L^2(S)}{f} = \left( \frac{1}{|S|} \sum_{s \in S} |f(s)|^2 \right)^{1/2}.
\end{equation*}
We will occasionally add a subscript to $L^2$ to indicate which symbol is being treated as the variable.
For example, $\norm{L^2_x(S)}{f(x)}$ means the same thing as $\norm{L^2(S)}{f}$.

\begin{restatable}{theorem}{QuantitativeTE} \label{thm: quantitative TE}
	Let $\mO_K$ be the ring of integers of a global field $K$ with characteristic $c \in \mathbb{P} \cup \{\infty\}$.
	Let $\{0\} \ne I \lneq \mO_K$ be a nontrivial ideal.
	Let $P(x) \in \mO_K[x]$ be a polynomial of degree $d$ and derivational degree $k$ with $P(0) = 0$.
    Let $H_{I,P}$ be the subgroup of $(\mO_K/I,+)$ generated by $\{P(n) : n \in \mO_K/I\}$.
    Then for any $f : \mO_K/I \to \C$,
    \begin{equation} \label{eq: L^2 estimate}
		\norm{L^2_x(\mO_K/I)}{\frac{1}{[\mO_K:I]}\sum_{n \in \mO_K/I} f(x+P(n)) - \frac{1}{|H_{I,P}|} \sum_{n \in H_{I,P}} f(x+n)} \ll_{d,k,c} \lpf(I)^{-1/2^{k-1}} \cdot \norm{L^2(\mO_K/I)}{f}.
	\end{equation}
\end{restatable}

Interpreting Theorem \ref{thm: quantitative TE} asymptotically, if $\lpf(I) \to \infty$, then the polynomial average
\begin{equation*}
    \E_{n \in \mO_K/I} f(x + P(n))
\end{equation*}
is approximated by the linear average
\begin{equation*}
    \E_{z \in H_{I,P}} f(x + z)
\end{equation*}
over the subgroup generated by the values of $P$.
In other words, the distribution of the polynomial $P(n)$ becomes increasingly uniform within the subgroup it generates.
Theorem \ref{thm: quantitative TE} (in its just-described asymptotic formulation) is analogous to the dynamical result expressed in Theorem \ref{thm: strongly ergodic}.
To make the analogy more explicit, we can define a space of functions $\mathcal{H}_{I,P} = \{g : \mO_K/I \to \C : g(x+P(n)) = g(x)~\text{for all}~x, n \in \mO_K/I\}$ and note that the function
\begin{equation*}
    x \mapsto \E_{x \in H_{I,P}} f(x+z)
\end{equation*}
is equal to the orthogonal projection of $f$ onto the space $\mathcal{H}_{I,P}$.

If $P(x) \in \mO_K[x]$ has a nonzero constant term, say $P(0) = c$, then we can apply Theorem \ref{thm: quantitative TE} directly to the polynomial $P(x) - c$.
Theorem \ref{thm: quantitative TE} then says that for the group $H$ generated by $\{P(x) - c: x \in \mO_K/I\}$, if $f : \mO_K/I \to \C$ is any function, then
\begin{equation*}
		\norm{L^2_x(\mO_K/I)}{\frac{1}{[\mO_K:I]}\sum_{n \in \mO_K/I} f(x+P(n)) - \frac{1}{|H|} \sum_{n \in H+c} f(x+n)} \ll_{d,k,c} \lpf(I)^{-1/2^{k-1}} \cdot \norm{L^2(\mO_K/I)}{f}.
\end{equation*}
Notice that the average in $n$ is now over the coset $H + c$ rather than the group $H$.
However, if $c \in H$, then $H + c = H$, so we recover an average over a subgroup.

In the proof of Theorem \ref{thm: quantitative TE}, appearing in Section \ref{sec:ATE proof}, we obtain a formula for the implicit constant in \eqref{eq: L^2 estimate}, but it is rather cumbersome.
Defining a function $B : \N \times (\N \cup \infty) \to \N$ by $B(d,p) = p^{2\floor{\log_p{d}}}$ for $p$ a prime number and $B(d,c) = 1$ otherwise, the constant is equal to
\begin{equation*}
    \left( B(d,c) \cdot (k-1) \right)^{1/2^{k-1}}.
\end{equation*}
Using the bounds $B(d,m) \le d^2$ and $k \le d$, we have the weaker but less cumbersome bound

\begin{equation*}
    \norm{L^2(\mO_K/I)}{\frac{1}{[\mO_K:I]}\sum_{n \in \mO_K/I} T_{I,P(n)} f - \frac{1}{|H_{I,P}|} \sum_{n \in H_{I,P}} T_{I,n} f} \le \left( \frac{d^3}{\lpf(I)} \right)^{1/2^{d-1}} \norm{L^2(\mO_K/I)}{f}.
\end{equation*}


\section{Finitary combinatorial applications} \label{sec:finitary applications}

We now discuss combinatorial applications of Theorem \ref{thm: quantitative TE} and the notion of asymptotic total ergodicity in the context of finite rings.
For these results, we will not make any direct reference to infinite rings but will instead state the results in terms of properties intrinsic to the finite rings themselves.
This is enabled in part by the observation that for an ideal $I \le \mO_K$, the quantity $\lpf(I)$ can be computed inside of the quotient ring $\mO_K/I$; see Proposition \ref{prop: lpf finite ring ideals}.
The class of rings to which our results apply is the class of finite principal ideal rings, which, by Proposition \ref{prop: good iff principal}, coincides with the family of rings that can be obtained as quotients of rings of integers of global fields.
As a reminder, the means of obtaining a finite principal ideal ring as a quotient is non-unique (see Example \ref{eg: F_4 realization}).
As a preliminary step toward obtaining combinatorial applications, we give a reformulation of Theorem \ref{thm: quantitative TE} in the language of finite principal ideal rings:

\begin{theorem} \label{thm: finite ring polynomial average}
	Let $R$ be a finite principal ideal ring, and let $P(x) \in R[x]$ be a polynomial of degree $d$ and derivational degree $k$.
    Assume $P(0) = 0$, and suppose $d < \lpf(R)$.
	Let $H \le (R, +)$ be the subgroup generated by $\{P(x) : x \in R\}$.
	Then for any $f : R \to \C$,
	\begin{equation} \label{eq: polynomial equidistribution}
		\norm{L^2_x(R)}{\E_{y \in R} f(x + P(y)) - \E_{z \in H} f(x + z)} \le \left( B(d,\textup{char}(R)) \frac{k-1}{\lpf(R)} \right)^{1/2^{k-1}} \norm{L^2(R)}{f}.
	\end{equation}
\end{theorem}

\begin{remark}
    As with Theorem \ref{thm: quantitative TE}, if $P(0) = c \ne 0$, then we must replace the average over $z \in H$ by an average $z \in H + c$, where $H$ is the subgroup generated by $\{P(x) - c : x \in R\}$.
\end{remark}

An informal interpretation of Theorem \ref{thm: finite ring polynomial average} is that the polynomial sequence $(P(y))_{y \in R}$ is \emph{approximately equidistributed} in the subgroup $H$ that it generates, with the approximation improving as $\lpf(R)$ tends to $\infty$.

As an application of Theorem \ref{thm: finite ring polynomial average}, we deduce a strong quantitative form of the Furstenberg--S\'{a}rk\"{o}zy theorem in finite principal ideal rings:

\begin{restatable}{corollary}{Sarkozy} \label{cor: intersective Sarkozy}
    Let $R$ be a finite principal ideal ring.
	Let $P(x) \in R[x]$ be a polynomial of degree $d$ and derivational degree $k$, and suppose $P$ has a root in $R$.
	If $A \subseteq R$ does not contain distinct $a, b \in A$ with $b - a = P(x)$ for some $x \in R$, then
	\begin{equation*}
		|A| \ll_{d,k} |R| \cdot \lpf(R)^{-1/2^{k-1}}.
	\end{equation*}
\end{restatable}

Corollary \ref{cor: intersective Sarkozy} was previous established in the special cases of modular rings $R = \Z/N\Z$ (see \cite{bb1}) and finite fields $R = \F_q$ (see, e.g. \cite{ls,ab_Sarkozy}).
For finite fields, the exponent $\frac{1}{2^{k-1}}$ can in fact be improved to $\frac{1}{2}$ using the Weil bound for exponential sums over finite fields (see \cite[Theorem 1.3]{ab_Sarkozy}), but the exponent in Corollary \ref{cor: intersective Sarkozy} matches the best known exponent for general modular rings.
The statement in Corollary \ref{cor: intersective Sarkozy} applies to a vast collection of finite rings, some of which are enumerated in Section \ref{sec:setting}, that are not encompassed in previous special cases.

An equivalent form of the conclusion of Corollary \ref{cor: intersective Sarkozy} is that any set $A$ with density at least
\begin{equation*}
    \frac{|A|}{|R|} \gg_{d,k} \lpf(R)^{-1/2^{k-1}}
\end{equation*}
contains two points $a, b \in A$ whose difference is a value of the polynomial $P$ (cf. Theorem \ref{thm: global field Sarkozy}(iii)).

Let us now make an observation related to the assumption that $P$ has a root in the ring $R$.
Suppose $K$ is a global field, and $\mO_K$ is its ring of integers.
If $P(x) \in \mO_K[x]$ is an intersective polynomial (see Definition \ref{defn: intersective}), then its reduction modulo an ideal has a root, so Corollary \ref{cor: intersective Sarkozy} is satisfied for $P$ in all of the quotient rings $R = \mO_K/I$.

We will show in Section \ref{sec:FS good rings} that the necessary and sufficient condition on a polynomial $P(x) \in R[x]$ to satisfy the conclusion of Corollary \ref{cor: intersective Sarkozy} is that the constant term $P(0)$ belongs to the subgroup $H \le (R,+)$ generated by $\{P(x) - P(0) : x \in R\}$.
This is weaker than $P$ having a root in $R$ (for example, a consequence of Theorem \ref{thm: equivalences} below is that if $\deg{P} < \char(R)$, then the group $H$ is equal to $R$, regardless of the behavior of the roots of $P$) but may be more difficult to interpret and to check in general.
Nevertheless, in the case that $R$ is a finite field and $P$ has integer coefficients, there is a simple algorithmic method to test whether or not $P(0)$ belongs to the group generated by $\{P(x) - P(0) : x \in R\}$; see \cite[Theorem 1.12]{ab_Sarkozy}. \\

Let $R$ be a finite principal ideal ring and $P(x) \in R[x]$ a nonconstant polynomial.
As we have just discussed, the group $H$ generated by $\{P(x) - P(0) : x \in R\}$ is the main algebraic object governing the Furstenberg--S\'{a}rk\"{o}zy phenomenon.
A natural question is to describe for which polynomials the group $H$ is equal to the entirety of $(R,+)$.
As a consequence of Theorem \ref{thm: finite ring polynomial average}, such polynomials have an approximate equidistribution property in the ring $R$, which allows for stronger combinatorial results.
Theorem \ref{thm: equivalences} below summarizes the stronger forms of the Furstenberg--S\'{a}rk\"{o}zy theorem that hold in the case that $H = R$ and elucidates the connection between this finitary equidistribution property and an infinitary equidsitribution property.
We will see afterwards that there is yet another interesting interpretation in terms of quasirandomness of graphs.

\begin{restatable}{theorem}{Equivalences} \label{thm: equivalences}
    Let $R$ be a finite principal ideal ring, and let $P(x) \in R[x]$ be a nonconstant polynomial.
	The following are equivalent:
	\begin{enumerate}[(i)]
		\item
			\begin{equation*}
				\sup_{\norm{L^2(R)}{f} = 1}{\norm{L^2_x(R)}{\E_{y \in R}{f(x+P(y))} - \E_{z \in R}{f(z)}}}
				 = o_{\lpf(R) \to \infty}(1);
			\end{equation*}
		\item	there exist $C_1, C_2, \gamma > 0$ such that if $\lpf(R) \ge C_1$, then
			\begin{equation*}
				\sup_{\norm{L^2(R)}{f} = 1}{\norm{L^2_x(R)}{\E_{y \in R}{f(x+P(y))} - \E_{z \in R}{f(z)}}}
				 \le C_2 \cdot \lpf(R)^{-\gamma};
			\end{equation*}
		\item there exists $C > 0$ such that if $\lpf(R) \ge C$, then the group generated by $\{P(x) - P(0) : x \in R\}$ is $(R, +)$.
		\item	for any $\delta > 0$, there exists $N > 0$ such that if $\lpf(R) \ge N$ and $A, B \subseteq R$ are subsets with $|A| |B| \ge \delta |R|^2$, then there exist $x, y \in R$ such that $x \in A$ and $x + P(y) \in B$;
		\item	there exist $C_1, C_2, \gamma > 0$ such that if $\lpf(R) \ge C_1$, then
			\begin{equation*}
				\bigg| \left| \left\{ (x,y) \in R \times R : x \in A, x + P(y) \in B \right\} \right| - |A| |B| \bigg|
				 \le C_2 |A|^{1/2} |B|^{1/2} |R| \cdot \lpf(R)^{-\gamma}.
			\end{equation*}
	\end{enumerate}
	Moreover, if $R = \mO_K/I$ and $P$ is the reduction mod $I$ of a polynomial $\tilde{P} \in \mO_K[x]$ that is good for irrational equidistribution, then each of the properties (i)-(v) holds.
\end{restatable}

\begin{remark}
    For some classes of polynomials $P$, we can show that the polynomial $\tilde{P}$ being good for irrational equidistribution is also necessary for items (i)-(v) to hold, but we do not know in general if the properties (i)-(v) can be fully characterized in terms of infinitary equidistribution properties.
    See Subsection \ref{sec: A,B equidistribution} for additional discussion.
\end{remark}

We may interpret item (v) in Theorem \ref{thm: equivalences} as a statement about Paley-type graphs over finite principal ideal rings.
The \emph{Paley graph} $\bfP_q$ for a prime power $q \equiv 1 \pmod{4}$ is the graph with vertex set $\F_q$ and edges $\{x,y\}$ if $y - x$ is a nonzero square in $\F_q$.
The Furstenberg--S\'{a}rk\"{o}zy theorem has been connected to Paley graphs from its inception.
If one modifies the construction of the Paley graph and considers instead the graph $G = (V,E)$ with vertex set $V = \{1, \ldots, N\}$ and edges $\{x,y\} \in E$ if $y-x$ is a nonzero square, then the Furstenberg--S\'{a}rk\"{o}zy theorem is equivalent to the statement that the largest independent set in $G$ has size $o(N)$.
Knowing well the properties of Paley graphs (for example, that the independence number is very small, bounded by $\sqrt{q}$), Lov\'{a}sz\footnote{Lov\'{a}sz's conjecture never appeared in print, but S\'{a}rk\"{o}zy credits Lov\'{a}sz as the source of the conjecture in the paper in which he proves it (see \cite{sarkozy}).} was the first to conjecture what later became the Furstenberg--S\'{a}rk\"{o}zy theorem.

Over time, the graph-theoretic aspects of the Furstenberg--S\'{a}rk\"{o}zy theorem have become more muted.
This is at least in part due to the fact that many interesting extensions of the Furstenberg--S\'{a}rk\"{o}zy theorem, when approached from an additive combinatorial or ergodic-theoretic point of view, are not easily formulated in graph-theoretic language.
For example, a dynamical reformulation of the Furstenberg--S\'{a}rk\"{o}zy theorem is the assertion that $\{n^2 : n\in \N\}$ is a set of recurrence in $\N$.
A natural line of inquiry suggested by this result is to determine for which other polynomials $P$ the set $\{P(n) : n \in \N\}$ is a set of recurrence (solved by \cite{km}, see Theorem \ref{thm: necessary and sufficient FS} and Remark \ref{rem: necessary and sufficient FS}).
Moving beyond polynomials, there has also been extensive activity in finding other natural sequences $(a_n)_{n \in \N}$ for which $\{a_n : n\in\N\}$ is a set of recurrence; see, e.g., \cite{sarkozy_prime,bfm,bks}.
Another direction of generalization is to produce Szemer\'{e}di-type extensions (that is, establishing multiple recurrence along polynomials or other sequences); see, e.g. \cite{bl_polynomial_sz}.

There are, however, many interesting previously unexplored aspects of the graph-theoretic side of the Furstenberg--S\'{a}rk\"{o}zy theorem over finite rings.
One of the aims of the present paper is to revisit the linkage between Furstenberg--S\'{a}rk\"{o}zy-type results and graph theory, with additional connections to algebra and ergodic theory.
We mainly focus on the graph-theoretic property of quasirandomness, which captures the phenomenon that graphs defined by simple deterministic means may nevertheless share many of the features of random graphs.

Randomness of graphs is epitomized by the Erd\H{o}s--R\'{e}nyi random graph model, which is in a sense the ``most random'' random graph.\footnote{Strictly speaking, any graph whatsoever can be considered as a random graph by utilizing a constant random variable. The Erd\H{o}s--R\'{e}nyi random graph fits with the intuitive idea of what ``random'' should mean for a graph, and the fact that it is the ``most random'' can be formalized in terms of the entropy of a random variable, for example (see \cite[Section III. A.]{pn}).}
Given $n \in \N$ and $p \in [0,1]$, the \emph{Erd\H{o}s--R\'{e}nyi random graph} $\mathbb{G}(n,p)$ is the random graph on $n$ vertices where each edge is included independently at random with probability $p$.

A family of graphs is \emph{quasirandom} if it satisfies the statements in the following theorem:

\begin{theorem}[cf. {\cite{cgw}}] \label{thm: CGW quasirandom}
    Let $G_n = (V_n, E_n)$ be a finite graph for each $n \in \N$, and let $p \in (0,1)$.
    Suppose the graphs $(G_n)_{n \in \N}$ have edge densities satisfying $\frac{|E_n|}{\binom{|V_n|}{2}} = p + o(1)$.
    The following are equivalent:
    \begin{enumerate}[(i)]
        \item for any $m, e \in \N$ and any graph $H$ with $m$ vertices labeled $\{u_1, \dots, u_m\}$ and $e$ edges, the number of induced\footnote{The graph $H$ is an \emph{induced subgraph} of a graph $G$ if there is a collection of $m$ vertices $\{v_1, \dots, v_m\}$ of $G$ such that $\{v_i,v_j\}$ is an edge in $G$ if and only if $\{u_i,u_j\}$ is an edge in $H$.} copies of $H$ in $G_n$, $n \in \N$, satisfies
            \begin{equation*}
                \left| \{(v_1, \dots, v_m) \in V_n^m : \{v_i,v_j\} \in E_n \iff \{u_i,u_j\} \in E(H)\} \right| = (1 + o_{n \to \infty}(1)) p^e (1-p)^{\binom{m}{2}-e}|V_n|^m;
            \end{equation*}
        \item for any $n \in \N$ and any subsets $A, B \subseteq V_n$,
            \begin{equation*}
                \left| \{(a,b) \in A \times B : \{a,b\} \in E_n\} \right| = p|A| |B| + o_{n \to \infty}(|V_n|^2);
            \end{equation*}
        \item letting $\lambda_{n,1}, \dots, \lambda_{n,k_n}$ be the eigenvalues of the adjacency matrix of $G_n$ with $|\lambda_{n,1}| \ge \dots \ge |\lambda_{n,k_n}|$, one has $|\lambda_{n,2}| = o_{n \to \infty}(|V_n|)$.
    \end{enumerate}
\end{theorem}

\begin{remark}
    Item (i) expresses a natural property of the Erd\H{o}s--R\'{e}nyi random graph.
    Fix a graph $H$ with $m$ vertices and $e$ edges.
    Suppose $n \ge m$.
    Then selecting $m$ vertices $(v_1, \dots, v_m)$ at random from $\mathbb{G}(n,p)$, the probability that the induced subgraph formed by $(v_1, \dots, v_m)$ is isomorphic to the graph $H$ is exactly $p^e (1-p)^{\binom{m}{2}-e}$, since each of the $e$ edges from $H$ has a corresponding edge in $\mathbb{G}(n,p)$ with probability $p$ and each of the $\binom{m}{2}-e$ non-edges from $H$ corresponds to a non-edge in $\mathbb{G}(n,p)$ with probability $1-p$.

    Item (ii) similarly expresses a randomness property, stating that the number of edges between any two sets is roughly (up to lower order terms) the expected number in a random graph with the same edge density $p$.

    It is not obvious \emph{a priori} that the two randomness properties (i) and (ii) should be equivalent, and even less so that the algebraic criterion (iii) captures this same randomness phenomenon.
    This is part of the richness of quasirandomness as a notion in graph theory, and we will see that condition (iii) is an especially useful technical tool for establishing quasirandomness of families of Cayley graphs and other graphs of algebraic origin.
\end{remark}

The definition of quasirandomness was formally introduced in \cite{cgw}, where it was also shown that Paley graphs are an example of a quasirandom family of graphs (see \cite[Section 5]{cgw}).
Related randomness properties of Paley graphs were observed in earlier work (see, e.g., \cite[Theorem 3]{bt}).

Extending beyond the classical Paley graphs defined above, we consider quasirandomness of graphs generated by values of polynomials over finite rings.
Let $R$ be a finite commutative ring and $d \in \N$.
We let $\bfP(R,d)$ be the graph with vertex set $R$ and edges $\{x,y\}$ if $x-y$ or $y-x$ is a $d$th power in $R$.
Note that for a prime power $q \equiv 1 \pmod{4}$, the graph $\bfP(\F_q,2)$ is the Paley graph $\bfP_q$.

\begin{remark}
    In order to stay in the realm of undirected graphs, we have defined $\bfP(R,d)$ to be the Cayley graph for $(R,+)$ generated by $\{\pm x^d : x \in R\} \setminus \{0\}$.
    Traditionally, Paley graphs are defined only for prime powers $q \equiv 1 \pmod{4}$ so that $-1$ is a square and consequently the set of squares $\{x^2 : x \in \F_q \setminus \{0\}\}$ is a symmetric set generating $\bfP_q$ as a Cayley graph of the group $(\F_q,+)$.
    If one takes the set of squares as the generators for a Cayley graph of $(\F_q,+)$ with $q \equiv 3 \pmod{4}$, the result is a directed graph, sometimes called the \emph{Paley digraph} (see \cite[Section 9.7]{jones} for a discussion and historical remarks about Paley digraphs).
    Because half of the elements of $\F_q \setminus \{0\}$ are squares, the set $\{\pm x^2 : x \in \F_q \setminus \{0\}\}$ is all of $\F_q \setminus \{0\}$ when $q \equiv 3 \pmod{4}$, so the undirected version of the Paley graph is the complete graph on $q$ vertices.
    Complete graphs form a quasirandom family, but admittedly not a very interesting one.
    By considering more general rings, we can obtain nontrivial graphs $\bfP(R,2)$, and moving to higher degree produces graphs $\bfP(R,d)$ over finite rings that are nontrivial even for $R = \F_q$ with $q \equiv 3 \pmod{4}$.
\end{remark}

It is important in Theorem \ref{thm: CGW quasirandom} that the sequence of graphs is \emph{dense}, i.e. that the number of edges $|E_n|$ grows as $c|V_n|^2$ for some constant $c$ ($ = \frac{p}{2}$).
In our general setup, a sequence of graphs $\mathbf{P}(R_n,d)$ may have edge densities tending to 0 as $n \to \infty$.
We therefore consider notions of \emph{sparse quasirandomness}, as introduced in \cite{cg}.
We will say that a sequence of graphs $G_n = (V_n, E_n)$ is \emph{(sparse) quasirandom} if $G_n$ has edge density $p_n = \frac{|E_n|}{\binom{|V_n|}{2}}$ and for every $A, B \subseteq V_n$,
\begin{equation*}
    \left| \{(a,b) \in A \times B : \{a,b\} \in E_n\} \right| = p_n |A| |B| + o_{n \to \infty}(p_n|V_n|^2).
\end{equation*}

The following theorem, to be proved in Section \ref{sec:quasirandomness}, establishes a two-way connection between quasirandomness of a sequence of Paley-type graphs and asymptotic total ergodicity of the underlying sequence of rings.

\begin{restatable}{theorem}{Quasirandom} \label{thm: quasirandom iff aTE}
	Let $(R_n)_{n \in \N}$ be a sequence of finite principal ideal rings such that either $\textup{char}(R_n) \to \infty$ or $\textup{char}(R_n)$ is constant.\footnote{This assumption is to ensure that there exists $d \ge 2$ such that $\char(R_n) \nmid d$ for all but finitely many $n \in \N$.
    We note that every sequence of rings has a subsequence where either the characteristic is constant or diverges to infinity.}
	Then:
	\begin{enumerate}[(1)]
        \item If $\frac{|R_n^{\times}|}{|R_n|} \to 1$, then for any $d \ge 2$ such that $\char(R_n) \nmid d$ for all but finitely many $n \in \N$, the family of graphs $\left( \bfP(R_n,d) \right)_{n \in \N}$ is quasirandom.
        \item If $d \ge 2$ such that $\char(R_n) \nmid d$ for all but finitely many $n \in \N$ and $\bfP(R_n, d)$ is quasirandom, then
        \begin{equation*}
            \min \left\{ [R_n, \p] : \p~\text{is a prime ideal of}~R_n~\text{and}~\{\pm x^d : x \in R_n/\p\} \ne R_n/\p \right\} \to \infty.
        \end{equation*}
	\end{enumerate}
\end{restatable}

The conditions in Theorem \ref{thm: quasirandom iff aTE} are related but not identical to the notion of asymptotic total ergodicity captured by the quantity $\lpf(R_n)$.
Expressing a finite principal ideal ring $R$ as a quotient $\mO_K/I$ for some ring of integers $\mO_K$ of a global field $K$ and some ideal $I$, we can factor $I = \p_1^{e_1} \ldots \p_r^{e_r}$ and see that
\begin{equation*}
    |R^{\times}| = |R| \cdot \prod_{i=1}^r \left( 1 - \frac{1}{[R:\p_i]} \right) \leq |R| \cdot \left( 1 - \frac{1}{\lpf(R)} \right).
\end{equation*}
Hence, if $\frac{|R_n^{\times}|}{|R_n|} \to 1$, then $\lpf(R_n) \to \infty$.
However, the converse is not true.
If the number of prime factors $r$ grows very rapidly, one can construct a sequence for which $\lpf(R_n) \to \infty$ but $\frac{|R_n^{\times}|}{|R_n|} \not\to 1$.
We give concrete examples in Example \ref{eg: Paley-type graphs} below. 
The quantity
\begin{equation*}
    \min \left\{ [R, \p] : \p~\text{is a prime ideal of}~R~\text{and}~\{\pm x^d : x \in R/\p\} \ne R/\p \right\}
\end{equation*}
looks almost like the definition of $\lpf(R)$, the difference being that we now disregard prime factors $\p$ for which every element of $R/\p$ is either a $d$th power or its negative.
The reason for discarding such prime ideals is that the corresponding graph $\bfP(R/\p,d)$ is a complete graph, which is trivially quasirandom regardless of the size of $R/\p$.

The condition $\textup{char}(R_n) \nmid d$ is to ensure that the graph $\bfP(R_n,d)$ is connected.
If $\char(R_n) = p$ and $p \mid d$, say $d = p^r k$ for some $r \in \N$ and $p \nmid k$, then $\bfP(R_n,d)$ may be disconnected, but each of its connected components is isomorphic to $\bfP(R'_n,k)$, where $R'_n$ is the ring $R'_n = \{x^{p^r} : x \in R_n\}$.
Determining whether the connected components are quasirandom thus reduces to a problem that can be addressed by Theorem \ref{thm: quasirandom iff aTE}.

The following examples illustrate the content of Theorem \ref{thm: quasirandom iff aTE} in natural special cases.
To our knowledge, outside of the case that the rings $R_n$ are finite fields, all of the examples of quasirandom sequences produced by item (1) in Theorem \ref{thm: quasirandom iff aTE} are new.

\begin{example} \label{eg: Paley-type graphs}~
    \begin{itemize}
        \item For a finite field $\F_q$, the multiplicative group has size $|\F_q^{\times}| = q - 1$, so the family of finite fields satisfies $\frac{|\F_q^{\times}|}{|\F_q|} = \frac{q-1}{q} \to 1$ as $q \to \infty$.
        Item (1) in Theorem \ref{thm: quasirandom iff aTE} therefore recovers the known fact that the sequence of Paley graphs (and its higher degree generalizations) is quasirandom (see \cite[Section 5]{cgw}).
        \item For a modular ring $R = \Z/N\Z$, the cardinality $|R^{\times}|$ is given by Euler's totient function $\phi(N)$.
        Hence, if $R_n = \Z/N_n\Z$, we have $\frac{|R_n^{\times}|}{|R_n|} \to 1$ if and only if $\frac{\phi(N_n)}{N_n} \to 1$.
        Factoring $N_n = p_{n,1}^{e_{n,1}} \cdot \ldots \cdot p_{n,k_n}^{e_{n,k_n}}$ and using $\frac{\phi(N_n)}{N_n} = \prod_{i=1}^{k_n} \left( 1 - \frac{1}{p_{n,i}} \right)$, we see that a sufficient condition for $(\bfP(\Z/N_n\Z,d))_{n \in \N}$ to form a quasirandom sequence is that $\sum_{i=1}^{k_n} \frac{1}{p_{n,i}} \to 0$ as $n \to \infty$.
        On the other hand, a necessary condition is that for each prime factor $p_{n,i}$, either $p_{n,i}$ is large (growing to infinity as $n \to \infty$) or every element is a $d$th power or its negative mod $p_{n,i}$.
        For instance, if $d = 2$, then we require that $\min\{p_{n,i} : 1 \le i \le k_n, p_{n,i} \equiv 1 \pmod{4}\} \to \infty$.
        \item The above analysis of modular rings generalizes with only minor modifications to arbitrary finite principal ideal rings.
        Given $R = \mO_K/I$, we can factor $I = \p_1^{e_1} \cdot \ldots \cdot \p_k^{e_k}$ and compute
        \begin{equation*}
            \frac{|R^{\times}|}{|R|} = \prod_{i=1}^k \left( 1 - \frac{1}{[\mO_K:\p_i]} \right).
        \end{equation*}
        By Theorem \ref{thm: quasirandom iff aTE}, if $\sum_{i=1}^k \frac{1}{[\mO_K : \p_i]}$ is sufficiently small and $\char(R) \nmid d$, then $\bfP(R,d)$ is highly quasirandom.
    \end{itemize}
\end{example}

The estimates provided by item (v) in Theorem \ref{thm: equivalences} suggest that a version of Theorem \ref{thm: quasirandom iff aTE} may hold for graphs generated by values of more general polynomials $P(x) \in R[x]$ in place of $x^d$.
However, when translating the results into a graph-theoretic setting, it is important that all of the values of the polynomial $P(x)$ occur with equal frequency, and this is often not the case for general polynomials.
This algebraic issue is also the reason that we need information about the multiplicative subgroup in item (1) of Theorem \ref{thm: quasirandom iff aTE} and information about $\lpf$ is insufficient on its own.

Theorem \ref{thm: quasirandom iff aTE} as stated establishes only a qualitative relationship between asymptotic total ergodicity and quasirandomness.
It turns out that this relationship can be made quantitative.
The quantitative statement is given in Theorem \ref{thm: quasirandom}, and we will deduce Theorem \ref{thm: quasirandom iff aTE} from Theorem \ref{thm: quasirandom} in Section \ref{sec:quasirandomness}.


\section{Total ergodicity and polynomial ergodic averages} \label{sec:TE polynomial}

The goal of this section is to prove Theorems \ref{thm: TE}, \ref{thm: strongly ergodic}, and \ref{thm: TE + ud}, connecting the behavior of polynomial ergodic averages with the property of total ergodicity.
The proof strategy for these theorems is broadly the same in the number field and global function field settings, but there are important differences in some technical arguments arising from the distinction between finite and infinite characteristic\footnote{We remind the reader that in this paper we adopt the convention that the characteristic is equal to the additive order of the multiplicative identity to enable more convenient formulations of our theorems. The term ``infinite characteristic'' therefore refers to what is more often referred to as ``zero characteristic''.}.
As we shall see, the main ingredients in the proofs of Theorems \ref{thm: TE} and \ref{thm: strongly ergodic} are appropriate versions of Weyl's equidistribution theorem for polynomials over rings of integers of global fields.
The number field version of Weyl's equidistribution theorem that we require was already proved in Proposition \ref{prop: number field irrational equidistribution} above.
In the function field setting, we will use two different equidistribution theorems, one dealing with low degree (relative to the characteristic) polynomials and the other with high degree polynomials.
While both of these theorems are also valid in number fields (and follow directly from Proposition \ref{prop: number field irrational equidistribution}), some of the assumptions are not as meaningful in infinite characteristic, so one should keep in mind global function fields when interpreting the statements of Theorems \ref{thm: Weyl rings of integers low degree} and \ref{thm: Weyl rings of integers} below.

The first equidistribution theorems deals with polynomials of low degree and closely resembles the classical polynomial equidistribution theorem of Weyl:

\begin{theorem} \label{thm: Weyl rings of integers low degree}
    Let $K$ be a global field with ring of integers $\mO_K$.
    Let $\chi : \mO_K \to S^1$ be an irrational character, and let $P(x) \in \mO_K[x]$ be a nonconstant polynomial with $\deg{P} < \char(K)$.
    Then $\UClim_{n \in \mO_K} \chi(P(n)) = 0$.
\end{theorem}

The second equidistribution theorem deals with polynomials of high degree, and this introduces new difficulties, as polynomials whose degree exceeds the characteristic have more intricate distributional behavior.
Indeed, Example \ref{eg: irrational equidistribution} above shows that the conclusion of Theorem \ref{thm: Weyl rings of integers low degree} may fail for high degree polynomials.
However, we are able to obtain an appropriate version of Weyl's equidistribution theorem for high degree polynomials by replacing the assumption of irrationality by strong irrationality:

\begin{theorem} \label{thm: Weyl rings of integers}
    Let $K$ be a global field with ring of integers $\mO_K$.
    Let $\chi : \mO_K \to S^1$ be a strongly irrational character, and let $P(x) \in \mO_K[x]$ with $P(0) = 0$.
    Then either
    \begin{itemize}
        \item $\chi(P(n)) = 1$ for every $n \in \mO_K$, or
        \item $\UClim_{n \in \mO_K} \chi(P(n)) = 0$.
    \end{itemize}
\end{theorem}

We will use a similar argument to the one appearing in the proof of Proposition \ref{prop: number field irrational equidistribution} to prove the function field case of Theorems \ref{thm: Weyl rings of integers low degree} and \ref{thm: Weyl rings of integers}.
The argument makes use of the algebraic facts contained in the next three lemmas.

The first lemma is classical, asserting the existence of integral bases for rings of integers of function fields:

\begin{lemma}[cf. {\cite[Chapter IV, Theorem 3]{chevalley}}]
    Let $K$ be a degree $d$ extension of the field of rational functions $\F_q(t)$ over a finite field $\F_q$, and let $\mO_K$ be the ring of integers in $K$.
    Then there exist elements $b_1, \dots, b_d \in \mO_K$ such that
    \begin{equation*}
        \mO_K = \left\{ \sum_{j=1}^d n_j b_j : n_j \in \F_q[t] \right\}.
    \end{equation*}
\end{lemma}

The integral basis $\{b_1, \dots, b_d\}$ corresponds to a group isomorphism $\mO_K \cong \F_q[t]^d$.
We can use this isomorphism to describe the dual group $\hat{\mO}_K$.
Let $\F_q((t^{-1}))$ be the field of formal Laurent series $\sum_{k=-\infty}^N a_k t^k$ with coefficients $a_k \in \F_q$.
Then $\hat{\F_q[t]}$ is isomorphic to the quotient group $\F_q((t^{-1}))/\F_q[t]$, and we will use the suggestive notation $\mathbf{T}$ for the space $\F_q((t^{-1}))/\F_q[t]$, indicating its analogous role to the circle group $\T$.
To make the isomorphism more explicit, every character $\chi : \F_q[t] \to S^1$ is uniquely representable as $\chi(n) = e(n\alpha)$ for some $\alpha \in \mathbf{T}$, where $e : \F_q[t] \to S^1$ is the map defined by
\begin{equation*}
    e \left( \sum_{k=-\infty}^N a_k t^k \right) = \textup{Tr}_{\F_q/\F_p}(a_{-1}).
\end{equation*}
Hence, characters on $\mO_K$ are represented by $d$-tuples $\bm{\alpha} = (\alpha_1, \dots, \alpha_d) \in \mathbf{T}^d$ via
\begin{equation*}
    \sum_{j=1}^d n_j b_j \mapsto e \left( \sum_{j=1}^d n_j \alpha_j \right).
\end{equation*}
The following lemma relates properties of a character $\chi$ on $\mO_K$ with corresponding properties of its representative tuple $\bm{\alpha} \in \mathbf{T}^d$.
For the relevant definitions, see Section \ref{sec:FS and TE}.

\begin{lemma} \label{lem: rational character rational element}
    Let $K$ be a degree $d$ extension of the field of rational functions $\F_q(t)$ over a finite field $\F_q$, and let $\mO_K$ be the ring of integers in $K$.
    Let $\{b_1, \dots, b_d\} \subseteq \mO_K$ be an integral basis.
    Let $\chi : \mO_K \to S^1$ be a group character, and express $\chi$ in coordinates as
    \begin{equation*}
        \chi \left( \sum_{j=1}^d n_j b_j \right) = e \left( \sum_{j=1}^d n_j \alpha_j \right)
    \end{equation*}
    for some $\bm{\alpha} = (\alpha_1, \dots, \alpha_d) \in \mathbf{T}^d$.
    \begin{enumerate}[(1)]
        \item $\chi$ is rational if and only if $\bm{\alpha} \in \F_q(t)^d$.
        \item $\chi$ is weakly rational if and only if there exists $N \ge 0$ and coefficients $(C_{k,l})_{0 \le k \le N, 1 \le l \le d}$ in $\F_q[t]$ such that
        \begin{equation*}
            n \mapsto e \left( \sum_{k=0}^N \left( \sum_{l=1}^d C_{k,l} \alpha_l \right) n^{p^k} \right)
        \end{equation*}
        is a nontrivial rational character on $\F_q[t]$.
    \end{enumerate}
\end{lemma}

\begin{proof}
   (1)
   Suppose $\bm{\alpha} \in F_q(t)^d$.
   Then there exists $m \in \F_q[t]$ such that $m\bm{\alpha} \in \F_q[t]^d$.
   Let $I$ be the set of elements $\sum_{j=1}^d n_j b_j \in \mO_K$ with $m \mid n_j$ for every $j \in \{1, \ldots, d\}$.
   It is easy to see that $I$ is an ideal in $\mO_K$.
   Moreover, if $n = \sum_{j=1}^d n_j b_j \in I$, then $\sum_{j=1}^d n_j \alpha_j \in \F_q[t]$, so $\chi(n) = 1$.
   Thus, $\chi$ is periodic along the ideal $I$.

   Conversely, suppose that $\chi$ is rational.
   Let $m \in \mO_K$ such that $\chi(mn) = 1$ for every $n \in \mO_K$.
   We can write $\frac{1}{m} = \sum_{j=1}^d a_j b_j$ for some coefficients $a_j \in \F_q((t))$.
   Taking $D \in \F_q[t]$ so that $Da_j \in \F_q[t]$ for every $j$, we have $\frac{D}{m} \in \mO_K$.
   Therefore, for every $n \in \F_q[t]$,
   \begin{equation*}
       e(n D\alpha_j) = \chi(nDb_j) = \chi \left( m \cdot n\frac{D}{m}b_j \right) = 1.
   \end{equation*}
   That is, the character $n \mapsto e(nD\alpha_j)$ is the trivial character on $\F_q[t]$, so $D\alpha_j \in \F_q[t]$.
   Hence, $\bm{\alpha} \in \frac{1}{D} \F_q[t]^d \subseteq \F_q(t)^d$. \\
   
   (2)
   Let $\eta(x) = \sum_{k=0}^N c_k x^{p^k} \in \mO_K[x]$ be an additive polynomial.
   We want a formula for evaluating the character $\chi \circ \eta$.
   Write $c_k = \sum_{i=1}^d c_{k,i} b_i$.
   Then
   \begin{equation*}
       \eta \left( \sum_{j=1}^d n_j b_j \right) = \sum_{k=0}^N \left( \sum_{i=1}^d c_{k,i} b_i \right) \left( \sum_{j=1}^d n_j b_j \right)^{p^k}
       = \sum_{k=0}^N \left( \sum_{i=1}^d c_{k,i} b_i \right) \sum_{j=1}^d n_j^{p^k} b_j^{p^k}.
   \end{equation*}
   For $i, j \in \{1, \ldots, d\}$ and $k \ge 0$, we may write
   \begin{equation*}
       b_i b_j^{p^k} = \sum_{l=1}^d a_{i,j,k,l} b_l
   \end{equation*}
   for some coefficients $a_{i,j,k,l} \in \F_q[t]$.
   Hence,
   \begin{equation*}
       \eta \left( \sum_{j=1}^d n_j b_j \right) = \sum_{l=1}^d \left( \sum_{i,j=1}^d \sum_{k=0}^N a_{i,j,k,l} c_{k,i} n_j^{p^k} \right) b_l.
   \end{equation*}
   Therefore,
   \begin{equation*}
       (\chi \circ \eta) \left( \sum_{j=1}^d n_j b_j \right) = e \left( \sum_{l=1}^d \left( \sum_{i,j=1}^d \sum_{k=0}^N a_{i,j,k,l} c_{k,i} n_j^{p^k} \right) \alpha_l \right)
       = e \left( \sum_{j=1}^d \sum_{k=0}^N \left( \sum_{i,l=1}^d a_{i,j,k,l} c_{k,i} \alpha_l \right) n_j^{p^k} \right).
   \end{equation*}
   Let $d_{j,k,l} = \sum_{i=1}^d a_{i,j,k,l} c_{k,i} \in \F_q[t]$, and let $\beta_{j,k} = \sum_{l=1}^d d_{j,k,l} \alpha_l$.
   For each $j \in \{1, \ldots, d\}$, the map
   \begin{equation*}
       n \mapsto e\left( \sum_{k=0}^N \beta_{j,k} n^{p^k} \right)
   \end{equation*}
   is a group character on $\F_q[t]$, so there exists $\gamma_j$ such that
   \begin{equation*}
       e\left( \sum_{k=0}^N \beta_{j,k} n^{p^k} \right) = e(\gamma_j n)
   \end{equation*}
   for every $n \in \F_q[t]$.
   Thus,
   \begin{equation*}
       (\chi \circ \eta) \left( \sum_{j=1}^d n_j b_j \right) = e \left( \sum_{j=1}^d n_j \gamma_j \right).
   \end{equation*}
   By (1), $\chi \circ \eta$ is a nontrivial rational character if and only if $\bm{\gamma} \in \F_q(t)^d \setminus \{\bm{0}\}$.
   Taking $C_{k,l} = d_{j,k,l}$ for $j$ with $\gamma_j \ne 0$ proves (2).
\end{proof}

\begin{lemma} \label{lem: algebraically independent}
    Let $K$ be a degree $d$ extension of the field of rational functions $\F_q(t)$ over a finite field $\F_q$, and let $\mO_K$ be the ring of integers in $K$.
    Let $\{b_1, \dots, b_d\} \subseteq \mO_K$ be an integral basis.
    Let $P(x) \in \mO_K[x]$ be a nonconstant polynomial, and assume that $P - P(0)$ is not additive.
    Then expressing $P$ in coordinates as
    \begin{equation}
         P \left( \sum_{j=1}^d n_j b_j \right) = \sum_{j=1}^d P_j(n_1, \dots, n_d) b_j,
    \end{equation}
    the polynomials $P_1, \dots, P_d$ are algebraically independent over $\F_q(t)$.
\end{lemma}

\begin{proof}
    By the Jacobi criterion for algebraic independence\footnote{The Jacobi criterion is often stated only in infinite characteristic, where it is both a necessary and sufficient condition for algebraic independence. The finite characteristic context is more subtle, but having full rank Jacobian is nevertheless sufficient for algebraic independence (see \cite{speyer}).}, it suffices to prove that the Jacobian matrix
    \begin{equation*}
        J = \left( \begin{array}{ccc}
            \frac{\partial P_1}{\partial n_1} & \ldots & \frac{\partial P_d}{\partial n_1} \\
            \vdots & \ddots & \vdots \\
             \frac{\partial P_1}{\partial n_d} & \ldots & \frac{\partial P_d}{\partial n_d}
        \end{array} \right)
    \end{equation*}
    has full rank.
    Note that the $j$th row of the matrix is given by $\frac{\partial P}{\partial n_j} = P' \cdot b_j$, and the elements $b_1, \ldots, b_d$ are linearly independent over $\mO_K$.
    Therefore $J$ has full rank if and only if $P' \ne 0$.
    But $P' \ne 0$ is equivalent to $P - P(0)$ not being additive.
\end{proof}

\begin{remark}
    Lemma \ref{lem: algebraically independent} may fail for additive polynomials, as shown by the following example.
    Let $K = \F_2(t)(\omega)$, where $\omega$ satisfies $\omega^2 = t^3 + t + 1$.
    (That is, $K$ is the field of rational functions on the elliptic curve $y^2 = x^3+x+1$.)
    Then $\mO_K$ has an integral basis $\{1, \omega\}$.
    Writing the additive polynomial $P(x) = x^2$ in this basis, we have
    \begin{equation*}
        P(n_1 + n_2\omega) = (n_1+n_2\omega)^2 = n_1^2 + n_2^2(t^3+t+1),
    \end{equation*}
    so $P_1(n_1,n_2) = n_1^2 + n_2^2(t^3+t+1)$ and $P_2(n_1,n_2) = 0$.
    A similar computation shows that the coordinates of an arbitrary additive polynomial over $\mO_K$ will be linearly dependent over $\F_2(t)$.
\end{remark}

The inapplicability of Lemma \ref{lem: algebraically independent} for some additive polynomials is no hindrance to proving Theorems \ref{thm: Weyl rings of integers low degree} and \ref{thm: Weyl rings of integers}.
First, in the context of Theorem \ref{thm: Weyl rings of integers low degree}, if $P(x) \in \mO_K[x]$ is nonconstant and $\deg{P} < \char(K)$, then $P$ cannot be additive.
Second, in the context of Theorem \ref{thm: Weyl rings of integers}, if $P(x) \in \mO_K[x]$ is an additive polynomial, then $\chi \circ P$ is a group character on $\mO_K$, so Theorem \ref{thm: Weyl rings of integers} in this case reduces to the classical fact that nontrival characters have uniform Ces\`{a}ro averages equal to 0.
We may therefore assume that $P$ is not an additive polynomial, so combining the previous three lemmas, the function field case of Theorems \ref{thm: Weyl rings of integers low degree} and \ref{thm: Weyl rings of integers} is reduced to the following proposition:

\begin{proposition} \label{prop: equidistribution theorems, exponential sum form}
    Let $\F_q$ be a finite field of characteristic $p$, and let $d \in \N$.
    For each $j \in \{1, \dots, d\}$, let $P_j(x_1, \dots, x_d) \in \F_q(t)[x_1, \dots, x_d]$ be an $\F_q[t]$-valued polynomial such that $P_1, \dots, P_d$ are algebraically independent over $\F_q(t)$.
    Let $\bm{\alpha} = (\alpha_1, \ldots, \alpha_d) \in \mathbf{T}^d$.
    \begin{enumerate}[(1)]
        \item If $\deg{P_j} < p$ for every $j \in \{1, \ldots, d\}$ and at least one $\alpha_j$ is irrational, then
        \begin{equation*}
            \UClim_{\bm{n} \in \F_q[t]^d} e \left( \sum_{j=1}^d P_j(\bm{n}) \alpha_j \right) = 0.
        \end{equation*}
        \item If $\bm{\alpha}$ is strongly irrational\footnote{By this, we mean that $\bm{\alpha}$ is not weakly rational in the sense of item (2) in Lemma \ref{lem: rational character rational element}.}, then either $e \left( \sum_{j=1}^d P_j(\bm{n})\alpha_j \right) = 1$ for every $\bm{n} \in \F_q[t]^d$ or
    \begin{equation*}
        \UClim_{\bm{n} \in \F_q[t]^d} e \left( \sum_{j=1}^d P_j(\bm{n}) \alpha_j \right) = 0.
    \end{equation*}
    \end{enumerate}
\end{proposition}

In order to apply results from \cite{bl} in the proof, it is convenient to reformulate Proposition \ref{prop: equidistribution theorems, exponential sum form} in terms of \emph{well-distribution} of sequences in $\mathbf{T}^d$.
A sequence $(x_{\bm{n}})_{\bm{n} \in \F_q[t]^d}$ is \emph{well-distributed} in $\mathbf{T}^d$ if for every F{\o}lner sequence $(\Phi_N)_{N \in \N}$ in $\F_q[t]^d$,
\begin{equation*}
    \lim_{N \to \infty} \frac{1}{|\Phi_N|} \sum_{\bm{n} \in \Phi_N} \delta_{x_{\bm{n}}} = m_{\mathbf{T}},
\end{equation*}
where the limit is taken in the weak* topology and $m_{\mathbf{T}}$ is the Haar probability measure on $\mathbf{T}$.
With this terminology, Proposition \ref{prop: equidistribution theorems, exponential sum form} can be restated in the following equivalent form.

\begin{proposition} \label{prop: equidistribution theorems}
    Let $\F_q$ be a finite field of characteristic $p$, and let $d \in \N$.
    For each $j \in \{1, \dots, d\}$, let $P_j(x_1, \dots, x_d) \in \F_q(t)[x_1, \dots, x_d]$ be an $\F_q[t]$-valued polynomial such that $P_1, \dots, P_d$ are algebraically independent over $\F_q(t)$.
    Let $\bm{\alpha} = (\alpha_1, \ldots, \alpha_d) \in \mathbf{T}^d$.
    \begin{enumerate}[(1)]
        \item If $\deg{P_j} < p$ for every $j \in \{1, \ldots, d\}$ and at least one $\alpha_j$ is irrational, then $\left( \sum_{j=1}^d P_j(\bm{n})\alpha_j \right)_{\bm{n}\in\F_q[t]^d}$ is well-distributed in $\mathbf{T}$.
        \item If $\bm{\alpha}$ is strongly irrational, then $\left( \sum_{j=1}^d P_j(\bm{n})\alpha_j \right)_{\bm{n}\in\F_q[t]^d}$ is well-distributed in a subgroup of $\mathbf{T}$.
    \end{enumerate}
\end{proposition}

We will deduce Proposition \ref{prop: equidistribution theorems} using a very general result from \cite{bl}, for which we need some additional notation and terminology.

\begin{definition}
    Let $\F_q$ be a finite field of characeristic $p$, and let $\mathbf{T} = \F_q((t^{-1}))/\F_q[t]$.
    The \emph{splitting isomorphism} $\psi_1 : \mathbf{T} \to \mathbf{T}^p$ is the map
    \begin{equation*}
        \psi_1 \left( \sum_{j=1}^{\infty} c_j t^{-j} \right) = \left( \sum_{j=1}^{\infty} c_{p(j-1)+1}t^{-j}, \sum_{j=1}^{\infty} c_{p(j-1)+2}t^{-j}, \ldots, \sum_{j=1}^{\infty} c_{pj}t^{-j} \right).
    \end{equation*}
    The \emph{merging isomorphism} $\varphi_1$ is the inverse map $\varphi_1 = \psi_1^{-1} : \mathbf{T}^p \to \mathbf{T}$.
\end{definition}

The reason for introducing the splitting isomorphism is that we have the identity $\psi_1(\alpha n^p) = \psi_1(\alpha)n$ for $n \in \F_q[t]$ and $\alpha \in \mathbf{T}$.
We can thus use the map $\psi_1$ to convert the additive polynomial $n^p$ into something linear.
In order to handle general additive polynomials, we introduce higher order versions of the splitting isomorphism.
For $k \ge 2$, define
\begin{equation*}
    \psi_k = \psi_1^{\times p^{k-1}} \circ \psi_{k-1},
\end{equation*}
where $\psi_1^{\times p^{k-1}}$ is the map that applies $\psi_1$ to each of $p^{k-1}$ coordinates.
This is the appropriate definition to make in order to obtain the identity $\psi_k(\alpha n^{p^k}) = \psi_k(\alpha) n$ for $n \in \F_q[t]$ and $\alpha \in \mathbf{T}$.
We also define $\varphi_k = \psi_k^{-1}$.

Now let $k \in \N$.
For notational convenience, write $\mathbf{T}^{(k)}$ for the group
\begin{equation*}
    \mathbf{T}^{(k)} = \mathbf{T} \oplus \mathbf{T}^p \oplus \mathbf{T}^{p^2} \oplus \ldots \oplus \mathbf{T}^{p^k}.
\end{equation*}
Define $\sigma_k : \mathbf{T}^{(k)} \to \mathbf{T}$ by
\begin{equation*}
    \sigma_k(x_0, x_1, \ldots, x_k) = x_0 + \varphi_1(x_1) + \ldots + \varphi_k(x_k).
\end{equation*}

The following proposition shows that additive polynomials can be viewed as projections of linear polynomials from higher dimensions:

\begin{proposition}[{\cite[Proposition 7.1]{bl}}] \label{prop: additive linear representation}
    Let $\eta(x) \in \F_q((t^{-1}))[x]$ be an additive polynomial, say
    \begin{equation*}
        \eta(x) = \sum_{j=0}^k \alpha_j x^{p^j}.
    \end{equation*}
    Then there exists $\bm{\beta} \in \mathbf{T}^{(k)}$ such that $\eta(n) = \sigma_k(\bm{\beta} n) \pmod{\F_q[t]}$ for $n \in \F_q[t]$.
\end{proposition}

We now introduce a notion of irrationality for elements of $\mathbf{T}^{(k)}$ and use it to define a corresponding notion of irrationality for additive polynomials.

\begin{definition} \label{defn: irrational element}
    Call a nonzero element $\alpha \in \mathbf{T}$ \emph{rational} if $\alpha \in \F_q(t)/\F_q[t]$ and \emph{irrational} otherwise.
    For $l \in \N$, we say that an element $\bm{\alpha} = (\alpha_1, \ldots, \alpha_l) \in \mathbf{T}^l$ is \emph{irrational} if for every $c_1, \ldots, c_l \in \F_q[t]$ the element
    \begin{equation*}
        \sum_{j=1}^l c_j \alpha_j
    \end{equation*}
    is either $0$ or irrational.
    An additive polynomial $\eta(x) \in \F_q((t^{-1}))[x]$ is \emph{irrational} if $\eta(n) = \sigma_k(\bm{\beta} n)$ for some $k \in \N$ and an irrational element $\bm{\beta} \in \mathbf{T}^{(k)}$.
\end{definition}

\begin{remark}
    A few comments are in order on the notion of irrationality defined here, particularly regarding the possibility that a linear combination $\sum_{j=1}^l c_j \alpha_j$ is equal to 0.
    An analogy with Weyl's equidistribution theorem for polynomial sequences in tori is useful for explaining the role of this notion of irrationality.
    Consider an element $\bm{\alpha} = (\alpha_1, \ldots, \alpha_l) \in \T^l$.
    The orbit-closure $H = \overline{\{n\bm{\alpha} : n \in \Z\}} \subseteq \T^{l}$ is a subgroup of $\T^{l}$.
    This subgroup can be connected (a \emph{subtorus} isomorphic to $\T^k$ for some $k \le l$) or it can be disconnected, in which case it has finitely many connected components, each of which is a translate of a fixed subtorus.
    The subgroup $H$ is a subtorus precisely when $\bm{\alpha}$ is irrational in the sense of Definition \ref{defn: irrational element}, i.e. when every linear combination $\sum_{j=1}^l c_j \alpha_j$ with $c_1, \ldots, c_l \in \Z$ is either 0 or irrational.
    Nontrivial subgroups of the finite characteristic object $\mathbf{T}^l = (\F_q((t^{-1}))/\F_q[t])^l$ are totally disconnected, so the connectedness properties of subtori are not retained in finite characteristic.
    However, several other important features of irrationality remain.

    The first property has to do with equidistribution of linear sequences.
    An element $\bm{\alpha} \in \T^l$ is irrational if and only if it is \emph{totally well-distributed} in its orbit closure $H = \overline{\{n\bm{\alpha} : n \in \Z\}}$, meaning that $((mn+r)\bm{\alpha})_{n \in \Z}$ is well-distributed in $H$ for every $m \in \N$ and $r \in \Z$.
    The corresponding property holds in finite characteristic: $\bm{\alpha} \in \mathbf{T}^l$ is irrational if and only if $((mn+r)\bm{\alpha})_{n \in \F_q[t]}$ is well-distributed in $H = \overline{\{n\bm{\alpha} : n \in \F_q[t]\}}$ for every $m \in \F_q[t] \setminus \{0\}$ and $r \in \F_q[t]$.

    Second, and most importantly for our purposes, irrationality has a direct bearing on polynomial equidistribution properties.
    In the classical integer setting, if $\bm{\alpha} \in \T^l$ is irrational and $P$ is a nonconstant polynomial, then $(P(n)\bm{\alpha})_{n \in \Z}$ is well-distributed in the subtorus $H = \overline{\{n\bm{\alpha} : n \in \Z\}}$.
    This may fail for general elements of $\T^l$.
    Consider $\bm{\alpha} = (\sqrt{2}, \frac{1}{3}) \in \T^2$.
    The orbit-closure $H = \overline{\{(n\sqrt{2}, \frac{n}{3}) : n \in \Z\}}$ can be written as $H = \bigcup_{j=0}^2 \T \times \{\frac{j}{3}\}$.
    The different connected components of $H$ receive unequal treatment from polynomial sequences.
    For example, taking the polynomial $P(n) = n^2$ and using the fact that $1^2 \equiv 2^2 \equiv 1 \pmod{3}$, one can check that $(P(n)\bm{\alpha})_{n \in \Z}$ is well-distributed not with respect to the Haar measure on $H$ but with respect to the measure $\frac{\mu_0 + 2\mu_1}{3}$, where $\mu_j$ is the Haar measures on $\T \times \{\frac{j}{3}\}$.
    A similar property carries over to finite characteristic: if $\bm{\alpha} \in \mathbf{T}^l$ is irrational, then $(P(n)\bm{\alpha})_{n \in \F_q[t]}$ is well-distributed in $\overline{\{n\bm{\alpha} : n \in \F_q[t]\}}$ for every separable polynomial $P$.
    (This may fail for nonseparable polynomials, which is why we utilize a decomposition into additive and separable pieces in Theorem \ref{thm: multivariable equidistribution} below.)
\end{remark}

We can now state the main equidistribution theorem from \cite{bl} for multivariable polynomial sequences.

\begin{theorem}[{\cite[Theorem 10.1]{bl}}] \label{thm: multivariable equidistribution}
    Let $P(x_1, \ldots, x_d) \in \F_q((t^{-1}))[x_1, \ldots, x_d]$ be a polynomial with $P(\bm{0}) = 0$.
    Write $P$ in the form
    \begin{equation*}
        P(\bm{n}) = \sum_{j=1}^m \eta_j(\bm{n}^{\bm{r}_j}),
    \end{equation*}
    where $\bm{n} = (n_1, \ldots, n_d) \in \F[t]^d$, the monomials $\bm{n}^{\bm{r}_j}$ are distinct and separable, meaning $n_1^{r_{j,1}} \cdot \ldots \cdot n_d^{r_{j,d}}$ with $p \nmid \gcd(r_{j,1}, \ldots, r_{j,d})$, and $\eta_j(x) \in \F_q((t^{-1}))[x]$ is an additive polynomial for each $j$.
    If all of the polynomials $\eta_1, \ldots, \eta_d$ are irrational, then $(P(\bm{n}))_{\bm{n} \in \F[t]^d}$ is well-distributed in a subgroup of $\mathbf{T}$.
\end{theorem}

\begin{proof}[Proof of Proposition \ref{prop: equidistribution theorems}]
    (1) The polynomials $P_1, \ldots, P_d$ are linearly independent, so $Q(\bm{n}) = \sum_{j=1}^d P_j(\bm{n}) \alpha_j$ is a polynomial of degree less than $p$ with an irrational coefficient other than the constant term.
    That is, if we write
    \begin{equation*}
        Q(\bm{n}) = Q(\bm{0}) + \sum_{k=1}^{p-1} \sum_{r_1+\ldots+r_d=k} \beta_{\bm{r}} \bm{n}^{\bm{r}},
    \end{equation*}
    then $\beta_{\bm{r}}$ is irrational for some $\bm{r}$.
    Since $\eta(n) = \beta_{\bm{r}}n$ is well-distributed in $\mathbf{T}$ (see \cite[Theorem 0.1]{bl}), it follows that $Q(\bm{n})$ is well-distributed in $\mathbf{T}$ by Theorem \ref{thm: multivariable equidistribution}. \\

    (2) Let $Q(\bm{n}) = \sum_{j=1}^d P_j(\bm{n}) \alpha_j$, and write $Q(\bm{n})$ in the form
    \begin{equation*}
        Q(\bm{n}) = Q(\bm{0}) + \sum_{i=1}^m \eta_i(\bm{n}^{\bm{r}_i}),
    \end{equation*}
    where $\eta_1, \ldots, \eta_m$ are additive polynomials and the ``monomials'' $\bm{n}^{\bm{r}_i} = n_1^{r_{i,1}} \ldots n_d^{r_i,d}$ are distinct and separable.
    The assumption that $\bm{\alpha}$ is strongly irrational together with the algebraic independence of the polynomials $P_1, \ldots, P_d$ ensures that each of the additive polynomials $\eta_i$ is irrational.
    Hence, by \cite[Theorem 10.1]{bl}, the $\F[t]_q^d$-sequence $Q(\bm{n})$ is well-distributed in a subgroup of $\mathbf{T}$.
\end{proof}

With Theorem \ref{thm: Weyl rings of integers low degree} at hand, we can now prove Theorem \ref{thm: TE}.

\begin{proof}[Proof of Theorem \ref{thm: TE}]
    Let us first prove the equivalences between (i), (ii), and (iii).
    We deal with the contrapositive of each statement.

    (i) $\implies$ (ii).
    Suppose $(X, \mu, (T_n)_{n \in \mO_K})$ has a nontrivial periodic factor, and let $I$ be the ideal of periods of this factor.
    Let $A \subseteq X$ be measurable with respect to this periodic factor and with $\mu(A) \in (0,1)$.
    Then $T_n^{-1}A = A$ for all $n \in I$ by periodicity of the factor, so $(T_n)_{n \in I}$ is not ergodic.

    (ii) $\implies$ (iii).
    Suppose there is a nonzero rational eigencharacter $\chi : \mO_K \to S^1$, and let $G$ be the subgroup of $\hat{\mO}_K$ generated by $\chi$.
    Then the space of functions
    \begin{equation*}
        \textup{span}\{f \in L^2(\mu) : f~\text{is an eigenfunction with eigencharacter in}~G\}
    \end{equation*}
    defines a periodic factor.

    (iii) $\implies$ (i).
    Let $I$ be a nonzero propert ideal of $\mO_K$, and suppose $(T_n)_{n \in I}$ is not ergodic.
    Then there is a non-constant function $f \in L^2(\mu)$ with $T_nf = f$ for all $n \in I$.
    For a character $\chi : \mO_K/I \to S^1$, define a new function
    \begin{equation*}
        g_{\chi} = \sum_{r \in \mO_K/I} \chi(-r) T_r f.
    \end{equation*}
    Note that $T_rf$ is well-defined for $r \in \mO_K/I$, since $f$ is an $I$-invariant function.
    Then for any $n \in \mO_K$,
    \begin{equation*}
        T_ng_{\chi} = \sum_{r \in \mO_K/I} \chi(-r) T_{n+r} f
         = \sum_{r \in \mO_K/I} \chi(n) \chi(-(n+r)) T_{n+r} f
         = \chi(n) g.
    \end{equation*}
    That is, $g$ is an eigenfunction with $I$-periodic eigencharacter $\tilde{\chi} = \chi \circ \pi$, where $\pi : \mO_K \to \mO_K/I$ is the projection map. \\

    Now let us check that (iv) implies ergodicity along ideals.
    We will prove (iv) $\implies$ (iii) by contrapositive.
    Suppose there is a nonzero rational eigencharacter $\chi : \mO_K \to S^1$, and let $f \in L^2(\mu)$ be an associated eigenfunction.
    Let $m$ be a period of $\chi$.
    Then taking $P(n) = mn$, we have
    \begin{equation*}
        \UClim_{n \in \mO_K} T_{P(n)}f = \UClim_{n \in \mO_K} \chi(mn)f = f \ne \int_X f~d\mu,
    \end{equation*}
    so (iv) fails.

    Finally, let us check that ergodicity along ideals (as expressed by items (i), (ii), and/or (iii)) implies (iv).
	By the spectral theorem for actions of $(\mO_K,+)$ by unitary operators on a Hilbert space, we may work with the Hilbert space $\Hil = L^2 \left( \hat{\mO}_K, \sigma \right)$, where $\sigma$ is a positive Borel measure on the dual group $\hat{\mO}_K$, and the unitary action $(T_n)_{n \in \mO_K}$ is represented by the multiplication operators $(U_nh)(x) = \chi(n) h(\chi)$ for $h \in \Hil$ and $\chi \in \hat{\mO}_K$.

    Let $P(n) \in \mO_K[n]$ be a nonconstant polynomial with $\deg{P} < \char(K)$.
    We want to show that for $\sigma$-a.e. $\chi \in \hat{\mO}_K$, $\UClim_{n \in \mO_K} \chi(P(n)) = 0$.
The assumption that $(T_n)_{n \in \mO_K}$ is ergodic along ideals means that $\sigma$ is supported on irrational characters, so this follows immediately from Theorem \ref{thm: Weyl rings of integers low degree}.
\end{proof}

The same argument with Theorem \ref{thm: Weyl rings of integers} in place of Theorem \ref{thm: Weyl rings of integers low degree} proves Theorem \ref{thm: strongly ergodic}. \\

Now we use the strategy behind the proof of Theorem \ref{thm: TE} to deduce Theorem \ref{thm: TE + ud}.

\begin{proof}[Proof of Theorem \ref{thm: TE + ud}]
    (i)$\implies$(ii).
    Upon replacing $T_n$ by the multiplication operators $(U_nh)(x) = \chi(n) h(\chi)$ on $L^2 \left( \hat{\mO}_K, \sigma \right)$, we use the fact that $P$ is good for irrational equidistribution to conclude
    \begin{equation*}
    	\UClim_{n \in \mO_K} U_{P(n)} h(\chi)
	   = \UClim_{n \in \mO_K} \chi(P(n)) h(\chi)
	   = \begin{cases}
	 	 h(0), & \text{if}~\chi \equiv 1; \\
		  0, & \text{otherwise}
	   \end{cases}
    \end{equation*}
    in $L^2(\sigma)$.
    This corresponds to the desired $L^2$ convergence result
    \begin{equation*}
	   \UClim_{n \in \mO_K} T_{P(n)}f = \int_X f~d\mu.
    \end{equation*}
    
    (ii)$\implies$(i).
    Suppose $P$ is not good for irrational equidistribution, and let $\chi \in \hat{\mO}_K$ be an irrational character such that $\UClim_{n \in \mO_K} \chi(P(n)) \ne 0$.
    We then take as our totally ergodic system $X = \hat{\mO}_K$, $\mu = m_X$ the normalized Haar measure on the compact group $X$, and $T_nx$ is the character $(T_nx)(m) = \chi(nm) x(m)$.
    For the function $f(x) = x(1)$, we have $\int_X f(x)~dx = 0$, since $x \mapsto x(1)$ is a nontrivial character on $X$, while
    \begin{equation*}
	   \UClim_{n \in \mO_K} T_{P(n)} f
	   = \UClim_{n \in \mO_K} \chi(P(n))f \ne 0.
    \end{equation*}
\end{proof}


\section{Necessary and sufficient conditions for the Furstenberg--S\'{a}rk\"{o}zy theorem over rings of integers of global fields} \label{sec:nec/suff}

In this section, we prove Theorem \ref{thm: global field Sarkozy}, establishing necessary and sufficient conditions for the Furstenberg--S\'{a}rk\"{o}zy theorem over rings of integers of global fields.
For convenience of the reader, we reproduce the statement below.

\NecessarySufficient*

In the number field case, several of the equivalences were established in \cite{br}.
To handle the function field case, we need to incorporate some additional ideas from \cite{bl} expressed in Theorem \ref{thm: multivariable equidistribution} above.
One additional ingredient that we need and that was not present in Section \ref{sec:TE polynomial} is a lemma allowing us to pass to the case that the additive polynomials $\eta_j$ appearing in the decomposition of a polynomial
\begin{equation*}
    P(\bm{n}) = P(0) + \sum_{j=1}^m \eta_j(\bm{n}^{\bm{r}_j})
\end{equation*}
are irrational:

\begin{lemma} \label{lem: clear denominators}
    Let $\eta(x) \in \F_q((t^{-1}))[x]$ be an additive polynomial.
    There exists $m_0 \in \F_q[t]$ with the following property: if $m \in \F_q[t]$ and $m_0 \mid m$, then $\eta(mx)$ is irrational.
\end{lemma}

\begin{proof}
    This lemma is proved implicitly in \cite{bl} (see the proof of Theorem 8.1 in \cite{bl} and the remark thereafter).
    For completeness, we include a short proof here.
    By Proposition \ref{prop: additive linear representation}, let $k \in \N$ and $\bm{\beta} \in \mathbf{T}^{(k)}$ such that $\eta(n) = \sigma_k(\bm{\beta} n)$.
    By \cite[Lemma 4.1]{bl}, there exists $m_0 \in \F_q[t]$ such that $m_0 \bm{\beta}$ is an irrational element of $\mathbf{T}^{(k)}$.

    Let $m \in \F_q[t]$ with $m_0 \mid m$.
    Then $m\bm{\beta}$ is irrational (an arbitrary linear combination of coordinates of $m\bm{\beta}$ can also be expressed as a linear combination of coordinates of $m_0\bm{\beta}$), and $\eta(mn) = \sigma_k((m\bm{\beta})n)$, so $n \mapsto \eta(mn)$ is irrational as desired.
\end{proof}

We are now ready to turn to the proof of Theorem \ref{thm: global field Sarkozy}.

\begin{proof}[Proof of Theorem \ref{thm: global field Sarkozy}]
    We will prove the equivalences expressed in the following diagram.

    
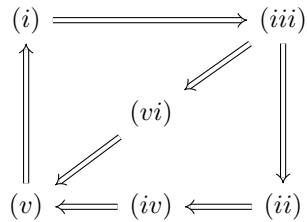
\begin{center}
	    \begin{tikzcd}
		      (i) \arrow[rr, Rightarrow] & & (iii) \arrow[dd, Rightarrow] \arrow[ld, Rightarrow] \\
             & (vi) \arrow[ld, Rightarrow] & \\
            (v) \arrow[uu, Rightarrow] & (iv) \arrow[l, Rightarrow] & (ii) \arrow[l, Rightarrow]
        \end{tikzcd}
    \end{center}

    (i)$\implies$(iii).
	It suffices to show that for any measure-preserving $\mO_K$-system $\left( X, \mu, (T_n)_{n \in \mO_K} \right)$,
	any measurable set $A \subseteq X$, and any $\eps > 0$, there exists $m, r \in \mO_K$ with $m \ne 0$ such that
	\begin{align*}
		\UClim_{n \in \mO_K} \mu \left( A \cap T_{P(mn+r)}^{-1}A \right) > \mu(A)^2 - \eps.
	\end{align*}
    When $K$ is a number field, this was shown in \cite[Theorem 4.1]{ab-roi} (see also \cite[Theorem 1.17]{br} for a different method of establishing (iii) in the number field setting), so we will prove only the function field case here.

    Let $K$ be a degree $d$ extension of $\F_q(t)$, and let $\{b_1, \ldots, b_d\}$ be an integral basis for $\mO_K$.
    We may then express characters on $\mO_K$ in the form
    \begin{equation*}
        \sum_{j=1}^d n_jb_j \mapsto e \left( \sum_{j=1}^d n_j\alpha_j \right)
    \end{equation*}
    with $\bm{\alpha} = (\alpha_1, \ldots, \alpha_d) \in (\F_q((t^{-1}))/\F_q[t])^d = \mathbf{T}^d$.
    
	Let $f = \ind_A$.
	By the spectral theorem, there exists a measure $\sigma_f$ on $\mathbf{T}^d$ so that the action of $T$ on the closure of the orbit
	$\{T_nf : n \in \mO_K\}$ in $L^2(X, \mu)$ is isomorphic to the action of the multiplication operators $U_n h(\bm{x}) = e \left( \sum_{j=1}^d n_j x_j \right) h(\bm{x})$ on $L^2(\mathbf{T}^d, \sigma_f)$, where $n = \sum_{j=1}^d n_j b_j$.
	By the mean ergodic theorem, the function $\tilde{f} = \UClim_{n \in \mO_K} T_nf$ is represented in $L^2(\mathbf{T}^d, \sigma_f)$
	by the function $\ind_{\{\bm{0}\}}$.
	Hence,
 \begin{equation*}
     \sigma_f(\{\bm{0}\}) = \norm{L^2(\mathbf{T}^d, \sigma_f)}{\ind_{\{\bm{0}\}}}^2 = \norm{L^2(X, \mu)}{\tilde{f}}^2 \ge \mu(A)^2.
 \end{equation*}

    Let $Q(x_1, \ldots, x_d) \in \F_q[t]^d[x_1, \ldots, x_d]$ be the polynomial such that
    \begin{equation*}
        \sum_{j=1}^d Q_j(\bm{n})b_j = P \left( \sum_{j=1}^d n_j b_j \right),
    \end{equation*}
    where $Q = (Q_1, \ldots, Q_d)$.
    Write $Q$ in the form
    \begin{equation*}
        Q(\bm{n}) = Q(\bm{0}) + \sum_{j=1}^l \eta_j(\bm{n}^{\bm{r}_j})
    \end{equation*}
    with $\eta_j$ additive polynomials and $\bm{n}^{\bm{r}_j}$ distinct separable monomials.
    For $\bm{x} = (x_1, \ldots, x_d) \in \mathbf{T}^d$, let $\eta_{j,\bm{x}}$ be the $\mathbf{T}$-valued additive polynomial $\eta_{j,\bm{x}}(n) = \eta_{j}(n) \cdot\bm{x}$.
    Note that if $n = \sum_{i=1}^d n_i b_i \in \mO_K$, then
    \begin{equation*}
        U_{P(n)} h(\bm{x}) = e \left( Q(\bm{n}) \cdot \bm{x} \right) h(\bm{x})
         = e \left( Q(\bm{0}) \cdot \bm{x} + \sum_{j=1}^l \eta_{j,\bm{x}}(\bm{n}^{\bm{r}_j}) \right) h(\bm{x}).
    \end{equation*}

    Fix $\bm{x} = (x_1, \ldots, x_d) \in \mathbf{T}^d$.
    By Lemma \ref{lem: clear denominators}, for each $j$, let $m_{j,\bm{x}} \in \F_q[t]$ such that if $m_{j,\bm{x}} \mid m$, then $\eta_{j,\bm{x}}(mn)$ is irrational (as a polynomial in $n \in \F_q[t]$).
    Set $m_{\bm{x}}$ to be the least common multiple of $m_{1,\bm{x}}, \ldots, m_{l,\bm{x}}$..
    By construction, if $m_{\bm{x}} \mid m$, then each of the additive polynomials $\eta_{j,\bm{x}}(mn)$ is irrational.
    Therefore, by Theorem \ref{thm: multivariable equidistribution},
    \begin{equation*}
        \bm{n} \mapsto Q(m\bm{n}) \cdot \bm{x}
    \end{equation*}
    is well-distributed in $H_{\bm{x}} + Q(\bm{0}) \cdot \bm{x}$ for some subgroup $H_{\bm{x}}$ of $\mathbf{T}$.
    In fact, for any $\bm{r} \in \F_q[t]^d$,
    \begin{equation*}
        \bm{n} \mapsto Q(m\bm{n} + \bm{r}) \cdot \bm{x}
    \end{equation*}
    is well-distributed in $H_{\bm{x}} + Q(\bm{r}) \cdot \bm{x}$.
	
	Choose $m \in \F_q[t]$ sufficiently divisible so that $D = \{\bm{x} \in \mathbf{T}^d : m_{\bm{x}} \mid m\} \cup \{\bm{0}\}$ satisfies $\sigma_f(\mathbf{T}^d \setminus D) < \eps$.
	By condition (i), there exists $r \in \mO_K$ such that $P(r) \equiv 0 \pmod{m}$.
    (Here we view $m$ as an element of $\mO_K$.)
    Hence, for $\bm{x} \in D$,
    \begin{equation*}
        \UClim_{n \in \mO_K} U_{P(mn+r)}h(\bm{x}) = \UClim_{\bm{n} \in \F_q[t]^d} e \left( Q(m\bm{n}+\bm{r}) \cdot \bm{x} \right) h(\bm{x}) = \begin{cases}
		 	h(\bm{x}), & \text{if}~e(H_{\bm{x}}) = \{1\} \\
			0, & \text{otherwise}.
		 \end{cases}
	\end{equation*}
	In particular,
	\begin{equation*}
        \UClim_{n \in \mO_K} \overline{h(\bm{x})} \cdot U_{P(mn+r)}h(\bm{x}) \ge 0
	\end{equation*}
	for every $\bm{x} \in D$.
	Therefore,
	\begin{multline*}
		\UClim_{n \in \mO_K} \mu \left( A \cap T_{P(mn+r)}^{-1}A \right)
		 = \UClim_{\bm{n} \in \F_q[t]^d} \int_{\mathbf{T}^d}{e \left( Q(m\bm{n}+\bm{r}) \cdot \bm{x} \right)~d\sigma_f(x)} \\
		= \sigma_f(\{\bm{0}\})
		 + \int_{D \setminus \{\bm{0}\}}{\UClim_{\bm{n} \in \F_q[t]^d} e \left( Q(m\bm{n}+\bm{r}) \cdot \bm{x} \right)}
		+ \int_{\mathbf{T}^d \setminus D}{\UClim_{\bm{n} \in \F_q[t]^d} e \left( Q(m \bm{n} + \bm{r}) \cdot \bm{x} \right)} \\
		\ge \sigma_f(\{\bm{0}\}) - \sigma_f(\mathbf{T}^d \setminus D) > \mu(A)^2 - \eps.
	\end{multline*}
    
    The implication (iii)$\implies$(ii) is trivial. \\
    
    (ii)$\implies$(iv).
    Suppose (ii) holds, and let $E \subseteq \mO_K$ with $d^*(E) > 0$.
    By the Furstenberg correspondence principle (Theorem \ref{thm: correspondence}), there is a measure-preserving system $(X, \mu, (T_n)_{n\in \mO_K})$ and a set $A \subseteq X$ with $\mu(A) > 0$ such that
    \begin{equation} \label{eq: correspondence inequality}
        d^* \left( \bigcap_{j=1}^k (E - n_j) \right) \ge \mu \left( \bigcap_{j=1}^k T^{-n_j}A \right)
    \end{equation}
    for every $k \in \N$ and $n_1, \dots, n_k \in \mO_K$.
    Applying (ii), there exists $n \in \mO_K \setminus \{0\}$ such that
    \begin{equation*}
        \mu \left( A \cap T_{P(n)}^{-1}A \right) > 0.
    \end{equation*}
    Returning to the inequality \eqref{eq: correspondence inequality}, we conclude that $d^*(E \cap (E - P(n)) > 0$; in particular, $E \cap (E - P(n)) \ne \es$.
    Taking $y \in E \cap (E - P(n))$ and letting $x = y + P(n)$ provides the desired points $x,y \in E$ with $x-y = P(n)$. \\
    
    (iv)$\implies$(v).
    Given any finite coloring of $\mO_K$, one of the color classes must have positive upper Banach density, so (iv)$\implies$(v). \\

    (v)$\implies$(i).
    Suppose (i) fails.
    Let $I$ be a nonzero ideal of $\mO_K$ such that $P(\mO_K) \cap I = \es$.
    We can then color $\mO_K$ by cosets mod $I$ to see that (v) fails. \\

    (iii)$\implies$(vi).
    We use an argument originating in \cite{b_density-schur} to establish partition regularity of the equation $x - y = P(z)$.
    Let $\mO_K = \bigcup_{i=1}^r C_i$ be a finite coloring of $\mO_K$.
    Fix a F{\o}lner sequence $(\Phi_N)_{N \in \N}$.
    We partition the colors into ``large'' and ``small'' color classes as
    \begin{equation*}
        L = \{1 \le i \le r : \overline{d}_{\Phi}(C_i) > 0\} \qquad \text{and} \qquad S = \{1 \le i \le r : \overline{d}(C_i) = 0\},
    \end{equation*}
    where $\overline{d}_{\Phi}(C) = \limsup_{N \to \infty} \frac{|C\cap\Phi_N|}{|\Phi_N|}$.
    Consider the set $E = \prod_{i \in L} C_i \subseteq \mO_K^L$.
    We have $d_{\Psi}(E) > 0$ for the F{\o}lner sequence
    \begin{equation*}
        \Psi_k = \prod_{i \in L} \Phi_{N_{i,k}},
    \end{equation*}
    where the subsequence $\Phi_{N_{i,k}}$ is chosen so that $\overline{d}_{\Phi}(C_i) = \lim_{k \to \infty} \frac{|C_i \cap \Phi_{N_{i,k}}|}{|\Phi_{N_{i,k}}|}$.
    By the Furstenberg correspondence principle, there exists a measure-preserving system $(X, \mu, (T_n)_{n \in \mO_K})$ and a measurable set $A$ with $\mu(A) = d_{\Psi}(E) = \prod_{i \in L} \overline{d}_{\Phi}(C_i) > 0$ such that
    \begin{equation*}
        \overline{d}_{\Psi} \left( \bigcap_{j=1}^k (E - (n_j,\ldots,n_j)) \right) \ge \mu \left( \bigcap_{j=1}^k T_{n_j}^{-1}A \right)
    \end{equation*}
    for every $k \in \N$ and $n_1, \ldots, n_k \in \mO_K$.
    By (iii), the set
    \begin{equation*}
        R = \left\{ n \in \mO_K : \mu \left( A \cap T_{P(n)}^{-1}A \right) > \frac{\mu(A)^2}{2} \right\}
    \end{equation*}
    is syndetic.
    In particular, $\underline{d}_{\Phi}(R) > 0$, so $R \cap \bigcup_{i \in L} C_i \ne \es$.
    Let $z \in R \cap C_{i_0}$ for some $i_0 \in L$.
    Then
    \begin{equation*}
        \overline{d}_{\Phi}(C_{i_0} \cap (C_{i_0} - P(z)) \ge \prod_{i \in L} \overline{d}_{\Phi}(C_i \cap (C_i - P(z)) \ge \overline{d}_{\Psi} \left( E \cap (E - (P(z),\ldots,P(z)) \right) \ge \mu(A \cap T_{P(z)}^{-1}A) > 0.
    \end{equation*}
    Let $y \in C_{i_0} \cap C_{i_0} - P(z)$, and let $x = y + P(z)$.
    Then $\{x,y,z\} \subseteq C_{i_0}$ is a monochromatic solution to the equation $x - y = P(z)$.
    \\
    
    (vi)$\implies$(v) is trivial.
\end{proof}

\begin{remark}
    The proof of the implication (i)$\implies$(iii) relies on multivariable polynomial equidistribution results.
    A similar application of equidistribution theorems in conjunction with the spectral theorem can be used to establish a multivariable enhancement: if $P(x_1, \ldots, x_l) \in \mO_K[x_1, \ldots, x_l]$ and for any nonzero ideal $I$ of $\mO_K$, there exists $\bm{n} \in \mO_K^l$ such that $P(\bm{n}) \in I$, then for any measure-preserving system $(X, \mu, (T_n)_{n \in \mO_K})$, any measurable set $A \subseteq X$ with $\mu(A) > 0$, and any $\eps > 0$, the set
    \begin{equation*}
        \left\{ \bm{n} \in \mO_K^l : \mu \left( A \cap T_{P(\bm{n})}^{-1}A \right) > \mu(A)^2 - \eps \right\}
    \end{equation*}
    is syndetic in $\mO_K^l$.
\end{remark}


\section{Asymptotic total ergodicity} \label{sec:ATE proof}

The goal of this section is to prove Theorem \ref{thm: quantitative TE}, restated here for the convenience of the reader.

\QuantitativeTE*

Using Fourier analysis, we will reduce Theorem \ref{thm: quantitative TE} to the following exponential sum estimate whose proof occupies the bulk of this section:

\begin{lemma} \label{lem: character bound}
    Let $\mO_K$ be the ring of integers of a global field $K$ with characteristic $c \in \P \cup \{\infty\}$.
	Let $P(x) \in \mO_K[x]$ be a nonconstant polynomial of degree $d$ and derivational degree $k$ with $P(0) = 0$.
	Then for any nontrivial ideal $I \le \mO_K$ and any $\chi \in \hat{\mO_K/I}$, if $\chi \notin H_{I,P}^{\perp}$, then
	\begin{equation*}
		\left| \frac{1}{[\mO_K:I]} \sum_{x \in \mO_K/I}{\chi(P(x))} \right|^{2^{k-1}} \le B(d,c) \frac{k-1}{\lpf(I)}.
	\end{equation*}
\end{lemma}

An important ingredient in the proof of Lemma \ref{lem: character bound} is the following variant of the van der Corput inequality.
We do not use any ring structure for this result, so we state and prove it in the setting of an arbitrary finite abelian group $G$.
For a function $f : G \to \C$, define a multiplicative differencing operator by $\Delta_vf(u) = f(u+v) \overline{f(u)}$,
and let $\Delta_{v_1, \dots, v_k}f = \Delta_{v_k}\left( \Delta_{v_1, \dots, v_{k-1}}f \right)$
for $k \in \N$ and $v_1, \dots, v_k \in G$.

\begin{lemma} \label{lem: vdC}
	Let $G$ be a finite abelian group, and let $H \le G$ be a subgroup.
	For any function $f : G \to \C$ and any $k \in \N$,
	\begin{equation*}
		\left| \E_{x \in G}{f(x)} \right|^{2^k} \le \E_{v_1, \dots, v_k \in H}{\E_{u \in G}{\Delta_{v_1, \dots, v_k}f(u)}}
	\end{equation*}
\end{lemma}

\begin{remark}
	It is worth commenting on two special cases of Lemma \ref{lem: vdC}.
	When $H = G$, the quantity on the right hand side is equal to $\norm{U^k(G)}{f}^{2^k}$, so the conclusion of Lemma \ref{lem: vdC} reduces to the inequality $\norm{U^1(G)}{f} \le \norm{U^k(G)}{f}$, which is a special case of monotonicity for the Gowers (semi)norms.
	On the other hand, when $H = \{0\}$, the right hand side is equal to $\mathbb{E}_{u \in G}{|f(u)|^{2^k}}$, so the conclusion of Lemma \ref{lem: vdC} follows by Jensen's inequality.
	The general case can be seen as interpolating between these two extremes.
\end{remark}

\begin{proof}
	Suppose $k =1$.
	Note that
	\begin{equation*}
		\E_{x \in G}{f(x)} = \E_{x \in G}{\E_{h \in H}{f(x+h)}}.
	\end{equation*}
	Therefore, by Jensen's inequality,
	\begin{equation*}
		\left| \E_{x \in G}{f(x)} \right|^2
		 = \left| \E_{x \in G}{\E_{h \in H}{f(x+h)}} \right|^2
		 \le \E_{x \in G}{\left| \E_{h \in H}{f(x+h)} \right|^2}
		 = \E_{x \in G}{\E_{h_1, h_2 \in H}{f(x+h_1) \overline{f(x+h_2)}}}.
	\end{equation*}
	Interchanging the order of averaging and making the substitutions $v = h_1 - h_2$, $u = x - h_2$,
	we obtain the desired inequality
	\begin{equation*}
		\left| \E_{x \in G}{f(x)} \right|^2 \le \E_{v \in H}{\E_{u \in G}{f(u+v) \overline{f(u)}}}.
	\end{equation*} \\
	
	Suppose the inequality holds for $k - 1$.
	Then
	\begin{equation*}
		\left| \E_{x \in G}{f(x)} \right|^{2^k}
		 = \left( \left| \E_{x \in G}{f(x)} \right|^{2^{k-1}} \right)^2
		 \le \left| \E_{v_1, \dots, v_{k-1} \in H}{\E_{u \in G}{\Delta_{v_1, \dots, v_{k-1}}f(u)}} \right|^2,
	\end{equation*}
	which is in turn bounded above by
	\begin{equation*}
		\E_{v_1, \dots, v_{k-1} \in H}{\left| \E_{u \in G}{\Delta_{v_1, \dots, v_{k-1}}f(u)} \right|^2}.
	\end{equation*}
	For fixed $v_1, \dots, v_{k-1}$, applying the $k=1$ case with the function $\Delta_{v_1, \dots, v_{k-1}}f$ gives
	\begin{equation*}
		\left| \E_{u \in G}{\Delta_{v_1, \dots, v_{k-1}}f(u)} \right|^2
		 \le \E_{v_k \in H}{\E_{u \in G}{\Delta_{v_k}\Delta_{v_1, \dots, v_{k-1}}f(x)}}
		 = \E_{v_k \in H}{\E_{u \in G}{\Delta_{v_1, \dots, v_{k-1}, v_k}f(x)}}.
	\end{equation*}
	Putting everything together,
	\begin{equation*}
		\left| \E_{x \in G}{f(x)} \right|^{2^k} \le \E_{v_1, \dots, v_k \in H}{\E_{u \in G}{\Delta_{v_1, \dots, v_k}f(u)}}.
	\end{equation*}
\end{proof}

Before proving Lemma \ref{lem: character bound}, we review some basic facts from Fourier analysis and introduce convenient notation for Fourier analysis in quotient rings $\mO_K/I$.
A classical result in Fourier analysis says that the dual group of $\mO_K/I$ is isomorphic to the group of characters $\chi : \mO_K \to S^1$ that annihilate $I$, which we denote by $I^{\perp}$ (see \cite[Theorem 2.1.2]{rudin}).
As mentioned in Section \ref{sec:TE polynomial}, characters on $\mO_K$ can be represented in the form $\chi(n) = e(n \cdot \alpha)$ for some $\alpha$ with coordinates in either $\T$ (if $\mO_K$ is a number field) or $\mathbf{T} = \F_q((t^{-1}))/\F_q[t]$ (if $\mO_K$ is a function field).
If $\alpha$ represents a character $\chi \in I^{\perp}$, we will also write $\alpha \in I^{\perp}$.
Then, in the finite ring $\mO_K/I$, we may decompose every function $f : \mO_K/I \to \C$ as a Fourier series
\begin{equation*}
    f(x) = \sum_{s \in I^{\perp}} \hat{f}(s) e(s \cdot x),
\end{equation*}
where the Fourier coefficients $\hat{f}(s)$ are given by
\begin{equation*}
    \E_{x \in \mO_K/I} f(x)e(-s \cdot x).
\end{equation*}

\begin{proof}[Proof of Lemma \ref{lem: character bound}]
	We first make a reduction to separable polynomials.
    Write $P(x) = \sum_{i=1}^n \eta_i(x^{r_i})$ with $\eta_i$ additive polynomials and $r_i$ distinct powers not divisible by the characteristic of $\mO_K$.
    Let $H_i = \eta_i(\mO_K/I)$.
    Note that the group $H_I$ generated by the values of $P$ is a subgroup of $H_1 + H_2 + \ldots + H_n$.
    Therefore, if $\chi \notin H_I^{\perp}$, then $\chi \notin H_i^{\perp}$ for some $i$.
    Let
	\begin{equation*}
		J = \left\{ 1 \le i \le n : \chi \notin H_i^{\perp} \right\} \ne \es.
	\end{equation*}
	For any $x \in \mO_K/I$,
	\begin{equation*}
		\chi(P(x)) = \chi \left( \sum_{i \in J}{\eta_i(x^{r_i})} \right).
	\end{equation*}
	Now, for each $i \in J$, the function $\chi_i(x) = \chi(\eta_i(x))$ is a nontrivial character on $\mO_K/I$, so there exists $s_i \in I^{\perp} \setminus \{0\}$ such that $\chi_i(x) = e(s_i \cdot x)$.

	It therefore suffices to prove the following: for any nonconstant separable polynomial $P(x)$ with coefficients in $I^{\perp}$,
	\begin{equation*}
		\left| \E_{x \in \mO_K/I}{e \left( P(x) \right)} \right|^{2^{k-1}} \le B(d,c) \frac{k-1}{\lpf(I)}.
	\end{equation*} \\
	
	Suppose $k = 1$.
	Then $P(x) = s_1 \cdot x$ with $s_1 \ne 0$, so by orthogonality of characters,
	\begin{equation*}
		\E_{x \in \mO_K/I}{e(P(x))} = \E_{x \in \mO_K/I}{e(s_1 \cdot x)} = 0.
	\end{equation*} \\
	
	Now suppose $k \ge 2$.
	Let $P(x) = \sum_{i=i}^n{s_i \cdot x^{r_i}} + P'(x)$, where $\ddeg{x^{r_i}} = k$ for $i \in \{1, \dots, n\}$, and $\ddeg{P'} \le k-1$.
	By Lemma \ref{lem: vdC},
	\begin{equation*}
		\left| \E_{x \in \mO_K/I}{e(P(x))} \right|^{2^{k-1}} \le \E_{v_1, \dots, v_{k-1} \in K}{\E_{u \in \mO_K/I}{e \left( \partial_{v_1, \dots, v_{k-1}}P(u) \right)}}
	\end{equation*}
	for any subgroup $K \le (\mO_K/I, +)$.
	(We will take a convenient choice for $K$ later.)
	We now wish to obtain an expression for $\partial_{v_1, \dots, v_k}P(u)$ that will allow us to bound the avaerage
	\begin{equation*}
		\E_{u \in \mO_K/I}{e \left( \partial_{v_1, \dots, v_{k-1}}P(u) \right)}.
	\end{equation*}
	
	For $v_1, \dots, v_{k-1} \in \mO_K/I$, one has that $\partial_{v_1, \dots, v_{k-1}}P'(u)$ is constant (as a function of $u$),
	since $\ddeg{P'} \le k-1$, so we can pull the constant $e \left( \partial_{v_1, \dots, v_{k-1}}P'(u) \right)$
	outside of the average.
    
    If the characteristic $c$ is a prime number $p$, we let $m = \floor{\log_p{d}}$ so that $p^m \le d$ and $p^{m+1} > d$.
    Otherwise, if $c = \infty$, take $m = 0$.
	For each $i = 1, \dots, n$, we may write $r_i = \sum_{j=0}^m{c_{i,j}p^j}$ with $c_{i,j} \in \{0, \dots, p-1\}$ and $\sum_{j=0}^m{c_{i,j}} = k$.
	Since $P$ is separable by assumption, we have $c_{i,0} \ne 0$ for $i \in \{1, \dots, n\}$.
	Then
	\begin{equation*}
		\partial_{v_1, \dots, v_{k-1}}(u^{r_i}) = b_i \sum_{j=1}^m{S_{i,j}(v_1, \dots, v_{k-1}) u^{p^j}} + R_i(v_1, \dots, v_{k-1}),
	\end{equation*}
	where $b_i = \prod_{j=1}^m{c_{i,j}!}$, $S_{i,j}(v_1, \dots, v_{k-1})$ is the sum of all monomials of the form
	\begin{equation*}
		\prod_{l=1}^{k-1}{v_l^{p^{j_l}}}
	\end{equation*}
	with
	\begin{equation*}
		\left| \left\{ 1 \le l \le k-1 : j_l = j' \right\} \right| = \begin{cases}
			c_{i,j'}, & \text{if}~j' \ne j \\
			c_{i,j} - 1, & \text{if}~j' = j
		\end{cases}
	\end{equation*}
	and $R_i$ is a symmetric polynomial in $k-1$ variables.
	(If $c_{i,j} = 0$, then $S_{i,j} = 0$.)
	We can therefore write
	\begin{align*}
		\partial_{v_1, \dots, v_{k-1}}P(u)
		 & = \sum_{i=1}^n{s_i \cdot b_i \sum_{j=1}^m{S_{i,j}(v_1, \dots, v_{k-1}) u^{p^j}}} + R(v_1, \dots, v_{k-1}) \\
		 & = \sum_{j=0}^m{\left( \sum_{i=1}^n{s_i \cdot b_i S_{i,j}(v_1, \dots, v_{k-1})} \right) u^{p^j}} + R(v_1, \dots, v_{k-1}),
	\end{align*}
	where $R = \sum_{i=1}^n{R_i}$.
	
	Let
	\begin{equation*}
		\eta_{v_1, \dots, v_{k-1}}(u) = \sum_{j=0}^m{\left( \sum_{i=1}^n{s_i \cdot b_i S_{i,j}(v_1, \dots, v_{k-1})} \right) u^{p^j}}.
	\end{equation*}
	Note that $\eta_{v_1, \dots, v_{k-1}}$ is a group homomorphism.
	It follows that
	\begin{equation*}
		\left| \E_{u \in \mO_K/I}{e \left( \partial_{v_1, \dots, v_{k-1}}P(u) \right)} \right|
		 = \left| \E_{u \in \mO_K/I}{e \left( \eta_{v_1, \dots, v_{k-1}}(u) \right)} \right|
		 = 0
	\end{equation*}
	whenever $\eta_{v_1, \dots, v_{k-1}}$ is a nonzero function.
	Noting that $e \circ \eta_{v_1, \dots, v_{k-1}}$ is a character on $\mO_K/I$, it may be written in the form $u \mapsto e \left( \varphi(v_1, \dots, v_{k-1}) \cdot u \right)$ for some $\varphi(v_1, \dots, v_{k-1}) \in I^{\perp}$.
	We have thus obtained the bound
	\begin{equation*}
		\left| \E_{x \in \mO_K/I}{e(P(x))} \right|^{2^{k-1}}
		 \le \frac{\left| \left\{ (v_1, \dots, v_{k-1}) \in K^{k-1} : \varphi(v_1, \dots, v_{k-1}) = 0 \right\} \right|}{|K|^{k-1}}.
	\end{equation*}
	
	The remainder of the proof consists of two main steps.
	First, we show that, for a convenient choice of $K$, the function $\varphi$ becomes (after a change of coordinates) a polynomial in $k-1$ variables.
	Next, we establish a bound on the number of roots of multivariable polynomials mod $I$. \\
	
	Recall $m = \floor{\log_p{d}}$ (or 0 if the characteristic of $K$ is infinite).
	Let $K = \{x^{p^m} : x \in \mO_K/I\}$.
	This is a subgroup, since the function $x \mapsto x^{p^m}$ is a homomorphism.
	For $0 \le j \le m$, let $C_j = \sum_{i=1}^n{s_i \cdot b_i S_{i,j}}$ so that
	\begin{equation*}
		\eta_{v_1, \dots, v_{k-1}}(u) = \sum_{j=0}^m{C_j(v_1, \dots, v_{k-1}) u^{p^j}},
	\end{equation*}
	and each of the polynomials $C_j(v_1, \dots, v_{k-1})$ is an additive polynomial of degree at most $p^m$ in each coordinate.
	In particular, $v_i \mid C_j(v_1, \dots, v_{k-1})$ for each $i \in \{1, \dots, k-1\}$.
	Making the substitution $v_i = w_i^{p^m}$, we therefore have
	\begin{equation*}
		\eta_{w_1^{p^m}, \dots, w_{k-1}^{p^m}}(u) = \sum_{j=0}^m{C_j \left( w_1^{p^m}, \dots, w_{k-1}^{p^m} \right) u^{p^j}}
		 = \sum_{j=0}^m{\tilde{C}_j \left( w_1, \dots, w_{k-1} \right) w_1^{p^j} \dots w_{k-1}^{p^j} u^{p^j}}
	\end{equation*}
	for some $\tilde{C}_j$.
	
	For each $j \ge 0$, the function $\chi(u) = e(u^{p^j})$ is a character on $\mO_K/I$, so there exists $z_j \in I^{\perp}$ such that $\chi(u) = e(z_j \cdot u)$.
	Hence, defining $\psi(w_1, \dots, w_{k-1}) = \varphi \left( w_1^{p^m}, \dots, w_{k-1}^{p^m} \right)$, we have
	\begin{equation*}
		\psi(w_1, \dots, w_{k-1}) = \sum_{j=0}^m{z_j \tilde{C}_j(w_1, \dots, w_{k-1}) w_1 \dots w_{k-1}}.
	\end{equation*}
	That is, $\psi$ is a polynomial of degree at most $p^{2m}$ in each coordinate.
	
	We claim that $\psi$ is not the zero polynomial.
	By definition, $z_0 = 1$.
	We also have
	\begin{equation*}
		C_0(v_1, \dots, v_{k-1}) = \sum_{i=1}^n{s_i b_i S_{i,0}(v_1, \dots, v_{k-1})}.
	\end{equation*}
	The coefficients $b_i$ are integers coprime to $p$, so $s_ib_i \ne 0$.
	Now, $S_{i,0}$ is a sum of terms of the form
	\begin{equation*}
		\prod_{l=1}^{k-1}{v_l^{p^j_l}}
	\end{equation*}
	with the property $\sum_{l=1}^{k-1}{p^{j_l}} = r_i - p$.
	Therefore, the monomials appearing in $S_{i,0}$ are distinct from the monomials appearing in $S_{i',0}$ for $i \ne i'$.
	It follows that $C_0$ is not the zero polynomial.
	Thus,
	\begin{equation*}
		z_0 \tilde{C}_0(w_1, \dots, w_{k-1})w_1 \dots w_{k-1} = C_0 \left( w_1^{p^m}, \dots, w_{k-1}^{p^m} \right)
	\end{equation*}
	is not the zero polynomial, and each monomial appearing has degree divisible by $p^m$.
	Finally, for $j \ne 0$, we have
	\begin{equation*}
		\tilde{C}_j(w_1, \dots, w_{k-1})w_1 \dots w_{k-1}
		 = \frac{C_j \left( w_1^{p^m}, \dots, w_{k-1}^{p^m} \right)}{\prod_{l=1}^{k-1}{w_l^{p^j-1}}},
	\end{equation*}
	which consists of monomials in which each variable has degree congruent to $1$ mod $p$.
	This proves that $\psi$ is not the zero polynomial. \\
	
	The final step is to show that $\psi$ has only a small number of zeros.
	
	\begin{lemma} \label{lem: nonzero polynomial}
		Let $l \in \N$, and let $T(x_1, \dots, x_l)$ be a nonzero polynomial with coefficients in $I^{\perp}$ of degree $d_i$ in the variable $x_i$ for $i = 1, \dots, l$.
		Then
		\begin{equation*}
			\left| \left\{ (x_1, \dots, x_l) \in (\mO_K/I)^l : T(x_1, \dots, x_l) = 0 \pmod{\mO_K} \right\} \right| \le \left( \sum_{i=1}^l{d_i} \right) \frac{[\mO_K:I]^l}{\lpf(I)}.
		\end{equation*}
	\end{lemma}

	\begin{proof}[Proof of Lemma]
		Let us first consider the case $l = 1$.
        Write $T(x) = \alpha_d x^d + \ldots + \alpha_1 x + \alpha_0$ with $\alpha_0, \ldots, \alpha_d \in I^{\perp}$.
        We view $T$ as a map on $\mO_K$ with $T(I) \subseteq \mO_K$.
        Then let $J$ be the ideal
        \begin{equation*}
            \left( \textup{span}_{\mO_K}\left\{ \alpha_0, \ldots, \alpha_d \right\} \right)^{\perp}.
        \end{equation*}
        In other words, $J$ is the largest ideal such that $T(J) \subseteq \mO_K$.
        By construction, $J \supseteq I$ and $J$ is a proper ideal of $\mO_K$.
        Let $\p$ be a prime ideal with $\p \mid J$.
        Letting $\tilde{T}$ be the reduction of $T$ mod $\p$, we have that $\tilde{T}$ is a nonzero polynomial.
        But $\mO_K/\p$ is a field, so $\tilde{T}$ can have at most $d$ roots mod $\p$.
		
		Now suppose $T(x) = 0 \pmod{\mO_K}$.
		Then $\tilde{T}(x) = 0$.
		Therefore,
		\begin{equation*}
			\left| \{x \in \mO_K/I : T(x) = 0 \pmod{\mO_K} \} \right| \le d \frac{[\mO_K:I]}{[\mO_K:\p]} \le d \frac{[\mO_K:I]}{\lpf(I)}.
		\end{equation*} \\
		
		Suppose $l \ge 2$.
		If $\lpf(I) \le \sum_{i=1}^l{d_i}$, then there is nothing to prove, so assume $\lpf(I) > \sum_{i=1}^l{d_i}$.
		Fix $x \in \mO_K/I$, and let $T_x(x_1, \dots, x_{l-1}) = T(x_1, \dots, x_{l-1}, x)$.
		If $T_x(x_1, \dots, x_{l-1})$ is not the zero polynomial, then by the induction hypothesis,
		\begin{equation*}
			\left| \left\{ (x_1, \dots, x_{l-1}) \in (\mO_K/I)^{l-1} : T_x(x_1, \dots, x_{l-1}) = 0 \right\} \right| \le \left( \sum_{i=1}^{l-1}{d_i} \right) \frac{[\mO_K:I]^{l-1}}{\lpf(I)}.
		\end{equation*}
		Hence,
		\begin{align*}
			& \left| \left\{ (x_1, \dots, x_l) \in (\mO_K/I)^l : T(x_1, \dots, x_l) = 0 \right\} \right| \\
			 &~= [\mO_K:I]^{l-1} \left| \{x \in \mO_K/I : T_x = 0\} \right|
			 + \left| \{(x_1, \dots, x_{l-1}, x) : T_x \ne 0, T_x(x_1, \dots, x_{l-1}) = 0\} \right| \\
			 &~\le [\mO_K:I]^{l-1} \left| \{x \in \mO_K/I : T_x = 0\} \right| + \left( \sum_{i=1}^{l-1}{d_i} \right) \frac{[\mO_K:I]^l}{\lpf(I)}.
		\end{align*}
		It therefore suffices to prove
		\begin{equation*}
			\left| \{x \in \mO_K/I : T_x = 0\} \right| \le d_l \frac{[\mO_K:I]}{\lpf(I)}.
		\end{equation*}
		
		Fix $x_1, \dots, x_{l-1} \in \mO_K/I$, and let $T^{x_1, \dots, x_{l-1}}(x) = T_x(x_1, \dots, x_{l-1}) = T(x_1, \dots, x_{l-1}, x)$.
		If $T_x = 0$, then $T^{x_1, \dots, x_{l-1}}(x) = 0$.
		Therefore, by the $l=1$ case above, it follows that
		\begin{equation*}
			\left| \{x \in \mO_K/I : T_x = 0\} \right|
			 \le \left| \{x \in \mO_K/I : T^{x_1, \dots, x_{l-1}}(x) = 0\} \right|
			 \le d_l \frac{[\mO_K:I]}{\lpf(I)},
		\end{equation*}
		unless $T^{x_1, \dots, x_{l-1}}$ is the zero polynomial.
		So, it remains to find $x_1, \dots, x_{l-1}$ such that $T^{x_1, \dots, x_{l-1}}$ is a nonzero polynomial.
		Note that the coefficients of $T^{x_1, \dots, x_{l-1}}$ are polynomial expressions in $x_1, \dots, x_{l-1}$ of degree at most $d_i$ in the variable $x_i$.
		Since $T$ is not the zero polynomial, there is at least one coefficient that is a nonzero polynomial $C(x_1, \dots, x_{l-1})$.
		By the induction hypothesis,
		\begin{equation*}
			\left| \left\{ (x_1, \dots, x_{l-1}) \in (\mO_K/I)^{l-1} : C(x_1, \dots, x_{l-1}) = 0 \right\} \right|
			 \le \left( \sum_{i=1}^{l-1}{d_i} \right) \frac{[\mO_K:I]^{l-1}}{\lpf(I)}.
		\end{equation*}
		Since $\lpf(I) > \sum_{i=1}^l{d_i} \ge \sum_{i=1}^{l-1}{d_i}$ by assumption, it follows that $C(x_1, \dots, x_{l-1}) \ne 0$ for some $(x_1, \dots, x_{l-1}) \in (\mO_K/I)^{l-1}$.
		For this choice of $x_1, \dots, x_{l-1}$, the polynomial $T^{x_1, \dots, x_{l-1}}(x)$ is not the zero polynomial, so we are done.
	\end{proof}
	
	Applying Lemma \ref{lem: nonzero polynomial} to $\psi$, we get the bound
	\begin{equation*}
		\left| \left\{ (w_1, \dots, w_{k-1}) \in (\mO_K/I)^{k-1} : \psi(w_1, \dots, w_{k-1}) = 0 \right\} \right|
		 \le p^{2m} (k-1) \frac{[\mO_K:I]^{k-1}}{\lpf(I)}.
	\end{equation*}
	Thus,
	\begin{align*}
		\left| \E_{x \in \mO_K/I}{e(P(x))} \right|^{2^{k-1}}
		 & \le \frac{\left| \left\{ (v_1, \dots, v_{k-1}) \in K^{k-1} : \varphi(v_1, \dots, v_{k-1}) = 0 \right\} \right|}{|K|^{k-1}} \\
		 & = \frac{\left| \left\{ (w_1, \dots, w_{k-1}) \in (\mO_K/I)^{k-1} : \psi(w_1, \dots, w_{k-1}) = 0 \right\} \right|}{[\mO_K:I]^{k-1}} \\
		 & \le \frac{p^{2m}(k-1)}{\lpf(I)}.
	\end{align*}
\end{proof}

We can now quickly deduce Theorem \ref{thm: quantitative TE}.

\begin{proof}[Proof of Theorem \ref{thm: quantitative TE}]
	Expand $f$ as a Fourier series $f = \sum_{\chi \in \hat{\mO_K/I}} \hat{f}(\chi) \chi$.
    Then by Parseval's identity,
    \begin{multline*}
        \norm{L^2_x(\mO_K/I)}{\frac{1}{[\mO_K:I]}\sum_{n \in \mO_K/I} f(x+P(n)) - \frac{1}{|H_{I,P}|} \sum_{n \in H_{I,P}} f(x+n)}^2 \\
         = \norm{L^2(\mO_K/I)}{\sum_{\chi \in \hat{\mO_K/I}} \hat{f}(\chi) \chi \cdot \left( \frac{1}{[\mO_K:I]}\sum_{n \in \mO_K/I} \chi(P(n)) - \ind_{H_{I,P}^{\perp}}(\chi) \right)}^2 \\
         = \sum_{\chi \in \hat{\mO_K/I}} \left| \hat{f}(\chi) \right|^2 \cdot \left| \frac{1}{[\mO_K:I]}\sum_{n \in \mO_K/I} \chi(P(n)) - \ind_{H_{I,P}^{\perp}}(\chi) \right|^2.
    \end{multline*}
    Applying Lemma \ref{lem: character bound} and Parseval's inequality again gives Theorem \ref{thm: quantitative TE}.
\end{proof}


\section{Furstenberg--S\'{a}rk\"{o}zy-type theorems in finite principal ideal rings} \label{sec:FS good rings}

In this section, we apply Theorem \ref{thm: quantitative TE} to deduce Furstenberg--S\'{a}rk\"{o}zy-type results over finite principal ideal rings (Corollary \ref{cor: intersective Sarkozy} and Theorem \ref{thm: equivalences}).
Let us start with Corollary \ref{cor: intersective Sarkozy}, which we restate for the convenience of the reader.

\Sarkozy*

\begin{proof}
    Let $c = \char(R)$.
    
    As mentioned in the discussion after Corollary \ref{cor: intersective Sarkozy}, we may prove the result under the weaker hypothesis that $P(0)$ belongs to the subgroup $H \le (R,+)$ generated by $\{P(x)-P(0) : x \in R\}$ so that $H + P(0) = H$.
    By Theorem \ref{thm: finite ring polynomial average} applied to the function $f = \ind_A$,
    \begin{equation*}
		\left| \E_{x,y \in R}{\ind_A(x) \ind_A(x + P(y))} - \E_{x \in R, z \in H}{\ind_A(x) \ind_A(x+z)} \right|
		 \le \frac{|A|}{|R|} \left( B(d,c) \frac{k-1}{\lpf(R)} \right)^{1/2^{k-1}}
	\end{equation*}
	by the Cauchy--Schwarz inequality.
	
	On the one hand, if $A$ contains no nontrivial patterns $\{x, x + P(y)\}$, then by Lemma \ref{lem: nonzero polynomial},
	\begin{equation*}
		\E_{x,y \in R}{\ind_A(x) \ind_A(x + P(y))} = \frac{|A|}{|R|} \frac{\left| \left\{ y \in R : P(y) = 0 \right\} \right|}{|R|}
		 \le \frac{|A|}{|R|} \frac{d}{\lpf(R)}.
	\end{equation*}
	On the other hand,
	\begin{equation*}
		\E_{x \in R, z \in H}{\ind_A(x) \ind_A(x+z)} \ge \left( \frac{|A|}{|R|} \right)^2
	\end{equation*}
	by Lemma \ref{lem: vdC}.
	Therefore,
	\begin{equation*}
		\left( \frac{|A|}{|R|} \right)^2 - \frac{|A|}{|R|} \frac{d}{\lpf(R)} \le C \frac{|A|}{|R|} \lpf(R)^{-1/2^{k-1}},
	\end{equation*}
	where $C = \left( B(d,c) (k-1) \right)^{1/2^{k-1}}$.
	Multiplying both sides by $\frac{|R|^2}{|A|}$, we get the desired bound
	\begin{equation*}
		|A| \le C |R| \cdot \lpf(R)^{-1/2^{k-1}} + d \frac{|R|}{\lpf(R)} \ll |R| \cdot \lpf(R)^{-1/2^{k-1}}.
	\end{equation*}

    As written, our bound depends on the characteristic $c$.
    However, we can use the characteristic-independent bound $B(d,c) \le d^2$ to eliminate dependence on the characteristic of $R$.
\end{proof}

Now we turn to the proof of Theorem \ref{thm: equivalences}, which we restate for ease of reference.

\Equivalences*

First we prove that irrational equidistribution implies condition (iii).

\begin{proposition} \label{prop: irrational equidistribution}
	Suppose $P(y)$ is good for irrational equidistribution,
	and let $H \le (\mO_K, +)$ be the group generated by $\{P(y) - P(0) : y \in \mO_K\}$.
	Then there exists $C > 0$ such that if an ideal $I$ satisfies $\lpf(I) \ge C$,
	then $H + I = \mO_K$.
	That is, (iii) holds.
\end{proposition}
\begin{proof}
    Suppose that $K$ is a number field.
    Let $d = \deg{P}$, and write
    \begin{equation*}
        P(x) = \sum_{j=0}^d a_j x^j.
    \end{equation*}
    Let $P_0(x) = P(x)$, and define inductively $P_k(x) = P_{k-1}(x+1) - P_{k-1}(x)$.
    We make two observations.
    First, if $k \ge 1$, then $P_k(\mO_K) \subseteq H$.
    Second, the polynomial $P_{d-1}$ is linear, so $P_{d-1}(\mO_K)$ is a coset of a nonzero ideal, say $P_{d-1}(\mO_K) = m+J$ with $m = P_{d-1}(0)$.
    Factor $J = \p_1 \ldots \p_l$ into prime ideals, and let $C = 1 + \max_{1 \le i \le l} [\mO_K : \p_i]$.
    If $I$ is an ideal in $\mO_K$ and $\lpf(I) \ge C$, then $I$ and $J$ are coprime, so $I + J = \mO_K$.
    But since $m+J \subseteq H$, we conclude that $H + I = \mO_K$. \\

    Now suppose $K$ is a function field.
    We prove the contrapositive.
	Suppose (iii) fails.
    Then there is a sequence of ideals $(I_n)_{n \in \N}$ in $\mO_K$ such that $\lpf(I_n) \to \infty$ and $H + I_n \ne \mO_K$ for $n \in \N$.
	Equivalently,
	\begin{align*}
		H^{\perp} \cap I_n^{\perp} = (H + I_n)^{\perp} \ne \{0\}.
	\end{align*}
    Refining the sequence $(I_n)_{n \in \N}$ if necessary, we may assume that the ideals $(I_n)_{n \in \N}$ are pairwise coprime.
    For each $n \in \N$, let $\chi_n \in H^{\perp} \cap I_n^{\perp}$ with $\chi_n \ne 0$.
    Note that $I_n^{\perp} \cap I_m^{\perp} = (I_n + I_m)^{\perp} = \{0\}$ by the coprimeness assumption, so $(\chi_n)_{n \in \N}$ is a sequence of distinct elements.
    Hence, $H^{\perp}$ is infinite.
	Every infinite compact group is uncountable, and there are only countable many rational points, so $H^{\perp}$ must contain an irrational element.
	That is, for some irrational $\chi$, we have $\chi(P(y)) = \chi(P(0))$ for every $y \in \mO_K$, so $P(y)$ is not good for irrational equidistribution.
\end{proof}

We will now prove the equivalences in Theorem \ref{thm: equivalences}
by showing the implications illustrated in the following diagram:

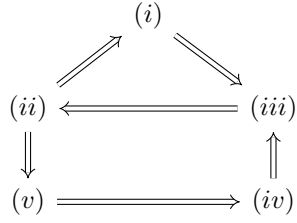
\begin{center}
	\begin{tikzcd}
         & (i) \arrow[rd, Rightarrow] & \\
		(ii) \arrow[ru, Rightarrow] \arrow[d, Rightarrow] & & (iii) \arrow[ll, Rightarrow] \\
		(v) \arrow[rr, Rightarrow] & & (iv) \arrow[u, Rightarrow]
	\end{tikzcd}
\end{center}

Condition (ii) is a quantitative refinement of condition (i), so we immediately have the implication (ii)$\implies$(i).
By Theorem \ref{thm: finite ring polynomial average}, we have the additional implications (i)$\implies$(iii)$\implies$(ii).

Condition (v) follows from (ii) by a straightforward application of the Cauchy--Schwarz inequality.

\begin{proposition} \label{prop: v,iv}
	(v)$\implies$(iv).
\end{proposition}
\begin{proof}
	Let $\delta > 0$.
	Let $C_1$, $C_2$, and $\gamma$ be as in (v).
	Suppose $\lpf(R) \ge C_1$.
	Let $A, B \subseteq R$ with $|A| |B| \ge \delta |R|^2$.
	Then by (v),
	\begin{align*}
		\left| \left\{ (x,y) \in R \times R : x \in A, x + P(y) \in B \right\} \right|
		 & \ge |A| |B| - C_2 |A|^{1/2} |B|^{1/2} |R| \cdot \lpf(R)^{-\gamma} \\
		 & = |A|^{1/2} |B|^{1/2} \left( |A|^{1/2} |B|^{1/2} - C_2 |R| \cdot \lpf(R)^{-\gamma} \right) \\
		 & \ge \delta^{1/2} \left( \delta^{1/2} - C_2 \cdot \lpf(R)^{-\gamma} \right) |R|^2.
	\end{align*}
	Thus, if
	\begin{align*}
		\lpf(R) \ge N(\delta) = \max{\left\{C_1, \left( \frac{C_2}{\delta^{1/2}} \right)^{1/\gamma} \right\}},
	\end{align*}
	then we can find $x, y \in R$ with $x \in A$ and $x + P(y) \in B$.
\end{proof}

We now prove the final implication to complete the proof of Theorem \ref{thm: equivalences}:

\begin{proposition} \label{prop: iv,iii}
	(iv)$\implies$(iii).
\end{proposition}
\begin{proof}
	We prove the contrapositive.
	Suppose (iii) fails.
	Then there is a sequence of rings $(R_n)_{n \in \N}$ with $\lpf(R_n) \to \infty$ such that the group $H_n$ generated by $\left\{ P(y) - P(0) : y \in R_n \right\}$ is a proper subgroup of $R_n$ for every $n \in \N$.
    Let $k_n$ be the index $[R_n : H_n]$.
	Let $A_n$ be a union of $\floor{\frac{k_n}{2}}$ cosets of $H_n$,
	and let $B_n = R_n \setminus \left( A_n + P(0) \right)$.
	For any $x \in A_n$ and $y \in R_n$, we have
	\begin{align*}
		x + P(y) = x + P(0) + (P(y) - P(0)) \in A_n + P(0) + H_n = A_n + P(0).
	\end{align*}
	That is, if $x \in A_n$ and $y \in R_n$, then $x + P(y) \notin B_n$.
	Moreover, since $k_n \ge 2$, we have
	\begin{align*}
		\frac{|A_n|}{|R_n|} = \frac{\floor{\frac{k_n}{2}}}{k_n} \ge \frac{1}{3} \qquad \text{and} \qquad
		 \frac{|B_n|}{|R_n|} \ge \frac{1}{2}.
	\end{align*}
	Therefore, property (iv) fails for $\delta = \frac{1}{6}$.
\end{proof}


\section{Quasirandomness of Paley-type graphs} \label{sec:quasirandomness}

The goal of this section is establish quasirandomness of graphs generated by polynomials over asymptotically totally ergodic sequences of finite principal ideal rings, as described in Theorem \ref{thm: quasirandom iff aTE}.
We will in fact prove a quantitative strengthening of Theorem \ref{thm: quasirandom iff aTE}, and for this we need to quantify quasirandomness.

\begin{definition}
    Let $G = (V, E)$ be a finite $r$-regular graph.
    We say that $G$ is \emph{$\eps$-uniform} for some $\eps > 0$ if for every $A, B \subseteq V$,
    \begin{equation*}
        \left| \left| \{(a,b) \in A \times B : \{a,b\} \in E\} \right| - \frac{r}{|V|} |A| |B| \right| \leq \eps r |V|.
    \end{equation*}
\end{definition}

We can now produce a quantitative relationship between quasirandomness and asymptotic total ergodicity, strengthening Theorem \ref{thm: quasirandom iff aTE}.
Recall that for $d, c \in \N$, $B(d,c) = c^{2\floor{\log_c{d}}}$ if $c$ is a prime number and $B(d,c) = 1$ otherwise.

\begin{theorem} \label{thm: quasirandom}
    Let $R$ be a finite principal ideal ring, and let $d \ge 2$ with $\textup{char}(R) \nmid d$.
	Let $k$ be the derivational degree of the monomial $x^d$ over the ring $R$.
	Let $\eps > 0$.
	\begin{enumerate}[(1)]
		\item	If $\frac{|R^{\times}|}{|R|} > 1 - \left( \frac{\eps}{3d \cdot B(d,\char(R)) (k-1)} \right)^{2^{k-1}}$, then $\bfP(R,d)$ is $\eps$-uniform.
		\item	If $\bfP(R,d)$ is $\eps$-uniform, then for each prime ideal $\p \le R$, either $\{\pm x^d : x \in R/\p\} = R/\p$ or $[R:\p] \ge \frac{d-1}{64d^2\eps^2}$.
	\end{enumerate}
\end{theorem}

\begin{remark}
    The bound $\frac{|R^{\times}|}{|R|} > 1 - \left( \frac{\eps}{3d \cdot B(d,\char(R)) (k-1)} \right)^{2^{k-1}}$ we obtain as a sufficient condition for $\eps$-uniformity can be improved when $R$ is a finite field (in which case $|R^{\times}| = |R| - 1$).
    This is because, in the finite field context, the work of Weil (see \cite[Theorem 3.2]{kowalski}) provides better exponential sum estimates with the exponent on $\eps$ no longer depending on the (derivational) degree of $x^d$; see \cite[Section 1.6]{ab_Sarkozy} for the use of the Weil bound in establishing quasirandomness of Paley-type graphs.
    However, Weil's bound crucially depends on algebraic properties of fields and cannot be extended to the generality of rings covered by Theorem \ref{thm: quasirandom}.
\end{remark}

We will prove the two parts of Theorem \ref{thm: quasirandom} separately.
We begin with the easier of the two, corresponding to item (1).

\begin{proposition}
    Let $R$ be a finite principal ideal ring, and let $d \ge 2$ with $\textup{char}(R) \nmid d$.
	Let $k$ be the derivational degree of the monomial $x^d$ over the ring $R$, and assume $k \ge 2$.
	Let $\eps > 0$.
	If $\frac{|R^{\times}|}{|R|} > 1 - \left( \frac{\eps}{3d \cdot B(d,\char(R)) (k-1)} \right)^{2^{k-1}}$, then $\bfP(R,d)$ is $\eps$-uniform.
\end{proposition}

\begin{proof}
    Let $\delta = \left( \frac{\eps}{3d \cdot B(d,\char(R)) (k-1)} \right)^{2^{k-1}}$.
    Suppose $|R^{\times}| \ge (1 - \delta)|R|$.
    Note that $|R^{\times}| = |R| \cdot \prod_{i=1}^r \left( 1 - \frac{1}{[R:\p_i]} \right)$, where $\p_1, \ldots, \p_r$ are the prime ideals of $R$.
    Hence, $\lpf(R) = \min_{1 \le i \le r} [R:\p_i] \ge \delta^{-1}$.

    Let $E = \{\{x,y\} \in R : y - x = \pm z^d~\text{for some}~z \in R\}$ be the set of edges in the graph $\bfP(R,d)$.
    The map $z \mapsto z^d$ is a homomorphism of the multiplicative group $R^{\times}$.
    Let $K = \{z \in R : z^d = 1\}$ be the kernel.
    Then for each $y \in R$, we have $\{z \in R : z^d = y^d\} = yK$.

    We can now count the number of edges between arbitrary subsets of $R$.
    Let $A, B \subseteq R$.
    Then
    \begin{multline*}
        N(A,B) = |\{(a,b) \in A \times B : \{a,b\} \in E\}|
        = \sum_{a \in A} \sum_{w \in \{\pm z^d : z \in R\}} \ind_B(a + w) \\
        = c \sum_{x \in R} \ind_A(x) \left( \sum_{y \in R} \frac{1}{|yK|}\ind_B(x+y^d) + \sum_{z \in R} \frac{1}{|zK|} \ind_B(x - z^d) \right) \\
         = \underbrace{c \sum_{x \in R} \ind_A(x) \left( \sum_{y \in R} \frac{1}{|K|} \ind_B(x+y^d) + \sum_{z \in R} \frac{1}{|K|} \ind_B(x-z^d) \right)}_{\mathcal{M}(A,B)} \\
         + \underbrace{c \sum_{x \in R} \ind_A(x) \left( \sum_{y \notin R^{\times}} \left( \frac{1}{|yK|} - \frac{1}{|K|} \right) \ind_B(x+y^d) + \sum_{z \notin R^{\times}} \left( \frac{1}{|zK|} - \frac{1}{|K|} \right) \ind_B(x-z^d) \right)}_{\mathcal{E}(A,B)},
    \end{multline*}
    where $c = 1$ if $-1$ is not a $d$th power in $R$ and $\frac{1}{2}$ if $-1$ is a $d$th power in $R$.
    
    To control the main term $\mathcal{M}(A,B)$, we use Theorem \ref{thm: finite ring polynomial average} with the estimate $\lpf(R) \ge \delta^{-1}$.
    Since $\char(R) \nmid d$, we have
    \begin{equation*}
        \norm{L^2_x(R)}{\E_{y \in R} f(x+y^d) - \E_{z \in R} f(z)} \le \left( B(d,\char(R)) (k-1)\delta \right)^{1/2^{k-1}} \norm{L^2(R)}{f} \le \frac{\eps}{2d} \norm{L^2(R)}{f}.
    \end{equation*}
    for every $f : R \to \C$ by Theorem \ref{thm: finite ring polynomial average}.
    Taking $f = \ind_B$ and apply the Cauchy--Schwarz inequality,
    \begin{multline*}
        \left| \sum_{x,y \in R} \ind_A(x) \ind_B(x+y^d) - |A||B| \right| = |R|^2 \cdot \left| \innprod{\ind_A}{\E_{y \in R}\ind_B(\cdot+y^d) - \frac{|B|}{|R|}}_{L^2(R)} \right| \\
         \le |R|^2 \frac{\eps}{2d} \norm{L^2(R)}{\ind_A} \norm{L^2(R)}{\ind_B} \le \frac{\eps}{3d} |R|^2.
    \end{multline*}
    A similar inequality holds for $-z^d$ in place of $+y^d$.
    
    Now we control the error term $\mathcal{E}(A,B)$ using the assumption $|R^{\times}| \ge (1-\delta)|R|$.
    Bounding $\ind_A$ and $\ind_B$ from above by 1 bounding $|xK|$ from below by 1 for $x \in R$, we have
    \begin{equation*}
        \mathcal{E}(A,B) \le 2 c |R| |R \setminus R^{\times}| \le 2c\delta |R|^2 \le \frac{\eps}{3d} |R|^2.
    \end{equation*}
    Thus,
    \begin{equation*}
        \left| N(A,B) - \frac{2c}{|K|} |A||B| \right| \le \frac{2}{3} \eps \frac{|R|}{d} |R|.
    \end{equation*}
    The graph $\mathbf{P}(R,d)$ is $r$-regular for some $r \ge \frac{1}{d} |R^{\times}| \ge \frac{1-\delta}{d} |R| \ge \frac{2|R|}{3d}$, so $\bfP(R,d)$ is $\eps$-uniform.
\end{proof}

To prove item (2) in Theorem \ref{thm: quasirandom}, we use a spectral characterization of quasirandomness in conjunction with bounds on exponential sums over finite fields.
Let us first state a consequence of quasirandomness for the eigenvalues of the adjaceny matrix of a Cayley graph.

\begin{proposition}[{\cite[Theorem 1.4]{cz}}] \label{prop: quasirandom eigenvalues}
    Let $G = (V,E)$ be an $r$-regular Cayley graph of a group of order $n$.
    Let $\lambda_1, \lambda_2, \ldots, \lambda_n$ be the eigenvalues of the adjacency matrix of $G$, ordered so that $|\lambda_1| \ge |\lambda_2| \ge \ldots \ge |\lambda_n|$.
    Let $\eps > 0$.
    If $G$ is $\eps$-uniform, then $|\lambda_2| \le 8\eps r$.
\end{proposition}

The eigenvalues of Cayley graphs (such as $\bfP(R,d)$) can be computed more readily than for arbitrary graphs.

\begin{lemma} \label{lem: eigenvalues of Cayley graph}
    Let $\Gamma$ be a finite abelian group.
    Let $S \subseteq \Gamma$ be a symmetric set not containing the identity element.
    If $G = (V,E)$ is the Cayley graph of $\Gamma$ generated by $S$ (i.e., $V = \Gamma$ and $\{x,y\} \in E$ if and only if $x-y \in S$), then the eigenvalues of the adjacency matrix of $G$ are
    \begin{equation*}
        \sum_{s \in S} \chi(s)
    \end{equation*}
    for $\chi \in \hat{\Gamma}$.
\end{lemma}

\begin{proof}
    This result is well-known, but we include a short proof for completeness.
    Let $A$ be the adjacency matrix of $G$.
    We check directly that the characters $\chi \in \hat{\Gamma}$, viewed as vectors in $\C^{\Gamma}$, are eigenvectors of $A$ with the desired eigenvalues:
    \begin{equation*}
        (A\chi)(x) = \sum_{y \in \Gamma} \ind_S(y-x) \chi(y) = \sum_{s \in S} \chi(x+s) = \left( \sum_{s \in S} \chi(s) \right) \cdot \chi(x).
    \end{equation*}
\end{proof}

We are now set to prove item (2) in Theorem \ref{thm: quasirandom}.

\begin{proposition}
    Let $R$ be a finite principal ideal ring, and let $d \ge 2$ with $\textup{char}(R) \nmid d$.
	Let $k$ be the derivational degree of the monomial $x^d$ over the ring $R$.
	Let $\eps > 0$.
	If $\bfP(R,d)$ is $\eps$-uniform and $\p \le R$ is a prime ideal, then either $\{\pm x^d : x \in R/\p\} = R/\p$ or $[R:\p] \ge \frac{d-1}{64d^2\eps^2}$.
\end{proposition}

\begin{proof}
    We will prove the contrapositive.
    Suppose there exists a prime ideal $\p \le R$ such that $[R:\p] < \frac{d-1}{64d^2\eps^2}$ and $\{\pm x^d : x \in R/\p\} \ne R/\p$.
    Let $q = [R:\p]$ so that $R/\p \cong \F_q$.
    The natural quotient map $R \to R/\p$ induces a surjective homomorphism $\pi : R \to \F_q$.
    Let $S = \{x^d, -x^d : x \in R\} \setminus \{0\}$ be the generating set of $\bfP(R,d)$ as a Cayley graph.
    Given an additive character $\chi : \F_q \to S^1$, we have
    \begin{equation*}
        \left| \sum_{s \in S} (\chi \circ \pi)(s) \right| = \frac{|R|}{q} \cdot \left| \sum_{t \in \pi(S)} \chi(t) \right|,
    \end{equation*}
    and $\pi(S) \cup \{0\} = \{x^d, -x^d : x \in \F_q\}$.
    (The set $\pi(S)$ may or may not contain $0$, depending on the interaction between $d$ and the other ideals of $R$.)
    We now use some basic properties of the Fourier transform to get a lower bound on $\left| \sum_{t \in \pi(S)} \chi(t) \right|$ for some choice of $\chi$.

    On the one hand,
    \begin{equation*}
        \norm{L^2(\F_q)}{\ind_{\pi(S)}}^2 = \frac{|\pi(S)|}{q}.
    \end{equation*}
    On the other hand, by Parseval's identity,
    \begin{equation*}
        \norm{L^2(\F_q)}{\ind_{\pi(S)}}^2 = \sum_{\chi \in \hat{\F_q}} \left| \frac{1}{q} \sum_{t \in \pi(S)} \chi(t) \right|^2 = \left( \frac{|\pi(S)|}{q} \right)^2 + \frac{1}{q^2} \sum_{\chi \ne 1} \left| \sum_{t \in \pi(S)} \chi(t) \right|^2.
    \end{equation*}
    Therefore,
    \begin{equation*}
        \frac{1}{q-1} \sum_{\chi \ne 1} \frac{q-1}{q^2} \left| \sum_{t \in \pi(S)} \chi(t) \right|^2 = \frac{|\pi(S)|}{q} - \left( \frac{|\pi(S)|}{q} \right)^2,
    \end{equation*}
    so there exists a nontrivial character $\chi$ such that
    \begin{equation*}
        \left| \sum_{t \in \pi(S)} \chi(t) \right|^2 \ge \frac{q^2}{q-1} \left( \frac{|\pi(S)|}{q} - \left( \frac{|\pi(S)|}{q} \right)^2 \right) \ge q \left( \frac{|\pi(S)|}{q} - \left( \frac{|\pi(S)|}{q} \right)^2 \right).
    \end{equation*}
    Now, by assumption, $\pi(S) \cup \{0\} \ne \F_q$.
    But $\pi(S) \setminus \{0\}$ is a multiplicative subgroup of $\F_q^{\times}$, so $|\pi(S)| - 1$ is a proper divisor of $q-1$.
    In particular, $|\pi(S)| \le \frac{q-1}{2} + 1 = \frac{q+1}{2}$.
    Moreover, the polynomial $x^d-1$ has at most $d$ roots in $\F_q$, so $|\pi(S)| - 1 \ge \frac{q-1}{d}$.
    Thus, for some nontrivial character $\chi$,
    \begin{equation*}
        \left| \sum_{t \in \pi(S)} \chi(t) \right|^2 \ge q \cdot \min_{x \in \left[ \frac{1}{d} + \frac{d-1}{dq}, \frac{1}{2} + \frac{1}{2q} \right]} x(1-x) \ge q \cdot \frac{d-1}{d^2}.
    \end{equation*}
    
    Thus, by Lemma \ref{lem: eigenvalues of Cayley graph}, the adjacency matrix of $\bfP(R,d)$ has an eigenvalue of size at least $\frac{|R|}{\sqrt{q}} \cdot \frac{\sqrt{d-1}}{d} > |R| \cdot \frac{8d\eps}{\sqrt{d-1}} \cdot \frac{\sqrt{d-1}}{d} = 8\eps|R|$.
    By Proposition \ref{prop: quasirandom eigenvalues}, it follows that $\bfP(R,d)$ is not $\eps$-uniform.
\end{proof}


\section{Open problems} \label{sec: open problems}

Our work offers a new point of view on the relationship in a ring between its algebraic properties, dynamical properties of its polynomial actions, combinatorial configurations found in its large subsets, and quasirandomness of graphs generated by polynomial values.
There are several aspects of the interrelations between algebra, dynamics, and combinatorics that are ripe for further exploration.
In this short section, we propose a few open problems related to the main results of this paper.


\subsection{Necessary and sufficient conditions for the polynomial Szemer\'{e}di theorem over rings of integers}

Recall that Theorem \ref{thm: global field Sarkozy} characterized the family of polynomials satisfying the the Furstenberg--S\'{a}rk\"{o}zy theorem (over the ring of integers of a global field) as the class of polynomials having a root modulo every nontrivial ideal.
We conjecture that a similar characterization is possible for a version of the polynomial Szemer\'{e}di theorem over rings of integers of global fields:

\begin{conjecture} \label{conj: nec, suff pSz}
    Let $K$ be a global field with ring of integers $\mO_K$.
    Suppose $P_1[x], \dots, P_k[x] \in \mO_K[x]$ are nonconstant polynomials.
    Then the following are equivalent:
    \begin{enumerate}[(i)]
        \item for any nonzero ideal $I \subseteq \mO_K$, there exists $n \in \mO_K$ such that $\{P_1(n), \dots, P_k(n)\} \subseteq I$.
        \item for any subset $A \subseteq \mO_K$ with $d^*(A) > 0$, there exist $x, n \in \mO_K$ such that
        \begin{equation*}
            \{x, x+P_1(n), \dots, x+P_k(n)\} \subseteq A.
        \end{equation*}
    \end{enumerate}
\end{conjecture}

We note that the special case when $\mO_K = \Z$ was established in \cite{bll} and later generalized to number fields in \cite{br}.
Hence, the only open case of Conjecture \ref{conj: nec, suff pSz} is when $K$ is a global function field.


\subsection{Irrational equidistribution and an asymmetric Furstenberg--S\'{a}rk\"{o}zy theorem over quotient rings} \label{sec: A,B equidistribution}

In Theorem \ref{thm: equivalences}, we proved an equivalence between several conditions on a polynomial $P$ over a finite principal ideal ring $R$ and showed that each condition holds if the polynomial is the reduction of a polynomial that is good for irrational equidistribution over the ring of integers of a global field.
(The property of being \emph{good for irrational equidistribution} is defined in Definition \ref{defn: irrational equidistribution}.)
Our next question asks if this implication can be reversed:

\begin{question} \label{quest: A,B equidistribution}
    Let $K$ be a global function field\footnote{We restrict our attention to function fields because every nonconstant polynomial over a number field is good for irrational equidistribution; see Proposition \ref{prop: number field irrational equidistribution}.} with ring of integers $\mO_K$, and let $P(x) \in \mO_K[x]$.
    Suppose that for every $\delta > 0$, there exists $N \in \N$ such that if $I \le \mO_K$ is an ideal with $\lpf(I) \ge N$, then for any $A, B \subseteq \mO_K/I$ with $|A|\cdot |B| \ge \delta [\mO_K:I]^2$, there exist $x,y \in \mO_K/I$ such that $x \in A$ and $x + P(y) \in B$.
    Must $P$ be good for irrational equidistribution?
\end{question}

Suppose we are in the special case that $P(x) \in \mO_K[x]$ has coefficients in $\F_p$ (here, $p = \char(K)$).
Then Question \ref{quest: A,B equidistribution} is essentially answered in \cite[Theorem 1.17]{ab_Sarkozy}.
We fill in the small additional details below.

\begin{proposition} \label{prop: A,B equidistribution F_p}
    Let $K$ be a global function field of characteristic $p$.
    Let $P(x) \in \F_p[x]$.
    Then the following are equivalent:
    \begin{enumerate}[(i)]
        \item $P$ is good for irrational equidistribution;
        \item for every $\delta > 0$, there exists $N \in \N$ such that if $I \le K$ is an ideal with $\lpf(I) \ge N$ and $A, B \subseteq \mO_K/I$ with $|A|\cdot|B| \ge \delta [\mO_K:I]^2$, then there exist $x,y \in \mO_K/I$ such that $x \in A$ and $x+P(y) \in B$.
    \end{enumerate}
\end{proposition}

\begin{proof}
    The implication (i)$\implies$(ii) follows from Theorem \ref{thm: equivalences}.

    Suppose (ii) holds.
    Applying condition (ii) only with prime ideals and using \cite[Theorem 1.17]{ab_Sarkozy}, we can write $P(x) = Q(x^{p^s})$ for some $s \ge 0$ and a polynomial $Q(x) \in \F_p[x]$ that is good for irrational equidistribution.
    Suppose for contradiction that $s \ne 0$.
    Then $P(x) - P(0) = Q(x^{p^s}) - Q(0)$ belongs to the group of $p$th powers for every $x$.
    But for arbitrary $\p$ (in particular, with $[\mO_K : \p]$ arbitrarily large), the subgroup $H_{\p} = \{x^p : x \in \mO_K/\p^2\}$ is a proper subgroup of $\mO_K/\p^2$.
    Arguing as in the proof of Proposition \ref{prop: iv,iii}, we can construct sets $A$ and $B$ as unions of cosets of $H_{\p}$ to establish the failure of (ii) for $\delta = \frac{1}{6}$.
    This contradicsts our assumption, so we must have $s = 0$.
    But then $P(x) = Q(x)$ is good for irrational equidistribution as desired.
\end{proof}

The main input in this argument, \cite[Theorem 1.17]{ab_Sarkozy}, relies on several algebraic lemmas that are heavily dependent on the restriction that $P$ has coefficients in $\F_p$.
It seems plausible that Question \ref{quest: A,B equidistribution} has an affirmative answer in general, but the methods used to handle the special case in Proposition \ref{prop: A,B equidistribution F_p} are presently inadequate for these purposes.


\subsection{Quasirandomness of Paley-type graphs over modular rings}

Theorem \ref{thm: quasirandom iff aTE} gives a sufficient condition for quasirandomness of a sequence of Paley-type graphs and also a necessary condition, but there is a gap between these two conditions.
Consequently, there are families of Paley-type graphs for which we do not know whether or not they are quasirandom.
To formulate a problem in as concrete terms as possible, we stick to families of modular rings.

\begin{problem}
    Let $d \in \N$.
    Give necessary and sufficient conditions on an increasing sequence $(N_n)_{n \in \N}$ of natural numbers so that the sequence of graphs $(\bfP(\Z/N_n\Z,d))_{n \in \N}$ is quasirandom.
\end{problem}


\end{document}